\documentclass[11pt]{amsart}
\usepackage{latexsym}
\usepackage{amssymb}
\usepackage{amsmath}
\usepackage{amsfonts}

\usepackage[dvipsnames]{xcolor}

\def\e{\textbf{e}}

\def\R{\mathbb{R}}

\newcommand{\bmat}{\left[\begin{matrix}}
\newcommand{\emat}{\end{matrix}\right]}
\usepackage[latin1]{inputenc} 
\usepackage{amsthm}
\usepackage{amsxtra}
\usepackage{amscd}
\usepackage{setspace}   
\usepackage{mathrsfs}   
\usepackage{multirow}
\usepackage{graphicx}
\usepackage{enumerate}
\usepackage{xcolor}
\usepackage{url}

\newtheorem{theorem}{Theorem}
\newtheorem{proposition}[theorem]{Proposition}
\newtheorem{lemma}[theorem]{Lemma}

\numberwithin{theorem}{section}
\numberwithin{equation}{section}

\theoremstyle{remark}
\newtheorem*{remark}{Remark}

\theoremstyle{definition}

\newcommand{\Z}{\mathbb{Z}}

\newcommand{\C}{\mathbb{C}}

\newcommand{\re}{\mathrm{Re}\,}
\newcommand{\im}{\mathrm{Im}\,}

\newcommand{\rank}{\mathrm{rank\,}}

\newcommand{\spn}{\mathrm{span}}

\newcommand{\GL}{\mathrm{GL}}
\newcommand{\SL}{\mathrm{SL}}
\newcommand{\rk}{\mathrm{rk\,}}

\newcommand{\ga}{\mathfrak{a}}
\newcommand{\Lg}{\mathfrak{g}}

\newcommand{\vol}{\mathrm{vol}}
\newcommand{\diag}{\mathrm{diag}}

\newcommand{\Id}{\mathrm{Id}}

\newcommand{\ve}{\mathbf{e}}

\newcommand{\G}{\Gamma}
\newcommand{\ignore}[1]{{}}

    \title[A truncated inner product formula]{A truncated inner product formula in the geometry of numbers}
    \author{Seokho Jin, Seungki Kim}

\begin{document}
\maketitle

\ignore{
\begin{abstract}
We provide an asymptotic inner product formula for truncations of the pseudo-Eisenstein series on $\SL(n,\Z)\backslash\SL(n,\R)$ of the form
\begin{align*}
E_{f}(g) = \sum_{L \subseteq \Z^n} f(\det Lg)
\end{align*}
for $f \in C^\infty_0(\R_{\geq 0})$, where the sum ranges over all primitive rank $k < n$ sublattices of $\Z^n$. We consider both the Arthur and the ``harsh'' truncations, with uses in the geometry of numbers in mind. As an application, we present an improved estimate on the number of the
primitive sublattices of a fixed lattice of bounded determinant, or equivalently, the number of the rational points on a Grassmannian variety of bounded height.
\end{abstract}

\section*{수정 제안 1 (핵심 내용을 강조한 버전)}

\begin{abstract}
We study the statistical distribution of primitive sublattices in the space of lattices $\SL(n,\Z)\backslash\SL(n,\R)$. A central difficulty in this area is that the second moment of the counting function for rank $k$ sublattices, where $2 \le k \le n-2$, diverges. To overcome this, we analyze the inner product of truncated pseudo-Eisenstein series of the form $E_{f}(g) = \sum_{L} f(\det Lg)$, where the sum is over primitive rank $k$ sublattices of $\Z^n$. We establish an asymptotic formula for this inner product for both the standard Arthur truncation and a "harsh" truncation that vanishes outside a compact set. Our analysis relies on several technical results of independent interest, including a proof of the uniform moderate growth (UMG) property for these pseudo-Eisenstein series and a new method for resolving singularities in the Maass-Selberg relations. As a primary application, we obtain a significant improvement on the discrepancy bound for the number of primitive sublattices. For almost every lattice, we improve the error term in counting rank $k$ sublattices with determinant up to $p$ to $O(p^{n-1/7+\epsilon})$, surpassing classical bounds for $\min(k, n-k) \ge 8$.
\end{abstract}
}

\begin{abstract}
We explore the theory of truncations of automorphic forms to study the fine-scale distribution of rational points on Grassmannians. Previous approaches based on variance estimates were hindered by the divergence of the $L^2$-norm of the counting function for sublattices of intermediate rank. We resolve this issue by developing an asymptotic formula for the truncated inner product of pseudo-Eisenstein series $$E_{f}(g) = \sum_{L \subseteq \Z^n \atop \rank k,\ \mathrm{prim.}} f(\det Lg)$$
on $\SL(n,\Z)\backslash\SL(n,\R)$. Our formula applies to both the standard Arthur truncation and the ``harsh'' truncation, which is more directly applicable to the geometry of numbers. Our analysis relies on several technical results of independent interest, including a proof of the uniform moderate growth property for pseudo-Eisenstein series, a comparison of the Arthur and harsh truncations, and a method for resolving singularities in the Maass-Selberg relations. As an application, we significantly sharpen the known error terms for the number of primitive sublattices with determinant up to $p$ of a generic lattice, improving the bound from $O(p^{n-1/\min(k,n-k)})$ to $O(p^{n-1/7+\epsilon})$ when $\min(k, n-k) \ge 8$.
\end{abstract}

\section{Introduction}


\subsection{Automorphic methods in the geometry of numbers}

It may be said that the automorphic approach to the geometry of numbers started with the following celebrated theorem of Siegel \cite{Sie45}. Let $G=\SL(n,\R), \Gamma=\SL(n,\Z)$, and $f:\R^n\rightarrow\R$ be a measurable function. Then
\begin{align} \label{eq:siegel}
\frac{1}{\vol(\Gamma \backslash G)}\int_{\Gamma \backslash G} \sum_{v \in \Z^n \atop v \neq 0} f(vg) dg = \int_{\R^n}f(x)dx.
\end{align} 
Siegel \cite{Sie45} also stated without proof that, for $1 \leq k < n$ and measurable $f:(\R^n)^k\rightarrow\R$,
\begin{align*}
\frac{1}{\vol(\Gamma \backslash G)}\int_{\Gamma \backslash G} \sum_{v_1,\ldots,v_k \in \Z^n \atop \mbox{\tiny lin. indep.}} f(v_1g,\ldots,v_kg) dg = \int_{(\R^n)^k}f(x_1,\ldots,x_k)dx_1\ldots dx_k,
\end{align*}
which was later shown by Rogers \cite{Rog55} and Schmidt \cite{Sch57}. These formulas, and their numerous extensions (e.g., \cite{Vee98, AGH24, GKY22, Kim24}), have been among the most powerful tools in the geometry of numbers to date, as evidenced by the records on sphere packing \cite{Rog47,Ven13, Kla25} and covering \cite{Rog59} by lattices. Moreover, nowadays they have become one of the standard methods in homogeneous dynamics as well, playing roles in a number of important works, such as \cite{EMM98,KM99,AM09,MS10,AM15}.

Recently, the field has seen several more fruitful applications of automorphic methods. For example, the second-named author and Venkatesh \cite{KV18} used the full force of the spectral theory of the $\GL(n)$ Eisenstein series to show that (on the Riemann Hypothesis) all but arbitrarily small percentage of lattices have about the same number of Siegel-reduced bases, a result that helps shape one's understanding of lattice reduction algorithms such as LLL \cite{LLL82}. Viazovska \cite{Via17} and Cohn-Kumar-Miller-Radchenko-Viazovska \cite{CKMRV}, with brilliant uses of the theory of modular forms, found the ``miracle functions'' whose existences were conjectured by Cohn-Elkies \cite{CE03}, proving that the $E_8$ and the Leech lattices provide the best sphere packings in dimensions 8 and 24, respectively. Gargava-Viazovska \cite{GV24} recently gave a new lower bound on the lattice packing densities in infinitely many dimensions, replacing the record of Venkatesh \cite{Ven13}, by adopting part of the techniques of Einsiedler-Lindenstrauss-Michel-Venkatesh \cite{ELMV} as well as the new subconvexity bound due to Petrow-Young \cite{PY23}.

\subsection{This paper}

The purpose of the present paper is to explore yet another aspect of the theory of automorphic forms for its uses in the geometry of numbers: the Arthur truncation and the inner product formula for the truncated Eisenstein series, known as the Maass-Selberg relations. Specifically, our goal is to develop the methods that facilitate their applications, as well as to demonstrate their power, in the field of lattice-point counting.

The very idea of applying the Arthur truncation and the related machinery to the geometry of numbers originates from S.D. Miller, who, in his lecture notes \cite{Mil20}, demonstrated an alternative proof of the Siegel integral formula \eqref{eq:siegel} via the Maass-Selberg relations. The Ph.D. thesis of Thurman \cite{Thu24} generalizes Miller's argument to evaluate the integrals of the form
\begin{align*}
\frac{1}{\vol(\Gamma\backslash G)}\int_{\Gamma\backslash G} \sum_{L_1 \subseteq_{p}\ldots\subseteq_{p}L_k \subseteq_p \Z^n \atop \rk L_j=i_j} f(\det L_1g,\ldots,\det L_k g) dg
\end{align*}
for given $1 \leq i_1 < \ldots < i_k <n$; here, $L\subseteq_p M$ means that $L$ is a primitive sublattice of $M$. For $G=\SL(n,\R)$ and $\Gamma=\SL(n,\Z)$, this was previously announced by \cite[Theorem 5]{Kim22}; but Thurman's formula also covers the case in which $G$ is any Chevalley group. Both \cite{Mil20} and \cite{Thu24} work with the truncated inner product of a ``lattice-point counting function'' --- such as $\sum_{v \in \Z^n} f(vg)$ --- against the constant function $1$. The present paper, on the other hand, is concerned with the product of two nontrivial lattice-point counting functions. More precisely, we are interested in the integrals of the form
\begin{align}\label{eq:intro_goal}
\frac{1}{\vol^C(\Gamma\backslash G)} \int_{\Gamma\backslash G} \Lambda^C\sum_{L_1 \subseteq _p\Z^n \atop \rk L_1=k_1} f_1(\det L_1g) \sum_{L_2 \subseteq_p\Z^n \atop \rk L_2=k_2} f_2(\det L_2g)dg,
\end{align}
where $\Lambda^C$ is the truncation operator (see Sec. \ref{sec:trunc} below), $\vol^C(\Gamma\backslash G)=\int_{\G\backslash G}\Lambda^C1dg$ is the volume of the truncated fundamental domain, $1 \leq k_1,k_2 \leq n$, and $f_1,f_2:\R_{\geq 0} \rightarrow \R$ are smooth functions of compact support.

Truncated versions of the mean value formulas over lattices, such as \eqref{eq:siegel}, are naturally of interest, since they allow us to dismiss from consideration those lattices sitting near the cusp of $\Gamma\backslash G$, which are expected to behave badly in terms of packing and covering. There also exists a more concrete motivation, arising from the fact \cite[Corollary to Theorem 4]{Kim22} that, for $2 \leq k \leq n-2$ and any measurable $f:\R_{\geq 0}\rightarrow\R$ of nonzero mass,
\begin{align*}
\int_{\Gamma\backslash G} \left(\sum_{L \subseteq_p \Z^n \atop \rk L=k} f(\det Lg)\right)^2dg = \infty.
\end{align*}
As explained in \cite{Kim22}, this stymies the attempt to study the metric properties of the rank $k$ sublattices of a random lattice, or equivalently the rational points of the Grassmannian $\mathrm{Gr}(n,k)$ with respect to a random twisted height (cf. \cite{Thu93}), via the standard methods that have been immensely successful in the case $k=1$. Truncation provides an appealing method for circumventing this problem, since it forces the above integral to become finite.

Our main result is the following formula for \eqref{eq:intro_goal}. Here we give a rough version of the actual statement Theorem \ref{thm:res_conc}, in order to quickly present the idea without laying out the technical details.
\begin{theorem}[Theorem \ref{thm:res_conc}, roughly stated] \label{thm:intro_main}
Let $f_1,f_2:\R_{\geq 0} \rightarrow \R$ be smooth functions of compact support.
The truncated inner product \eqref{eq:intro_goal} is equal to
\begin{align*}
Mf_1(n)Mf_2(n)\xi(n,k_1)\xi(n,k_2) + O_{n,f_1,f_2}(T^\kappa),
\end{align*}
where $Mf$ denotes the Mellin transform of $f$,
\begin{align*}
\xi(n,k)=\frac{\xi(2)\cdots\xi(n-k)}{\xi(k+1)\cdots\xi(n)},
\end{align*}
$\kappa < n(n^2-1)/6$ is an explicit constant, and $T\gg 0$ is a parameter that determines the truncated region.\footnote{The higher $T$ is, the less we truncate; see Sec. \ref{sec:trunc}, \ref{sec:maass} for details.}
\end{theorem}
The main term above is precisely what one would expect, based on the existing literature on the problem of counting sublattices and the rational points of a flag variety, e.g., \cite{Sch68, Thu93, Kim22, Kim23}. For instance, the notation $a(n,d)H^n$ in those works is precisely $Mf(n)\xi(n,k)$ when $f=\mathbf 1_{[0,H]}$, the characteristic function of $[0,H]$.

Of course, for Theorem \ref{thm:intro_main} to be useful, the error term must be manageable. It is somewhat inconvenient to state it precisely in the introduction, but for certain important applications, it takes the form
\begin{align}\label{eq:intro_error}
O_n\left(Mf_1(n)Mf_2(n-\eta)T^{-\frac{\eta k_2(n-k_2)}{2}}\right) + O_n\left(Mf_1(n-\eta')Mf_2(n-\eta)T^\kappa\right)
\end{align}
for $0 < \eta'<\eta < 1$. If $f_1=f_2=\mathbf 1_{[0,H]}$, say, then $Mf_1(s)=Mf_2(s)=H^s/s$, so the error is indeed less than the main term for all sufficiently large $H$.

We also provide a version of the main theorem (Theorem \ref{thm:wanted}) in which the truncation is ``harsh,'' at the cost of an additional error term $O(T^{-p})$ for any choice of $p > 0$. The ``harsh truncation'' of an automorphic form causes it to vanish in the truncated region, which is a desirable trait from the perspective of the geometry of numbers, since it enables one to study the statistical properties of lattices contained in a compact subset of $\G\backslash G$. In fact, we prove the following.
\begin{theorem}[Theorem \ref{thm:truncate}, roughly stated] \label{thm:intro_truncate}
Let $\phi$ be a function on $\G\backslash G$ of uniform moderate growth. Then its harsh and Arthur truncations differ by at most $O(T^{-p})$ pointwise in the truncated region parametrized by $T>0$. They coincide on the rest of $\G\backslash G$.
\end{theorem}
We also show (Proposition \ref{prop:umg}) that functions of the form $\sum_{L \subseteq_p \Z^n} f(\det Lg)$ are indeed of uniform moderate growth; so the above theorem is applicable to those.

Finally, as a demonstration of the power of the automorphic machinery in the geometry of numbers, we showcase an application of our results that extends the famous lattice-point counting discrepancy bound of Schmidt \cite{Sch60} to the problem of counting sublattices.
\begin{theorem}[Theorem \ref{thm:grass} restated]\label{thm:intro_grass}
Fix $\epsilon > 0$ small. Let $h_p:\R_{\geq0}\rightarrow\R$ be the characteristic function of $[0,p]$. Then for almost every $g \in \Gamma\backslash G$ and $p \gg_g 0$,
\begin{align*}
\sum_{L \subseteq_p \Z^n\atop\rk L=k} h_p(\det Lg) - Mh_p(n)\xi(n,k) = O_{n,\epsilon,g}(p^{n-\frac{1}{7}+\epsilon}).
\end{align*}
\end{theorem}
For $\min(k,n-k)\geq8$, this error term is substantially lower than the currently known estimates (e.g. \cite{Sch68, Thu93, Kim23})
\begin{align*}
\sum_{L \subseteq_p \Z^n\atop\rk L=k} h_p(\det Lg) - Mh_p(n)\xi(n,k) = O_{n,g}\left(p^{n-\max(\frac{1}{k},\frac{1}{n-k})}\right),
\end{align*}
as well as the conjecture of Franke-Manin-Tschinkel \cite[Corollary to Theorem 5]{FMT}, which is somewhat better than the above but still diminishes as $\min(k,n-k)$ grows.

Our work leaves much room for further developments. In particular, it seems that our techniques can be extended to generalized flag varieties, which we hope to return to in near future. It is also likely that certain parts of our analyses could be improved for quantitatively tighter results. Looking for applications in which truncation plays an essential role --- such as in Theorem \ref{thm:intro_grass} above --- would of course be especially interesting. To that end, it may be worth trying to strengthen our formula so that it takes into account the ``shapes'' of sublattices as well.

\subsection{A few words about the proof}

Let us briefly discuss the main technical challenges faced and resolved in the course of our proof of Theorem \ref{thm:intro_main}.

\subsubsection*{UMG property of pseudo-Eisenstein series}

It is well-known (see e.g., \cite[I.2]{MW95}) that an automorphic form, in the strict sense of the word (see \cite[I.2.17]{MW95}), is of uniform moderate growth (UMG), from which it can be shown that its Arthur truncation is of rapid decay. In particular, this applies to every Eisenstein series. However, a pseudo-Eisenstein series such as $\sum_{L \subseteq_p \Z^n} f(\det Lg)$ is not an automorphic form in the narrow sense; rather, it is a direct integral of a certain family of Eisenstein series. To show that it also has the UMG property, one needs to \emph{uniformly} bound the rate of growth of all Eisenstein series in the family. Curiously, it appears that this natural problem has not been considered previously, at least in writing. We supply the argument for what we need below (Propositions \ref{prop:uni_gro'} and \ref{prop:umg}), provided $f$ is smooth.


\subsubsection*{Harsh versus Arthur truncations}

As discussed earlier, from the geometry of numbers perspective, it is expedient for the truncations of lattice-point counting functions to have a property strictly stronger than rapid decay, namely that it is arbitrarily small everywhere in the truncated region. To achieve this, we refine the argument of Arthur \cite[Proposition 13.2]{Art05} by further exploiting his theory of truncation as well as the ``UMG trick'' of Harish-Chandra \cite{Har76}. As a result, we obtain Theorem \ref{thm:intro_truncate}, and the harsh-truncated version of Theorem \ref{thm:intro_main}.

\subsubsection*{Singularities in the Maass-Selberg relation}

The Maass-Selberg relation for the Borel Eisenstein series
\begin{align*}
\int_{\Gamma \backslash G} \Lambda^CE_B(\lambda_1,g) &\Lambda^CE_B(\lambda_2,g) dg \\
&= \sum_{w_1,w_2 \in W} \frac{e^{\langle C, w_1\lambda_1+w_2\lambda_2\rangle}}{\prod_{\alpha \in \Delta}\langle \alpha^\vee, w_1\lambda_1+w_2\lambda_2\rangle}M(w_1,\lambda_1)M(w_2,\lambda_2),
\end{align*}
(see Section 2 for the notations) is highly singular in general, e.g., when $\lambda_1=\lambda_2$, as can also be seen in the simpler formula for $\SL(2,\R)$.  It is expected that most singularities on the right cancel with one another, and the genuine singularities, if any, only arise from the intertwining operator $M(w,\lambda)$. However, to the best of our knowledge, no concrete computations have been carried out in this regard, except perhaps for groups of low rank. We are able to resolve this problem for the cases needed by our proof of Theorem \ref{thm:intro_main}, by first perturbing one of the $\lambda_i$'s to separate the singularities, and then resolving each of them via a combinatorial argument. It seems plausible that our method may extend to any pair of Borel Eisenstein series.

\subsection{Organization}

Section 2 lays out the background knowledge and the notations used throughout the paper. In Section 3, we show that the pseudo-Eisenstein series of our interest are of uniform moderate growth, and study the decay properties of the Arthur truncation, proving Theorem \ref{thm:intro_truncate}. In Section 4, we prove Theorem \ref{thm:L1-norm}, a truncated integral formula for a pseudo-Eisenstein series. It serves as a slight refinement of Miller \cite{Mil20} and parts of Thurman \cite{Thu24}, as well as a gentle introduction to the more complicated proof of Theorem \ref{thm:intro_main}, which is handled in Section 5. The singularity removal is delegated to Section 6. Finally, Section 7 proves Theorem \ref{thm:intro_grass}.

\subsection{Acknowledgments}

The series of works \cite{Kim22,Kim23} leading to the present paper really started when Phong Q. Nguyen asked the second-named author a question back in 2019 concerning the distribution of the smallest rank $k$ sublattice of a random lattice, with applications to lattice-based cryptography in mind --- despite the progress made, it is yet to be answered in full. In the technical aspect, we are heavily indebted to Stephen D. Miller; we are especially grateful to him for sharing his idea and an early version of the notes \cite{Mil20}, as well as lengthy discussions on the uniform moderate growth and rapid decay properties, including the reference to \cite{GMP17}. Sihun Jo helped us greatly with the saddle point computation in the appendix. We also thank Yeansu Kim, Min Lee, and Forrest Thurman for helpful discussions and suggestions.

Both authors acknowledge support by the NRF grant funded by the Korea government MSIT No. RS-2023-00253814. In addition, SK was also supported by NSF CNS-2034176.

\section{Preliminaries}

We recall the basic facts and set the notations to be used throughout the paper. In order to make the paper as accessible as possible, we will try to define all the notions specifically and explicitly for $\SL(n)$, and suppress the generalities whenever possible, even if it somewhat misrepresents the elegance of the general theory.

Garrett \cite{Gar18} is a good reference for most of the background knowledge needed for our paper. The notes by Miller \cite{Mil20} provides an accessible account for the Maass-Selberg formula and its application to the geometry of numbers.
More detailed accounts can be found in, for instance, Arthur \cite{Art05}, Lapid and Ouellette \cite{LO12}, and Moeglin and Waldspurger \cite{MW95}.

\subsection{Basic definitions}

We write $G = \SL(n,\R)$ and $\Gamma = \SL(n,\Z)$ throughout this paper. $G$ and $\Gamma \backslash G$ are endowed the right $G$-invariant measure $dg$, normalized so that
\begin{equation*}
\mathrm{vol}(\Gamma \backslash G) = \int_{\Gamma \backslash G} dg = \xi(2) \cdots \xi(n),
\end{equation*}
where $\xi(s) = \pi^{-s/2}\Gamma(s/2)\zeta(s)$.

Let $\Lg$ be the Lie algebra of $G$, and $U(\Lg)$ be its universal enveloping algebra. For $X \in \Lg$ an appropriate function $f$ on $G$, write
\begin{equation*}
Xf(g) = \frac{d}{dt}\bigg\rvert_{t=0} f(ge^{tX}).
\end{equation*}

Let $B=P_0$ be the \emph{standard minimal parabolic} or the \emph{Borel subgroup} of $G$, that is, the group consisting of the upper-triangular matrices in $G$. $P_0$ is characterized by the property that it is the stabilizer in $G$ of the flag $0 = E_0 \subseteq E_1 \subseteq E_2 \subseteq \ldots \subseteq E_n = \R^n$ under the natural action of $G$ on $\R^n$ (considered as row vectors) by right multiplication, where $E_k = \spn(\ve_{n-k+1},\ldots,\ve_n)$ and $\ve_i$ is the vector whose $i$-th entry is $1$ and the rest are zero.

A \emph{standard parabolic subgroup} (or a \emph{parabolic} for short) $P$ is by definition the stabilizer in $G$ of a flag of the form $0 \subseteq E_{i_1} \subseteq \ldots \subseteq E_{i_k} \subseteq \R^n$, $0 = i_0 < i_1 < \ldots < i_k < n$; we write $k = r(P)$. Each such $P=P(i_1,\dots, i_k)$ admits the \emph{Levi decomposition} $P = M_PN_P$, where
\begin{equation*}
M_{P} =
\begin{pmatrix}
\GL(n-i_k,\R) &  &  \\
  & \ddots &  \\
  &  & \GL(i_1-i_0,\R) \end{pmatrix} \cap G,\ 
N_{P} = \begin{pmatrix}
\Id_{n-i_k} & * & * \\
  & \ddots & * \\
  & & \Id_{i_1-i_0} \end{pmatrix}.
\end{equation*}

\subsection{The spaces $\ga_{P_1}^{P_2}$ and $(\ga_{P_1}^{P_2})^*$} \label{sec:spaces}

Let
\begin{equation*}
\ga := \ga_0 := \Big\{(a_1,\ldots,a_n) \in \R^n : \sum_i a_i = 0\Big\}.
\end{equation*}
This may be naturally identified with the standard choice of a Cartan subalgebra of $\Lg$.
Let $A \subseteq G$ be the subgroup consisting of the diagonal matrices with positive real entries. Then there is a map $\exp: \ga \rightarrow A$ defined by
\begin{equation*}
\exp(a) = \diag(e^{a_1}, \ldots, e^{a_n}).
\end{equation*}
For every \(g \in G\), we can factor \(g\) in the so-called ``NAK decomposition'' as
\[
  g \;=\; nak.
\] We write \(a(g)\) to denote the \(A\)-part of \(g\) in this decomposition, and define \(H(g) \in \ga\) to be the unique element satisfying
\[
  \exp\bigl(H(g)\bigr) \;=\; a(g).
\]

We next consider the dual space $\ga^* = \ga_0^*$ of $\ga$, i.e., the space of linear functionals $\ga \rightarrow \R$. It can be naturally identified with the quotient
\begin{equation*}
\R^n / \spn\{(1,\ldots,1)\}.
\end{equation*}
For $g \in G$ and $\lambda \in \ga^*$, we write $g^\lambda := \exp(\lambda(H(g)))$. We also write $\ga^*_\C = \ga^* \otimes \C$ for short.

More generally, for a parabolic $P=P(i_1,\dots, i_k)$, we define
\begin{align*}
\ga_P = \{(a_1,\ldots,a_n) \in \ga : a_1 = \ldots = a_{n-i_k}, a_{n-i_k+1} = \ldots = a_{n-i_{k-1}}, \ldots, \\
\ldots, a_{n-i_1+1} = \ldots = a_{n}\},
\end{align*}
and write $(\ga_P)^* =\ga_P^* $ for the dual of $\ga_P$. In particular, for $P = P_0$, $\ga_{P_0} = \ga_0$ and $\ga_{P_0}^* = \ga_0^*$.

For $g \in G$, the Levi decomposition with respect to $P$ admits an expression $g=mna_Pk$, with $m \in M_P$ such that the determinant of each block of $m$ is $\pm 1$, $n \in N_P$, $a_P \in \exp(\mathfrak a_P)$, and $k \in \mathrm{SO}(n,\R)$. Moreover, the choice of $a_P$ is unique. Therefore, we may define $H_P(g) \in \mathfrak a_P$ to be the element so that $\exp(H_P(g)) = a_P(g)$.

For parabolics $P_1(i_1,\dots, i_k) \subseteq P_2(j_1,\dots, j_\ell)$, i.e., $\{j_1,\dots, j_\ell\}\subseteq \{i_1,\dots, i_k\}$, there exist natural injections
\begin{equation*}
\ga_{P_2} \hookrightarrow \ga_{P_1},\ \ \ga_{P_2}^* \hookrightarrow \ga_{P_1}^*, 
\end{equation*}
and the corresponding dual surjections
\begin{equation*}
\ga_{P_1}^* \twoheadrightarrow \ga_{P_2}^*,\ \ \ga_{P_1} \twoheadrightarrow \ga_{P_2}.
\end{equation*}
These maps together yield split exact sequences, by which we define $\ga_{P_1}^{P_2}$:
\begin{align*}
0 \rightarrow \ga_{P_1}^{P_2} \rightarrow &\ga_{P_1} \rightleftarrows \ga_{P_2} \rightarrow 0, \\
0 \rightarrow \ga_{P_2}^* \rightleftarrows &\ga_{P_1}^* \rightarrow (\ga_{P_1}^{P_2})^* \rightarrow 0.
\end{align*}

 Note that $\ga_{P}^{G}=\ga_{P}$. In concrete terms, $\ga_{P_1}^{P_2}$ can be identified with the set of $a\in\ga_{P_1}$ such that for each block in $P_2$ that is the union of some blocks in $P_1$, the sum of those block-values is $0$.
From the split exact sequences, we have
\begin{equation*}
\ga_{P_1} = \ga_{P_1}^{P_2}\oplus\ga_{P_2}  , \ga_{P_1}^* = (\ga_{P_1}^{P_2})^*\oplus \ga_{P_2}^* ,
\end{equation*}
and in general, for $Q_1 \subseteq P_1 \subseteq P_2 \subseteq Q_2$,
\begin{equation} \label{eq:a_decomp}
\ga_{Q_1}^{Q_2} = \ga_{Q_1}^{P_1} \oplus \ga_{P_1}^{P_2} \oplus \ga_{P_2}^{Q_2},
\end{equation}
and likewise for $(\ga_{Q_1}^{Q_2})^*$.
 These decompositions are orthogonal under the standard inner product; this follows from the fact that one summand is  constant on each larger block, whereas the other summand has those same blocks summing to zero, so their dot product must vanish.

Throughout this paper, we adopt the following notational conventions. For $x \in \ga_{Q_1}^{Q_2}$, we write $x_{P_1}^{P_2}$ for the component of $x$ in $\ga_{P_1}^{P_2}$ with respect to the direct sum \eqref{eq:a_decomp}. If in addition $F$ is a function on $\ga_{P_1}^{P_2}$, we often write $F(x)$ to mean $F(x_{P_1}^{P_2})$.
When $P_1=P_0$ is the minimal parabolic, we either omit the subscripts or place $0$ instead, for example, $\ga_{P_0}^P=\ga^P$ and when $P_2=G$, we omit the superscript. For example, $\ga = \ga_{0}^G = \ga_{P_0}^G$. We extend these conventions to the dual spaces $(\ga_{P_1}^{P_2})^*$ as well.

In addition, there exists the \emph{Weyl group} $W$, which in our situation is isomorphic to the symmetric group on $n$ elements, that acts on $\ga$ by permuting its entries. More precisely, for $w \in W$ and $a \in \ga$, $(wa)_i = a_{w(i)}$; note that this is a right group action. $W$ also acts on $\ga^*$ as well: 
for $\lambda \in \ga^*$, $(w\lambda)_{w(i)} = \lambda_{i}$; this is a left group action. Observe that $w\lambda(a) = \lambda(wa)$. The notion of the Weyl group generalizes to $\ga_{P_1}^{P_2}$ and so on, but it is not needed for the present paper.

\subsection{Roots, coroots, and weights}

$\ga^*$ possesses a few distinguished elements that play an important role in Lie theory. A \emph{root} is an element of the form $\alpha_{ij} \in \ga^*$, $1 \leq i \neq j \leq n$, defined by $\alpha_{ij}(a) = a_i-a_j$ for $a \in \ga$. $\alpha_{ij}$ is said to be \emph{positive} if $i < j$, and \emph{negative} if $i > j$. The \emph{Weyl element} $\rho=\frac{1}{2}\sum_{\alpha_{ij}>0}\alpha_{ij} =\frac{1}{2}(n-1, n-3,\ldots, -n+1)\equiv (n,n-1,\ldots,1)$ is defined to be half the sum of all positive roots.


For $1 \leq i \leq n-1$, a root of the form $\alpha_{i,i+1}$ is said to be \emph{simple}, which we write $\alpha_i$ for short. We denote by $\Delta = \{\alpha_1, \ldots, \alpha_{n-1}\}$ the set of all simple roots; $\Delta$ forms a basis of $\ga^*$. We endow $\ga^*$ with an inner product defined by $\langle \lambda, \alpha_{ij} \rangle = \lambda_i - \lambda_j$. Note that the action of $W$ is unitary with respect to this inner product.

We define the \emph{coroot} $\alpha_i^\vee \in \ga$ of $\alpha_i \in \ga^*$ to be the $n$-vector with $1$ in its $i$-th entry, $-1$ in the $(i+1)$-st, and $0$ elsewhere. We write $\Delta^\vee = \{\alpha_1^\vee, \ldots, \alpha_{n-1}^\vee\}$; $\Delta^\vee$ forms a basis of $\ga$. Let $\widehat\Delta = \{\varpi_1,\ldots,\varpi_{n-1}\} \subseteq \ga^*$ be the dual basis of $\Delta^\vee$; the elements $\varpi_i (= \varpi_{\alpha_i})$ are in fact $\sum_{k=1}^i\e_k$, and are referred to as \emph{fundamental weights}. 
Analogously, let $\widehat\Delta^\vee = \{\varpi_1^\vee,\ldots,\varpi_{n-1}^\vee\} \subseteq \ga$ be the dual basis of $\Delta$; here $\varpi_i^\vee=\sum_{k=1}^i\e_k-\frac{i}{n}\sum_{k=1}^n\e_k$. This is the set of the \emph{fundamental coweights}.

Fix a standard parabolic $P$ stabilizing a flag $0 \subsetneq E_{i_1} \subsetneq \ldots \subsetneq E_{i_k} \subsetneq \R^n$. We define
\begin{align*}
\Delta^P = \{\alpha_i : i \neq i_j \mbox{ for any $j=1,\ldots,k$}\},
\end{align*}
a basis of $(\ga^P)^*$, and also define
\begin{align*}
\Delta_P &= \{\alpha \Big|_{\ga_{P}} : \alpha \in \Delta - \Delta^P\}, \\
\widehat\Delta_P &= \{\varpi_\alpha \Big|_{\ga_{P}} : \alpha \in \Delta - \Delta^P\},
\end{align*}
which are bases of $(\ga_{P})^*$. 
These have dual bases
\begin{align*}
\widehat\Delta_P^\vee &= \{\varpi_\alpha^\vee : \alpha \in \Delta_P\}, \\
\Delta_P^\vee &= \{\alpha^\vee : \alpha \in \Delta_P\}
\end{align*}
of $\ga_P$, respectively.

In general, for parabolics $P_1 \subseteq P_2$, we define
\begin{align*}
\Delta_{P_1}^{P_2} &= \{\alpha \Big|_{\ga_{P_1}^{P_2}} : \alpha \in \Delta^{P_2} - \Delta^{P_1}\}, \\
\widehat\Delta_{P_1}^{P_2} &= \{\varpi \Big|_{\ga_{P_1}^{P_2}} : \varpi \in \widehat\Delta_{P_1} - \widehat\Delta_{P_2}\},
\end{align*}
both bases of $(\ga_{P_1}^{P_2})^*$, and we define $(\widehat\Delta_{P_1}^{P_2})^\vee, (\Delta_{P_1}^{P_2})^\vee \subseteq \ga_{P_1}^{P_2}$ to be their dual bases, respectively.


\subsection{The pseudo-Eisenstein series $E_f$}

From now on, fix $k$ and the corresponding maximal parabolic $P = P(k)$ stabilizing the flag $0 \subseteq E_k \subseteq \R^n$. The right action of $G$ on $\R^n$ extend naturally to $\wedge^k\R^n$. Let $\ve= \ve_{n-k+1} \wedge \ldots \wedge \ve_n \in \wedge^k\R^n$.

To a function $f: \R_{\geq 0} \rightarrow \R$ for which the Mellin transform exists and the Mellin inversion formula applies,
we associate a pseudo-Eisenstein series
\begin{equation}\label{eq:pseisen1}
E_{P,f}(g) = \sum_{\gamma \in \Gamma \cap P \backslash \Gamma} f(\|\ve\gamma g\|).
\end{equation}
If $P$ is clear from the context, we write $E_{P,f} = E_f$.
Using the Mellin inversion formula
\begin{equation*}
f(x) = \frac{1}{2\pi i} \int_{\re s = c} Mf(s)x^{-s}ds
\end{equation*}
for any appropriate $c$, we can rewrite \eqref{eq:pseisen1} as
\begin{equation} \label{eq:pseisen2}
E_f(g) = \frac{1}{2\pi i} \int_{\re s = c} Mf(s) E_P(s,g) ds,
\end{equation}
where
\begin{equation} \label{eq:E_p}
E_P(s,g) = \sum_{\gamma \in \Gamma \cap P \backslash \Gamma} \|\ve\gamma g\|^{-s}.
\end{equation}
By e.g., \cite[Lemma 4.1]{Kim23}, $E_P(s,g)$ converges absolutely whenever $\re s > n$. $E_P(s,g)$ in fact admits meromorphic continuation to $\C$, by Proposition \ref{prop:epeb} below.

\ignore{
Of particular interest is the case is where we take $f$ to be
\begin{equation*}
f_H(x) = \begin{cases} 1 & \mbox{if $x \leq H$} \\ 0 & \mbox{otherwise}\end{cases}
\end{equation*}
for some $H \in \R^+$. In this case, we have
\begin{equation*}
Mf_H(s) = \frac{H^s}{s}
\end{equation*}
for $\re s > 0$.
}

\subsection{Borel Eisenstein series}

For $\lambda = (\lambda_1, \ldots, \lambda_n) \in \ga^*_\C$ in the \emph{Godement range} of absolute convergence i.e. satisfying $\alpha^\vee_i(\re \lambda) > 1$ for all $1 \leq i < n$, the Borel Eisenstein series $E_B(\lambda,g)$ is defined to be
\begin{equation*}
E_B(\lambda, g) = \sum_{\gamma \in \Gamma \cap B \backslash \Gamma} a(\gamma g)^{\lambda + \rho}.
\end{equation*}
The right-hand side converges absolutely in the Godement range.
The theory of Eisenstein series shows that $E_B$ admits meromorphic continuation to $\ga_\C$, with finitely many hyperplane singularities. In particular, it can be shown by the Langlands constant term formula that $E_B(-\rho, g) = 1$.

\begin{proposition}\label{prop:epeb}
For $\re s > n$, the equality $E_P(s,g) = E_B(\lambda,g)$ holds with the choice
\begin{equation} \label{eq:lambda}
\lambda = \lambda(s) \equiv (-n, -n+1, \ldots, -k-1; -s-k, -s-k+1, \ldots, -s-1).
\end{equation}
\end{proposition}
\begin{proof}

For $\lambda\equiv(\lambda_1,\ldots,\lambda_n)$ in the Godement range, we have
\[ E_B(\lambda, g) = \sum_{\gamma \in \G \cap P \backslash \G} \sum_{\gamma' \in \G \cap B \backslash \G \cap P} a(\gamma'\gamma g)^{\lambda+\rho}.\]
Write $a^{(k)} = a_1 \cdots a_k, a^{(n-k)} = a_{k+1} \cdots a_n$. Let $\lambda^{(k)} \equiv (\lambda^{(k)}_1,\ldots,\lambda^{(k)}_k,0,\ldots,0)$ and $\rho^{(k)} \equiv (k,k-1,\ldots,1,0,\ldots,0)$. Similarly, let $\lambda^{(-k)} \equiv (0,\ldots,0,\lambda^{(-k)}_{k+1},\ldots,\lambda^{(-k)}_n)$ and $\rho^{(-k)} \equiv (0,\ldots,0,n-k,\ldots,1)$.
Then we can rewrite
\begin{equation} \label{eq:ac}
E_B(\lambda,g) = \sum_{\gamma \in \G \cap P \backslash \G} \sum_{\gamma' \in \G \cap B \backslash \G \cap P} a^{(k)}(\gamma'\gamma g)^{\lambda^{(k)} + \rho^{(k)}} a^{(n-k)}(\gamma' \gamma g)^{\lambda^{(-k)} + \rho^{(-k)}} a(\gamma' \gamma g)^{\aleph}
\end{equation}
for $\aleph = (\lambda + \rho) - (\lambda^{(k)} + \rho^{(k)}) - (\lambda^{(-k)} + \rho^{(-k)})$.

Choose $\lambda, \lambda^{(k)}, \lambda^{(-k)}$ so that the following conditions are satisfied: (i) $\lambda$, $\lambda^{(k)}$ and $\lambda^{(-k)}$ are in the Godement range of the absolute convergence of $E_B$, $E_{B}^{(k)}$ and $E_B^{(n-k)}$, the Borel Eisenstein series in $\SL(n)$, $\SL(k)$ and $\SL(n-k)$, respectively (ii) $\aleph \equiv (0, \ldots, 0, -s, \ldots, -s)$. It is easy to see that such choices can indeed be made for $\re s > n$. Then the inner sum over $\gamma'$ in \eqref{eq:ac} is equal to
\begin{equation*}
a(\gamma g)^\aleph E_{B}^{(k)}(\lambda_k, \gamma g) E_{B}^{(n-k)}(\lambda_{n-k}, \gamma g),
\end{equation*}
since $a(\gamma' \gamma g)^\aleph = a(\gamma g)^\aleph$. Invoking the meromorphic continuation of the Eisenstein series, and setting $\lambda^{(k)} = -\rho_k$ and $\lambda^{(-k)} = -\rho_{n-k}$, this is equal to $a(\gamma g)^\aleph = \|\ve\gamma g\|^{-s}$, and thus \eqref{eq:ac} coincides with $E_P(s,g)$.
\end{proof}

\subsection{Truncation} \label{sec:trunc}

For a locally bounded function $\phi$ on $\Gamma \backslash G$ and a standard parabolic $P$, we define $\phi_P$, the \emph{constant term} of $\phi$ along $P$ by
\begin{equation*}
\phi_P(g) = \int_{\Gamma \cap {N_P} \backslash N_P} \phi(ng)dn
\end{equation*}
where $dn$ is the Haar measure on $N_P$ with unit volume on $\Gamma \cap {N_P} \backslash N_P$. Then for sufficiently regular $C \in \ga$, i.e. $\alpha_i(C) \gg 0$ for all $1 \leq i < n$, the \emph{Arthur truncation} $\Lambda^C\phi$ of $\phi$ is given by
\begin{equation} \label{eq:arthur_t}
\Lambda^C\phi(g) = \sum_{P \supseteq P_0} (-1)^{r(P)} \sum_{\delta \in \Gamma \cap P \backslash \Gamma}\phi_P(\delta g)\widehat\tau_P(H_P(\delta g) - C),
\end{equation}
where $\widehat\tau_P: \ga_P \rightarrow \C$ is the characteristic function of the obtuse Weyl chamber
\begin{equation} \label{eq:obt_Weyl}
\{t \in \ga_P: \varpi(t) \geq 0 \mbox{ for all } \varpi \in \widehat\Delta_P\}.
\end{equation}
$\Lambda^C$ is a self-adjoint and idempotent operator on the space of functions on $\Gamma \backslash G$. 

Our later argument necessitates a deeper understanding of the Arthur truncation. Let us recall the parts of the theory that are relevant to us.

Following Arthur \cite[Sec. 8]{Art05}, for a standard parabolic $P$ and elements $C_1^P, C^P$ of $\ga_0^P$, and a compact subset $\omega$ of $N_B$, define
\begin{align*}
\mathfrak S^P(C_1^P) &= \{g = nak: n \in \omega, a \in \exp(\ga_0), k \in \mathrm{O}(n,\R), \\
&\hspace{7mm}\alpha(H_0^{P}(g) - C_1^P) > 0\ \forall \alpha \in \Delta_0^P\}, \\
\mathfrak S^P(C_1^P,C^P) &= \{g \in \mathfrak S^P(C_1^P): \varpi(H_0^{P}(g) - C^P) < 0 \ \forall \varpi \in \widehat\Delta_0^P\}.
\end{align*}
By reduction theory, we can choose $\omega$ and $C_1$ so that $\mathfrak S^G(C_1)$ contains a fundamental domain for $\Gamma \backslash G$; let us fix these once and for all throughout the paper. In the discussion below, we take $C_1^P = (C_1)_0^P$ and $C^P = (C)_0^P$; we often omit these subscripts, as per the convention we adopted in Section \ref{sec:spaces}.

Let $F^P(g,C)$ be the characteristic function of the projection of $\mathfrak S^P(C_1,C)$ onto $\Gamma \cap P \backslash G$. Also for $H \in \ga_{P_1}$ let
\begin{equation} \label{eq:sigma}
\sigma_{P_1}^{P_2}(H) = \sum_{Q:Q \supseteq P_2} (-1)^{r(Q)-r(P_2)}\tau_{P_1}^Q(H)\widehat\tau_Q(H),
\end{equation}
where $Q$ runs over the standard parabolics, and $\tau_{P_1}^Q$ is defined as in \eqref{eq:obt_Weyl} by replacing $\ga_P$ by $\ga_{P_1}^Q$ and $\varpi \in \widehat\Delta_P$ by $\alpha \in \Delta_{P_1}^Q$. It turns out that $\sigma$ is a characteristic function of a set in $\ga_{P_1}$; more precisely, by the M\"obius inversion formula, it holds that for $P \supseteq P_1$
\begin{equation*}
\tau_{P_1}^P(H)\widehat\tau_P(H) = \sum_{P_2 : P_2 \supseteq P} \sigma_{P_1}^{P_2}(H)
\end{equation*}
(see \cite[(8.2)]{Art05}). In particular, for each $P_1$, the supports of $\sigma_{P_1}^{P_2}$, for $P_2 \supseteq P_1$, form a partition of the obtuse Weyl chamber \eqref{eq:obt_Weyl} in $\ga_{P_1}$. 

Then, for locally bounded $\phi:\Gamma \backslash G \rightarrow \C$, its truncation $\Lambda^C\phi(g)$ can be written as
\begin{equation*}
\sum_{P_1 \subseteq P_2} \sum_{\delta \in \Gamma \cap P_1 \backslash \Gamma} F^{P_1}(\delta x, C)\sigma_{P_1}^{P_2}(H_{P_1}(\delta x) - C)\phi_{P_1,P_2}(\delta x),
\end{equation*}
where $P_1 \subseteq P_2$ run over standard parabolic subgroups, and
\begin{equation*}
\phi_{P_1,P_2}(y) = \sum_{\{P: P_1 \subseteq P \subseteq P_2\}} (-1)^{r(P)} \phi_P(y).
\end{equation*}
(See Chs. 8 and 13 of \cite{Art05} for proof, especially pages 37-43.) Let us write $\chi_{P_1,P_2}^C(x) = F^{P_1}(x,C)\sigma_{P_1}^{P_2}(H_{P_1}(x)-C)$, so that
\begin{equation}\label{eq:arthur_t_defn2}
\Lambda^C\phi(x) = \sum_{P_1 \subseteq P_2} \sum_{\delta \in \Gamma \cap P_1 \backslash \Gamma} \chi_{P_1,P_2}^C(\delta x)\phi_{P_1,P_2}(\delta x).
\end{equation}
In particular, $\chi_{G,G}^C$ is the characteristic function of the ``untruncated region,'' on which $\Lambda^C \phi(x) = \phi(x)$.

This reformulation of the Arthur truncation can in particular be used to prove that, if $\phi$ is of uniform moderate growth, then $\Lambda^C\phi$ is of rapid decay --- see e.g., \cite[Proposition 13.2]{Art05} for a proof.

Let us also define the \emph{harsh truncation}
\begin{equation*}
\mathrm H^C\phi(g) = \begin{cases} \phi(g) & g \in \mathrm{supp}\,\chi^C_{G,G} \\ 0 & g \not\in \mathrm{supp}\,\chi^C_{G,G}. \end{cases}
\end{equation*}
Intuitively, this may be the more useful form of truncation for applications to the geometry of numbers. The major drawback, however, is that it lacks a tractable inner product formula, unlike the Arthur truncation.

\subsection{Maass-Selberg relations} \label{sec:maass}

For Borel Eisenstein series $E_B(\lambda_1,g)$ and $E_B(\lambda_2,g)$, the Maass-Selberg relation is given by
\begin{align*}
\int_{\Gamma \backslash G} \Lambda^CE_B(\lambda_1,g) &\Lambda^CE_B(\lambda_2,g) dg \\
&= \sum_{w_1,w_2 \in W} \frac{e^{\langle C, w_1\lambda_1+w_2\lambda_2\rangle}}{\prod_{\alpha \in \Delta}\langle \alpha^\vee, w_1\lambda_1+w_2\lambda_2\rangle}M(w_1,\lambda_1)M(w_2,\lambda_2),
\end{align*}
where
\begin{equation*}
M(w,\lambda) = \prod_{\alpha > 0 \atop w\alpha < 0} c(\langle \alpha^\vee, \lambda \rangle),\ c(z) := \frac{\xi(z)}{\xi(1+z)}.
\end{equation*}
This formula is initially stated for sufficiently regular $\lambda_1$ and $\lambda_2$, but it can be extended further by meromorphic continuation.
We note that $c(z)$ has a simple zero at $z = -1$ and a simple pole at $z = 1$. Also, $c(z) = (c(-z))^{-1}$.

In our later computation, we choose 
\begin{equation*}
C = \log T \cdot \rho^\vee := \log T \cdot \left(\frac{n-1}{2},\frac{n-3}{2},\ldots,\frac{1-n}{2}\right)
\end{equation*}
for a sufficiently large $T \in \R^+$, so that
\begin{equation*}
e^{\langle C, w_1\lambda_1+w_2\lambda_2\rangle} = T^{\langle \rho^\vee, w_1\lambda_1+w_2\lambda_2\rangle}.
\end{equation*}

\section{Growth properties of $\Lambda^CE_f$}

\subsection{Uniform moderate growth and rapid decay}

Recall that a $N_B \cap \Gamma$-invariant complex-valued function $\phi$ on $G$ is said to be \emph{of moderate growth of exponent $\lambda$} if it holds that $|\phi(g)| \ll_{C_1} a(g)^\lambda$ for all $g \in \mathfrak S^B(C_1)$. Recall also that $\phi$ is said to be \emph{of rapid decay} if it is of moderate growth of any exponent $\lambda \in \mathfrak a^*$, and is said to be of \emph{uniform} moderate growth of exponent $\lambda$, or $\lambda$-UMG for short, if $\phi$ is smooth and $L\phi$ is of moderate growth of the same exponent $\lambda$ for any $L \in U(\Lg)$.

In this section, and the next, we study the growth properties of the truncated function $\Lambda^CE_f$, with the goal of showing that it does not differ much from the ``harsh truncation'' (Theorem \ref{thm:truncate} below). In doing so we need to control the growth of $E_P(s,g)$ simultaneously for all $s$ lying on a vertical line. The following proposition accomplishes this goal.

\begin{proposition}\label{prop:uni_gro'}
Let $s \in \R$, $\chi \in \ga^*$ and $L \in U(\Lg)$. Then
\begin{equation*}
|La(g)^{s\chi}| \leq |p(s)a(g)^{s\chi}|
\end{equation*}
for some polynomial $p(s)$ whose degree is equal to the degree of $L$. $p(s)$ depends on the choice of $\chi$ and $L$.
\end{proposition}
\begin{proof}
Let us first consider the simplest case, in which $L=X \in \Lg$. The chain rule implies that
\begin{equation*}
Xa(g)^{s\chi} = -sa(g)^{(s-1)\chi} \cdot Xa(g)^\chi.
\end{equation*}
Taking the NAK decomposition $g=nak$, we can write
\begin{equation*}
Xa(g)^\chi = \frac{d}{dt}\bigg\rvert_{t=0}a( nae^{tkXk^{-1}}k)^\chi = \frac{d}{dt}\bigg\rvert_{t=0}a(ae^{tkXk^{-1}})^\chi.
\end{equation*}
Again taking the NAK decomposition $n(t)a(t)k(t)$ of $e^{tkXk^{-1}}$, we find that the above equals
\begin{equation*}
\frac{d}{dt}\bigg\rvert_{t=0}(a \cdot a(t))^\chi = a^\chi \cdot \frac{d}{dt}\bigg\rvert_{t=0}a(t)^\chi.
\end{equation*}
This quantity, of course, depends on $X$ and $k$. However, since $k$ lies in a compact subset of $G$ (namely $\mathrm{SO}(n,\R)$), its size can be bounded independently of $k$. Hence we conclude
\begin{equation*}
\left|Xa(g)^{s\chi}\right| \ll_X \left| sa(g)^{s\chi} \right|.
\end{equation*}

We proceed similarly for general $L \in U(\Lg)$.
It suffices to consider the case $L = X_d \cdots X_1$, $X_i \in \Lg$, and to prove that $|La(g)^\chi| \ll_L |a(g)^\chi|$. Again write $g=nak$ for the NAK decomposition of $g$, and recursively take, for $i=d, \ldots, 1$ and $t_d, \ldots, t_1 \in \R$, $n_i(t_d, \ldots, t_i)a_i(t_d, \ldots, t_i)k_i(t_d, \ldots, t_i)$ to be the NAK decomposition of $\exp(t_il_iX_il_i^{-1})$, where
\begin{equation*}
l_i = k_{i+1}(t_{i+1}) \cdots k_d(t_d) \cdot k.
\end{equation*}
Arguing similarly as in $d=1$ case above, we obtain
\begin{align*}
X_d \cdots X_1 a(g)^{\chi} &= \frac{d}{dt_d}\bigg\rvert_{t_d=0} \cdots \frac{d}{dt_1}\bigg\rvert_{t_1=0} (a \cdot a_d \cdots a_1)^{\chi} \\
&= a^{\chi} \cdot \frac{d}{dt_d}\bigg\rvert_{t_d=0} \cdots \frac{d}{dt_1}\bigg\rvert_{t_1=0} (a_d \cdots a_1)^{\chi}.
\end{align*}
Each of the derivatives in the last line depends on $L$ and $k$, but again, since $k$ lies in a compact set, its size is bounded independently of $k$, and thus of $g$. This completes the proof.
\end{proof}

\begin{remark}
From Proposition \ref{prop:uni_gro'} and the Leibniz integral rule, it follows that, for any $L \in U(\Lg)$ and $s \in \C$ with $\re s = c > n$,
\begin{equation*}
|LE_P(s,g)| \leq |p_L(s)E_P(c,g)|,
\end{equation*}
where $p_L(s)$ is some polynomial whose degree is equal to that of $L$. Since a Borel Eisenstein series, in particular $E_P(c,g)$, is of moderate growth, say of exponent $\chi$, by well-known results\footnote{see e.g., \cite[I.2.17]{MW95} or \cite[Appendix A]{GMP17}; it can also be derived from the result of \cite{Kim23}, with a specific value for $\chi$.}, this implies that
\begin{equation}\label{eq:more}
|LE_P(s,g)| \ll |p_L(s)a(g)^\chi|
\end{equation}
for any $s$ on the line $\re s = c > n$. This provides the desired uniform control over all $E_P(s,g)$ on a vertical line.

In fact, we will later need a slightly stronger statement than \eqref{eq:more}: namely
\begin{equation}\label{eq:more2}
|LE_B(\lambda(s)+\varepsilon\rho,g)| \ll |p_L(s,\varepsilon)a(g)^\chi|,
\end{equation}
where $\lambda(s)$ is as in the statement of Proposition \ref{prop:epeb}, $\re s = c > n$ as before, $\varepsilon \in \R$ lies in a small neighborhood of $0$, and $\chi$ is some element of $\ga^*$ depending on $c$. This follows from the fact that $|La(g)^{\lambda(s)+\varepsilon\rho}| \ll |p_L(s,\varepsilon)a(g)^{\lambda(s)+\varepsilon\rho}|$, which in turn follows from Proposition \ref{prop:uni_gro'}.
\end{remark}

The uniform moderate growth of $E_f$ and the rapid decay of $\Lambda^CE_f$ follow quickly from \eqref{eq:more}.

\begin{proposition} \label{prop:umg}
For $f:\R^+ \rightarrow \R$ such that $Mf(s)$ is of rapid decay on the vertical line $\re s = c > n$, $E_f(g)$ is of uniform moderate growth of the same exponent as $E_P(c,g)$. In addition, $\Lambda^CE_f$ is of rapid decay.
\end{proposition}


\begin{proof}

By the Leibniz integral rule and \eqref{eq:more},
\begin{align*}
\left| LE_f(g) \right| &= \left| \int_{\re s=c} Mf(s) LE_P(s,g) ds \right| \\
&\leq |E_P(c,g)| \int_{\re s=c} \left|p_L(s)Mf(s)\right| ds.
\end{align*}
The integral converges by the rapid decay of $Mf(s)$. Thus $E_f$ is of uniform moderate growth of the same exponent as $E_P(c,g)$.

The rapid decay statement follows from the fact that the Arthur truncation of a function of uniform moderate growth is of rapid decay, as mentioned in Section \ref{sec:trunc}.
\end{proof}


\subsection{Size in the truncated region}

The main goal of this section is to prove Theorem \ref{thm:truncate} below, which asserts that, for all $\phi$ of uniform moderate growth, and for all sufficiently large $T$, $\Lambda^C \phi$ is arbitrarily small outside $\mathrm{supp}\,\chi_{G,G}^C$, in which $\Lambda^C \phi(g) = \phi(g)$. It is a refinement of \cite[Proposition 13.2]{Art05} --- which shows that $\Lambda^C \phi$ is of rapid decay, but does not guarantee it is small everywhere in the truncated region --- with applications to the geometry of numbers in mind. In addition, we track the constants implicit in our estimates in case $\phi = E_P(s,g)$ or $E_B(\lambda(s)+\varepsilon\rho,g)$ with varying values of $\im s$.

We need some lemmas to proceed. Recall that we are taking
\begin{equation*}
C = C(T) = \log T \cdot \left(\frac{n-1}{2},\frac{n-3}{2},\ldots,\frac{1-n}{2}\right)
\end{equation*}
for our computations, though any sufficiently regular element of $\ga^*$ will do. We also assume $P_1 \subseteq P_2$, not both $G$, are standard parabolics in the lemmas below.

\begin{lemma} \label{lemma:umg1}
The support of $F^{P_1}(x, C)$ is bounded on $\ga^{P_1}$. More precisely, any $x$ in the support has $\varpi(x)$ bounded by a linear function on $\log T$ for all $\varpi \in \widehat\Delta^{P_1}$.
\end{lemma}
\begin{proof}
This can be read off from the definition of $F^{P_1}$. Any $x$ in the support of $F^{P_1}$ is also in $\mathfrak S^P(C_1,C)$, and hence for each $\varpi \in \widehat\Delta^{P_1}$, $\varpi(x)$ is bounded below by a constant (depending on $C_1$, which we fixed), and above by $\varpi(C)$, a linear function in $\log T$. 
\end{proof}

\begin{lemma} \label{lemma:umg1'}
Let $x = x_0 + x_1 \in \ga_{P_1}$ with $x_0 \in \ga_{P_1}^{P_2}, x_1 \in \ga_{P_2}$, such that $\sigma_{P_1}^{P_2}(x) = 1$. Then $\|x_1\| \ll \|x_0\|$, where the implied constant depends only on $P_1$ and $P_2$.
\end{lemma}
\begin{proof}
This is Lemma 8.3(b) of \cite{Art05}.
\end{proof}

\begin{lemma} \label{lemma:umg2}
If $\phi:\Gamma \backslash G \rightarrow \C$ is $\lambda$-UMG, then 
\begin{equation*}
|\phi_{P_1,P_2}(g)| \ll a(g)^{\lambda-\mu}
\end{equation*} for any linear combination $\mu$ of the elements in $\Delta^{P_2} - \Delta^{P_1}$. The implied constant depends on $\phi$ and $\mu$ only.
\end{lemma}
\begin{proof}
Observe first that it is completely harmless if any coefficient of $\alpha \in \Delta^{P_2} - \Delta^{P_1}$ in $\mu$ is negative, and if we restrict ourselves to the case in which the coefficients are integers. Hence it suffices to consider those $\mu$ whose coefficients are nonnegative integers.

We need the following statement well-known in the literature. Suppose $\Psi:\Gamma \backslash G \rightarrow \C$ is $\lambda$-UMG, and $P_\alpha$ is the maximal parabolic such that $\Delta - \Delta^{P_\alpha} = \{\alpha\}$. Then
\begin{equation} \label{eq:umg_trick}
s_\alpha\Psi(g) := \Psi(g) - \Psi_{P_\alpha}(g) \ll a(g)^{\lambda-\alpha},
\end{equation}
where the implied constant depends only on $\Psi$. 
For a proof, see \cite[8.3.6]{Gar18} or \cite[A.3.3]{GMP17} for instance.\footnote{Strictly speaking, the argument in \cite{GMP17} is written for cusp forms, and need to be modified accordingly: specifically, (A.35) there should be
\begin{equation*}
|\Psi_{i-1}(g) - \Psi_i(g)| \ll a(g)^{-\alpha}\max_{X \in \mathcal B \atop r \in [0,1)}|\delta(X)\Psi_{i-1}(\chi_{\beta_i}(r)g)| \ll a(g)^{\lambda-\alpha}.
\end{equation*}}

We apply \eqref{eq:umg_trick} to
\begin{equation*}
\Psi(g) = \phi_{P_1,P_2}(g) = (-1)^{r(P_2)}\left(\prod_{\alpha \in \Delta^{P_2}-\Delta^{P_1}} s_\alpha\right)\phi_{P_2}(g).
\end{equation*}
Note $\phi_{P_2}$ is $\lambda$-UMG because $\phi$ is.
Since $s_\alpha$'s are idempotent and pairwise commutative, the lemma follows from a repeated application of \eqref{eq:umg_trick}.
\end{proof}
\begin{remark}
It is sometimes helpful to know exactly what the hidden constant in the above lemma looks like. The arguments in \cite{Gar18} and \cite{GMP17} both reveal that the implied constant in \eqref{eq:umg_trick} is of the form
\begin{equation*}
\sum_{X \in \mathcal B} c(X)b(X,\Psi),
\end{equation*}
where $\mathcal B$ is a pre-determined basis of $\Lg$, $c(X)$ is a constant depending only on $X$, and $b(X,\Psi)$ is the constant appearing in the UMG condition $|X\Psi(g)| \ll a(g)^\lambda$ of $\Psi$. Iterating, it can be shown that, for $\mu$ with nonnegative integral coefficients that sum to $D$, the constant in Lemma \ref{lemma:umg2} is of the form
\begin{equation*}
\sum_{L \in \mathcal B^D} c(L)b(L,\phi).
\end{equation*}

Therefore, in case $\phi(g) = E_P(s,g)$ with $\re s > n$, the constant in the statement of Lemma \ref{lemma:umg2} is a polynomial in $s$ of degree at most $D$. Clearly, the analogous statement applies in case $\phi(g) = E_B(\lambda(s)+\varepsilon\rho,g)$ as well.
\end{remark}

\begin{lemma} \label{lemma:umg3}
Suppose $g \in \mathrm{supp}\,\chi_{P_1,P_2}^C$. Then there exists a constant $p_0 \in \R$ and a linear combination $\kappa$ of $\Delta^{P_2} - \Delta^{P_1}$, both independent of $g$, such that
\begin{equation*}
\sum_{\delta \in \Gamma \cap P_1 \backslash \Gamma} \chi_{P_1,P_2}^C(\delta g) \ll T^{p_0} a(g)^\kappa.
\end{equation*}
The implied constant depends only on $G$.
\end{lemma}
\begin{proof}
In the last paragraph of the proof of Proposition 13.2 in \cite{Art05}, on p.72-73, it is shown that $\chi_{P_1,P_2}^C(\delta g)$ is bounded above by $\widehat\tau_{P_1}(H_{P_1}(\delta g)-C)$, and that
\begin{equation*}
\sum_{\delta \in \Gamma \cap P_1 \backslash \Gamma} \widehat\tau_{P_1}(H_{P_1}(\delta g)-C) \leq c_T\|g\|^{N_1},
\end{equation*}
where $\|g\| = \left(\sum_{i,j}|g_{ij}|^2\right)^{1/2}$, and $N_1$ and $c_T$ are constants depending on $P_1$ and $T$, respectively. Observe that in our context $c_T$ is decreasing in $T$, because $\mathrm{supp}\,\widehat\tau_{P_1}(H_{P_1}(\delta g)-C)$ is. Hence we obtain
\begin{equation*}
\sum_{\delta \in \Gamma \cap P_1 \backslash \Gamma} \chi_{P_1,P_2}^C(\delta g) \ll_{P_1} \|g\|^{N_1}.
\end{equation*}

From the $NAK$ decomposition of $g$, and the fact that its $N$-component lies in a compact set and that $\|g\|$ is $K$-invariant, we have that $\|g\| \ll a(g)^{\kappa_0} = \exp(\kappa_0(H(g)))$ for some $\kappa_0 \in \ga^*$. Observe that $\Delta^{P_1} \cup (\Delta^{P_2}-\Delta^{P_1}) \cup \Delta_{P_2}$ is a basis of $\ga^*$, and write $\kappa_0 = \kappa^1 + \kappa^2 + \kappa_2$ accordingly. It follows from Lemma \ref{lemma:umg1} that
\begin{equation}\label{eq:p0_1}
\kappa^1(H(g)) = O(\log T).
\end{equation}
It remains to bound the size of $\kappa_2(H(g))$. Since $\kappa_2$ is a linear combination of the elements of $\Delta_{P_2}$, it suffices to bound $\beta(H(g))$ for each $\beta \in \Delta_{P_2}$.


Let us write
\begin{equation*}
H(g) = H_0^2 + H_2, \ H_0^2 = H_0^1 + H_1^2, \  H_i^j \in \ga_{P_i}^{P_j}
\end{equation*}
according to the decomposition $\ga = \ga^{P_1}\oplus\ga_{P_1}^{P_2}\oplus\ga_{P_2}$, and similarly for $C$. Lemma \ref{lemma:umg1'} implies that $\|H_2-C_2\| \ll \|H_1^2 - C_1^2\|$. Since $\Delta_{P_1}^{P_2}$ and $\Delta_{P_2}$ induce coordinate systems on $\ga_{P_1}^{P_2}$ and $\ga_{P_2}$ respectively, this shows that, for any $\beta \in \Delta_{P_2}$, there exists $\gamma \in \Delta_{P_1}^{P_2}$ (depending on $\beta$) such that
\begin{align*}
|\beta(H_2 - C_2)| \ll |\gamma(H_1^2 - C_1^2)| \Rightarrow |\beta(H_2)| \ll |\gamma(H_1^2)| + O(\log T).
\end{align*}
Since $\beta(H(g)) = \beta(H_2)$ (because $\beta(H_0^2) = 0$), and similarly $\gamma(H(g)) = \gamma(H_1^2)$, the above inequality implies
\begin{equation*}
|\beta(H(g))| \ll |\gamma(H(g))| + O(\log T).
\end{equation*}
Let us write $\gamma = \sum_{\alpha \in \Delta^{P_2}} c_\alpha \alpha \in \Delta_{P_1}^{P_2}$ in accordance with the decomposition $(\ga^{P_2})^* = (\ga_{P_1}^{P_2})^*\oplus(\ga^{P_1})^*$, so that
\begin{equation*}
|\gamma(H(g))| \leq \sum_{\alpha \in \Delta^{P_2}} |c_\alpha \alpha(H(g))|.
\end{equation*}
For $\alpha \in \Delta^{P_1}$, $\alpha(H(g)) = O(\log T)$ by Lemma \ref{lemma:umg1}. For $\alpha \in \Delta^{P_2} - \Delta^{P_1}$, the inequalities \eqref{eq:bowwow2} proved below imply that $\alpha(H(g)) > \log T$; in particular, it is positive. This shows that, for $\kappa_\beta = \sum_{\alpha \in \Delta^{P_2} - \Delta^{P_1}} |c_\alpha|\alpha$, we have
\begin{equation}\label{eq:p0_2}
|\beta(H(g))| \ll \kappa_\beta(H(g)) + O(\log T),
\end{equation}
as desired.

To complete the proof of the lemma, we set $\kappa = \kappa^2+\sum_{\beta \in \Delta_{P_2}}\kappa_\beta$, and determine $p_0$ by collecting the constants in the $O(\log T)$ terms from \eqref{eq:p0_1} and \eqref{eq:p0_2}.
\end{proof}

Combining the above statements, we arrive at the following desired conclusion.
\begin{theorem} \label{thm:truncate}
For $\lambda$-UMG $\phi$ and any $p > 0$, there exists a constant $c=c(p,\phi)>0$ such that $|\Lambda^C\phi(g)| < cT^{-p}$ outside the support of $\chi_{G,G}^C$.
\end{theorem}
\begin{proof}
In the light of \eqref{eq:arthur_t_defn2} and Lemma \ref{lemma:umg3}, it suffices to prove, for each $P_1 \subseteq P_2$, $P_1 \neq G$, and for $g$ contained in the support of $\chi_{P_1,P_2}^C$, that
\begin{equation*}
|\phi_{P_1,P_2}(g)| \ll T^{-p}a(g)^{-\kappa}
\end{equation*}
for any choice of $p>0$ and a certain $\kappa \in \mathrm{span}(\Delta^{P_2}-\Delta^{P_1})$.

On the other hand, we already know from Lemma \ref{lemma:umg2} that
\begin{equation*}
|\phi_{P_1,P_2}(g)| \ll a(g)^{\lambda-\mu}
\end{equation*}
holds for some $\lambda \in \ga^*$ and any $\mu \in \mathrm{span}(\Delta^{P_2} - \Delta^{P_1})$.
Write
\begin{align*}
H(g) &= H_0^2 + H_2, \ H_0^2 = H_0^1 + H_1^2, \ & H_i^j \in \ga_{P_i}^{P_j}, \\
\lambda-\mu &= \nu_0^2 + \nu_2, \ \nu_0^2 = \nu_0^1 + \nu_1^2, \ & \nu_i^j \in (\ga_{P_i}^{P_j})^*.
\end{align*}
Our goal is then to show that, for an appropriate choice of $\mu$,
\begin{equation*}
(\lambda-\mu)(H(g)) = \nu_0^1(H_0^1) + \nu_1^2(H_1^2) + \nu_2(H_2)
\end{equation*}
is bounded from above by $-p\log T$. Note that there is no loss of generality in assuming $\kappa=0$, by replacing $\mu$ with $\mu - \kappa$.


Fix an $\alpha = \alpha_i \in \Delta^{P_2}-\Delta^{P_1}$, and suppose the Levi block of $P_1$ containing the $i$-th row has blocksize $n_1$, and the block containing the $(i+1)$-st row has blocksize $n_2$. We claim that
\begin{equation}\label{eq:bowwow}
\alpha(H_0^2) > \frac{1}{2}\log T + \frac{1}{n_1+n_2}\alpha(H_1^2).
\end{equation}
This follows from the additional claims that
\begin{equation} \label{eq:bowwow2}
\alpha(H_0^1) > \left(1-\frac{n_1+n_2}{2}\right)\log T, \ \alpha(H_1^2) > \left(\frac{n_1+n_2}{2}\right)\log T,
\end{equation}
by adding the former and $(1-(n_1+n_2)^{-1})$ times the latter.
To see the former inequality, observe that $F^{P_1}(g,C) = 1$ implies
\begin{equation*}
\varpi(H_0^1) < \varpi(C) = \log T \cdot \varpi((\rho^\vee)^{P_1})
\end{equation*}
for all $\varpi \in \widehat\Delta_0^{P_1}$. $(\rho^\vee)^{P_1}$ here can be computed explicitly: if $P_1$ consists of the Levi blocks of size $k_1,\ldots,k_r$, then
\begin{equation*}
(\rho^\vee)^{P_1} = \left(\frac{k_1-1}{2}, \ldots, \frac{1-k_1}{2}; \ldots ; \frac{k_r-1}{2}, \ldots, \frac{1-k_r}{2} \right).
\end{equation*}
This implies that the $i$-th entry of $H_0^1$ is greater than $-1/2 \cdot (n_1-1)\log T$, and the $(i+1)$-st entry of $H_0^1$ is less than $1/2 \cdot (n_2-1)\log T$, hence the desired inequality. Similarly, the latter inequality of \eqref{eq:bowwow2} follows from our assumption $\sigma_{P_1}^{P_2}(H_{P_1}(g) - C) = 1$, which implies that $\tau_{P_1}^{P_2}(H_1^2 - C) = 1$ and thus
\begin{equation*}
\alpha(H_1^2) > \alpha(C_{P_1}^{P_2}) = \log T \cdot \alpha((\rho^\vee)_{P_1}^{P_2}).
\end{equation*}
It can be computed that $\alpha((\rho^\vee)_{P_1}^{P_2}) = 1/2 \cdot (n_1 + n_2)\log T$, proving \eqref{eq:bowwow2} in full.

To complete the proof of the theorem, let us start by writing
\begin{align*}
\lambda_0^2 := \lambda_{P_0}^{P_2} = \lambda' + \sum_{\alpha \in \Delta^{P_2} - \Delta^{P_1}} l_\alpha \alpha
\end{align*}
for some $\lambda' \in \ga^{P_1}$ and $l_\alpha$'s --- in fact, they are uniquely determined --- and
\begin{align*}
\mu = \sum_{\alpha \in \Delta^{P_2}-\Delta^{P_1}} m_\alpha \alpha,
\end{align*}
with $m_\alpha$'s to be determined soon. Take $m_\alpha$'s large enough so that $m'_\alpha := m_\alpha - l_\alpha \geq 0$.
Then by \eqref{eq:bowwow}, there exists a constant $c(P_1,P_2)>0$ such that
\begin{align*}
\nu_0^2(H_0^2) &= \lambda_0^2(H_0^2) - \mu(H_0^2) \\
&< \lambda'(H_0^1) -\frac{1}{2}\sum_{\alpha \in \Delta^{P_2} - \Delta^{P_1}} m'_\alpha\log T - c(P_1,P_2)\sum_{\alpha \in \Delta^{P_2} - \Delta^{P_1}} m'_\alpha\alpha(H_1^2) \\
&< O_\lambda(\log T) -\frac{1}{2}\sum_{\alpha \in \Delta^{P_2} - \Delta^{P_1}} m'_\alpha\log T - c(P_1,P_2)\sum_{\alpha \in \Delta^{P_2} - \Delta^{P_1}} m'_\alpha\alpha(H_1^2).
\end{align*}
The last inequality here is a consequence of Lemma \ref{lemma:umg1}.
It remains to show that, by increasing $m_\alpha$'s if necessary, 
\begin{equation*}
c(P_1,P_2)\sum_{\alpha \in \Delta^{P_2} - \Delta^{P_1}} m'_\alpha\alpha(H_1^2) + O(\log T) > \nu_2(H_2).
\end{equation*}
Since $\nu_2$ is a linear functional, we have $|\nu_2(H_2)| \ll \|H_2\|$. Also, by Lemma \ref{lemma:umg1'} we have $\|H_2-C_2\| \ll \|H_1^2-C_1^2\|$, which implies that $\|H_2\| \ll \|H_1^2\|+O(\log T)$ by the triangle inequality.
On the other hand, $\Delta^{P_2} - \Delta^{P_1}$ induces a coordinate system on $\ga_{P_1}^{P_2}$, so $\|H_1^2\| \ll \sum_{\alpha \in \Delta^{P_2} - \Delta^{P_1}} |\alpha(H_1^2)|$; but since all $\alpha(H_1^2) > 0$ by \eqref{eq:bowwow2}, we can take the absolute value sign off the latter. Summarizing, we obtain
\begin{equation*}
|\nu_2(H_2)| \ll \sum_{\alpha \in \Delta^{P_2} - \Delta^{P_1}} \alpha(H_1^2) + O(\log T),
\end{equation*}
as desired.
This completes the proof.
\end{proof}

\begin{remark}
In case $\phi(g) = E_P(s,g)$ with $\re s > n$, note that the dependence of the constant $c$ in the statement of Theorem \ref{thm:truncate} on $s$ is at most by a polynomial factor in $s$, since in our argument the dependence of $c$ on $\phi$ comes directly from Lemma \ref{lemma:umg2}, in which we already observed in the remark following its proof that the implicit constant there has that property. The analogous statement applies in case $\phi(g) = E_B(\lambda(s)+\varepsilon\rho,g)$.
\end{remark}


\ignore{
We would also like to show a somewhat stronger statement than Proposition \ref{prop:umg}. For $i=1,\ldots,n-1$, denote by $P_i$ the maximal standard parabolic subgroup of $G$ corresponding to the flag $0 \subseteq E_k = \mathrm{span}(\ve_{n-k+1},\ldots,\ve_n) \subseteq \R^n$. Define, for a locally bounded function $\phi$ on $\Gamma \backslash G$,
\begin{equation*}
\sigma_i\phi = \phi - \phi_{P_i}.
\end{equation*}
It can be checked that $\sigma_i^2 = \sigma_i$ and $\sigma_i\sigma_j=\sigma_j\sigma_i$. Also define $\sigma \phi(g) = \sigma_1 \cdots \sigma_{n-1}\phi(g)$. It holds that
\begin{equation*}
\sigma\phi(g) = \sum_{P \supseteq P_0} (-1)^{r(P)} \phi_P(g),
\end{equation*}
and that $\sigma\phi(g)=\Lambda^C\phi(g)$,
{\color{red} $\leftarrow$ this isn't obvious and I want to get around it. See Arthur, ``Intro to trace formula'' Prop 13.2}
provided $\alpha_i(H(g))$ is sufficiently large for any $i$.

\begin{proposition}\label{prop:unif_decay}
Let $\re s = c > n$. Then $\sigma E(s,g)$ (and thus $\Lambda^CE(s,g)$ too) decays almost uniformly in the sense that, for any Schwartz function $h:\R\rightarrow\R$, $h(|s|)\sigma E(s,g)$ decays at a rate independent of $\im s$.
\end{proposition}
\begin{proof}
Our proof is in two steps. First, we show that $\sigma E(s,g)$ is controlled by $E(c,g)$ in a certain sense. Then we show that $E(c,g)$ is of exponential decay, completing the proof. Much of the proof follows \cite[Appendix A]{GMP17} nearly verbatim, except that we need to make some minor changes so that their argument carries over to the Eisenstein series. For brevity, we will only point out what modifications are needed.

We fix a basis $\mathcal B$ of $\mathrm{Lie}(G)$, as in \cite[(A.7)]{GMP17}.
Let us write $\alpha_i \in \Delta$ for the simple root corresponding to $P_i$. Then, by replicating the argument of \cite[A.3.3]{GMP17}, it can be shown that, for $g \in \mathfrak S(t)$ and a smooth function $\phi$ on $\Gamma \backslash G$,
\begin{equation*}
|\sigma_i\phi(g)| \ll_t a(g)^{-\alpha_i}\max_{\gamma \in \Gamma \cap N^P \backslash N^P \atop X \in \mathcal B} |X \cdot \phi(g)|.
\end{equation*}
Applying this to $E_P(s,g), \sigma_{n-1}E_P(s,g), \sigma_{n-2}\sigma_{n-1}E_P(s,g),$ and so on iteratively, and also using \eqref{eq:diff_comp}, we obtain
\begin{align*}
|\sigma E_P(s,g)| &\ll_t a(g)^{-\sum_i \alpha_i}\max_{\gamma \in \Gamma \cap N \backslash N \atop X_1,\ldots,X_{n-1} \in \mathcal B} |X_1 \cdots X_{n-1} \cdot E_P(s,g)| \\
&\leq p(|s|)a(g)^{-\sum_i \alpha_i}\max_{\gamma \in \Gamma \cap N \backslash N \atop X_1,\ldots,X_{n-1} \in \mathcal B} |X_1 \cdots X_{n-1} \cdot E_P(c,g)|,
\end{align*}
for a polynomial $p$ of degree $n-1$.

{\color{red} modify this, and fuse with below to complete the proof. also, we can and should be specific about what $p$ is. ideally, I want (almost) uniform exponential decay, but that may not be so easy

See Arthur, ``Intro to trace formula'' Prop 13.2

}

We would now like to bound the size of $X_1 \cdots X_{n-1}E_P(c,g)$. It suffices to bound $E_P(c,g)$, since $X_1 \cdots X_{n-1}E_P(c,g)$ is a constant multiple of it. More generally, let $\phi$ be a $K$-invariant automorphic form on $\Gamma \backslash G$ that is of moderate growth of exponent $\lambda$.
By \cite[Proposition A.2.4]{GMP17}, there exists a function $b(g) \in C^\infty_c(G)$ (called $f_c$ in \cite{GMP17}), supported on a compact set $C \subseteq G$, satisfying
\[b * \phi = \phi\]
and
\[\|X\cdot b\|_1 \leq (\gamma N)^{\gamma N}\]
for all $X \in U(\mathrm{Lie}(G))$ of degree $N$, where $\gamma>0$ depends only on $G$.\footnote{There seems to be a typo in the statement of \cite[Proposition A.2.3]{GMP17}: ``for all $X \in \mathcal B$'' should be ``for all $X \in \mathfrak U$ of degree $N$.''} Therefore
\begin{equation*}
X \cdot \phi = X \cdot (b * \phi) = (X \cdot b) * \phi,
\end{equation*}
and, for all $g \in \mathfrak S(t)$,
\begin{align*}
|((X \cdot b) * \phi)(g)| &\leq \|X\cdot b\|_1 \|\phi(g)\mathbf{1}_{gC}\|_\infty \\
&\ll_{\phi,t} \|X\cdot b\|_1 \sup_{a \in C_A} a^\lambda \cdot a(g)^\lambda \\
&\ll_{\phi} (\gamma N)^{\gamma N}a(g)^\lambda,
\end{align*}
where $C_A \subseteq A$ is as in \cite[Lemma 8.3.4]{Gar18} --- see also the subsequent computation there, which is very similar to ours above. Using this inequality in place of \cite[Proposition A.2.5]{GMP17} in the proof of \cite[Theorem A.3.1]{GMP17}, we can show that there exists a ......

\end{proof}
} 

\section{The truncated $L^1$-norm} \label{sec:l1}

\subsection{Statements of results}

The main goal of this section is to prove the following truncated version of the Siegel integral formula.
\begin{theorem} \label{thm:L1-norm}
Choose a parameter $0 < \eta < 1$, and let $f:\R_{\geq 0} \rightarrow \R$ be a smooth function with compact support. Then
\begin{equation*}
\int_{\Gamma\backslash G} \mathrm H^C E_f(g)dg = \int_{\Gamma\backslash G} \Lambda^C E_f(g)dg,
\end{equation*}
and
\begin{align*}
&\int_{\Gamma\backslash G} \Lambda^C E_f(g)dg  \\
&= Mf(n)\xi(n,k)\vol^C(\Gamma\backslash G) + O_{n,f,\eta}\left(T^{-\frac{\eta k(n-k)}{2}}\right),
\end{align*}
where 
\begin{align*}
\xi(n,k)=\frac{\xi(2)\cdots\xi(n-k)}{\xi(k+1)\cdots\xi(n)}
\end{align*}
and $\vol^C(\Gamma\backslash G)$ is the volume of the truncated fundamental domain $\mathrm{supp}\,\chi^C_{G,G}$.
\end{theorem}

The error term above is given precisely by \eqref{eq:L1error} below, which for certain functions of interest take the form
\begin{equation*}
O_n\left(Mf(n-\eta)T^{-\frac{\eta k(n-k)}{2}}\right)
\end{equation*}
for any $0 < \eta < 1$.

Theorem \ref{thm:L1-norm} is a slight refinement of the works of Miller \cite{Mil20} and parts of Thurman \cite{Thu24}, in that it provides a direct comparison with the volume of the truncated domain. For comparison, we also show the following, which is likely not new but we were unable to find a reference for.
\begin{proposition} \label{prop:tr_fund_vol}
$\vol(\G\backslash G) - \vol^C(\G\backslash G) = O(T^{-n(n-1)/2}).$
\end{proposition}


\subsection{Proof of Theorem \ref{thm:L1-norm}}

The former statement is clear, since
\begin{align*}
\int_{\Gamma\backslash G} \Lambda^C E_f(g)dg = \int_{\Gamma\backslash G} E_f(g) \Lambda^C1dg = \int_{\Gamma\backslash G} \mathrm H^C E_f(g)dg.
\end{align*}

We now give a proof of the latter statement.
Write $\lambda = \lambda(s)$ as defined in \eqref{eq:lambda}, and $\tilde\lambda = \lambda + \varepsilon\rho$ for a small $\varepsilon > 0$, a ``perturbation'' of $\lambda$. With an appropriate adaptation of Proposition \ref{prop:formulation} below,\footnote{It is in fact an overkill, since $\int \Lambda^CE_f$ is actually an integral over a compact set.} and the Maass-Selberg relation, one has
\begin{align}
&\int_{\G\backslash G} \Lambda^CE_f(g)dg \label{eq:L1goal} \\
&= \lim_{\varepsilon\rightarrow0} \frac{1}{2\pi i}\int_{\re s = c} Mf(s)\sum_{w \in W} \frac{e^{\langle C, w\tilde\lambda-\rho\rangle}}{\prod_{\alpha \in \Delta} \langle\alpha^\vee,w\tilde\lambda-\rho\rangle}M(w,\tilde\lambda)ds \notag
\end{align}
for any $c>n$. 
The reason we perturb $\lambda$ is to avoid the possibility that $\langle\alpha^\vee,w\tilde\lambda-\rho\rangle=0$, which makes interpreting the expression difficult.

We will perform a residue calculus on the above expression, starting by setting $c=n+1/2$ and shifting the contour leftward to $c=n-\eta$ for any fixed $0 < \eta < 1$, picking up all the residue on the way. For this maneuver to be legitimate, the integrand need to vanish as $|\im s| \rightarrow \infty$ on the vertical strip $n-\eta \leq \re s \leq n+1/2$. The lemma below ensures that this is indeed the case.

\begin{lemma}\label{lemma:misc_growth}
Continue with the assumptions above. Then on the strip $n-\eta \leq \re s < n+1$, $Mf(s)$ is holomorphic and of rapid decay, and $M(w, \tilde\lambda)$ grows at most polynomially in $\im s$ --- in fact, it is of order $O_n(|\im s|^{0.501\eta+O(\varepsilon)})$.
\end{lemma}
\begin{remark}
As is made clear in the proof, the choice of the number $0.501$ is an arbitrary one, and it can be adjusted to be arbitrarily close to $1/2$ if necessary.
\end{remark}
\begin{proof}
The statement about $Mf$ is well-known; it follows immediately from the fact that $Mf(\sigma + it)$ is the Fourier transform of $f(e^x)e^{\sigma x}$ as a function of $t$. The claimed property in fact holds on the entire right half plane $\{z:\re z>0\}$.

As for $M(w, \tilde\lambda)$, it is a finite product of expressions of the form
\begin{equation*}
\frac{\xi(s+a)}{\xi(s+a+1)} = \pi^{1/2}\frac{\Gamma((s+a)/2)\zeta(s+a)}{\Gamma((s+a+1)/2)\zeta(s+a+1)},
\end{equation*}
where $a \in [-n+1,-1] + O(\varepsilon)$. Hence $\re (s+a) \in [1-\eta,n)+O(\varepsilon)$. 

Let us first consider the contributions form the zeta functions. It is well-known (e.g., \cite[Exercise 10.1.19e]{MV06}) that
\begin{align*}
|\zeta(\sigma+it)| \ll |t|^{\frac{1-\sigma}{2}+\iota}
\end{align*}
for any $0 \leq \sigma \leq 1$ and $\iota>0$; also, for $\sigma >1$, $|\zeta(\sigma+it)| \ll 1$ is trivial.
On the other hand, it is also known (e.g., Theorem 8.29 of \cite{IK04}) that
\begin{align*}
\frac{1}{\zeta(\sigma+it)} \ll (\log t)^\frac{2}{3}(\log\log t)^\frac{1}{3}
\end{align*}
for $\sigma \geq 1-c(\log t)^{-2/3}(\log\log t)^{-1/3}$ and $t \geq 3$, where $c>0$ is some absolute constant. Thus, on the said strip, $\zeta(s+a)/\zeta(s+a+1)$ is bounded by a polynomial in $\im s$. In fact, all but at most one factor of $M(w,\tilde\lambda)$ would have logarithmic growth in $\im s$, since $\langle \alpha^\vee,\tilde\lambda \rangle = s-n+1+O(\varepsilon)$ is possible if and only if $\alpha = \alpha_{1n}$, and in all other cases either $\langle \alpha^\vee,\tilde\lambda\rangle$ is constant with respect to $s$ or $\re\langle \alpha^\vee,\tilde\lambda\rangle >1$. Hence we can conclude that the product of the zeta factors are bounded by $O(|\im s|^{0.501\eta+O(\varepsilon)})$, say.

It remains to bound the gamma functions. From the product formula
\begin{align*}
\Gamma(z) = \frac{1}{z}\prod_{n=1}^\infty\frac{(1+1/n)^z}{1+z/n}
\end{align*}
of Euler, one can derive another well-known formula
\begin{equation*}
|\Gamma(\sigma + it)|^2 = |\Gamma(\sigma)|^2\prod_{k=0}^\infty\left(1+\frac{t^2}{(\sigma+k)^2}\right)^{-1}.
\end{equation*}
This implies, for $\sigma \geq 0$,
\begin{align*}
\frac{|\Gamma(\sigma+it)|^2}{|\Gamma(\sigma+1/2+it)|^2} &= \frac{|\Gamma(\sigma)|^2}{|\Gamma(\sigma+1/2)|^2}\prod_{k=0}^\infty \frac{\left(1+\frac{t^2}{(\sigma+1/2+k)^2}\right)}{\left(1+\frac{t^2}{(\sigma+k)^2}\right)} \\
& \leq \frac{|\Gamma(\sigma)|^2}{|\Gamma(\sigma+1/2)|^2}.
\end{align*}
Summing up everything so far, we deduce
\begin{equation*}
M(w,\tilde\lambda) = O_n(|\im s|^{0.501\eta+O(\varepsilon)}),
\end{equation*}
as desired.
\end{proof}

By inspection, it is clear that all residues of the right-hand side of \eqref{eq:L1goal} on the strip $\re s \in [n-\eta,n+1/2]$ occur at $s=n\pm a\varepsilon$ for $a \in \{0,1,\ldots,n-1\}$.
The proposition below deals with all of them at once.

\begin{proposition} \label{prop:chicken}
As $\varepsilon \rightarrow 0$,
\begin{equation} \label{eq:s1res}
\frac{1}{2\pi i}\mathrm{Res}_{s=n+O(\varepsilon)} Mf(s)\sum_{w \in W} \frac{e^{\langle C, w\tilde\lambda-\rho\rangle}}{\prod_{\alpha \in \Delta} \langle \alpha^\vee, w\tilde\lambda-\rho \rangle}M(w,\tilde\lambda)
\end{equation}
converges to 
\begin{equation*}
Mf(n)\xi(n,k) \vol^C(\Gamma\backslash G).
\end{equation*}
Here, $\mathrm{Res}_{s=n+O(\varepsilon)}$ denotes the sum of all $s$-residues at $n,n\pm\varepsilon,n\pm2\varepsilon\ldots,n\pm(n-1)\varepsilon$.
\end{proposition}
\begin{proof}
By the Maass-Selberg relation,
\begin{equation} \label{eq:trunc_domain}
\vol^C(\Gamma\backslash G)=\lim_{\varepsilon \rightarrow 0}\sum_{w \in W} \frac{e^{\langle C, -w\tilde\rho-\rho\rangle}}{\prod_{\alpha \in \Delta} \langle \alpha^\vee, -w\tilde\rho-\rho \rangle}M(w,-\tilde\rho).
\end{equation}
Therefore the required work is to compare \eqref{eq:s1res} and \eqref{eq:trunc_domain}.

Recall the functional equation
\begin{equation} \label{eq:fnleq_M}
M(w_1w_2,\lambda) = M(w_1,w_2\lambda)M(w_2,\lambda)
\end{equation}
for any $w_1,w_2 \in W$.
Let $w_* \in W$ be the unique element such that
\begin{align*}
w_*(n-k+1) &< w_*(n-k+2) < \ldots < w_*(n) \\
&< w_*(1) < w_*(2) < \ldots < w_*(n-k).
\end{align*}
Also let $\tilde\rho$ be such that $w_*\tilde\lambda(n) = -\tilde\rho$; observe that $\tilde\rho = \rho + O(\varepsilon)$ indeed. \eqref{eq:fnleq_M} implies
\begin{equation*}
M(ww_*,\tilde\lambda) = M(w,-\tilde\rho)M(w_*,\tilde\lambda)
\end{equation*}
for any $w\in W$, and thus the left-hand side of \eqref{eq:trunc_domain} is equal to
\begin{equation*}
\lim_{\varepsilon\rightarrow0}\frac{1}{M(w_*,\tilde\lambda(n))}\sum_{w \in W} \frac{e^{\langle C, w\tilde\lambda(n)-\rho\rangle}}{\prod_{\alpha \in \Delta} \langle \alpha^\vee, w\tilde\lambda(n)-\rho \rangle}M(w,\tilde\lambda(n)).
\end{equation*}
Here, one computes that
\begin{equation*}
M(w_*,\tilde\lambda(n)) = \xi(1+(n-1)\varepsilon)\cdot \frac{\xi(2+O(\varepsilon)) \cdots \xi(n-k+O(\varepsilon))}{\xi(k+1+O(\varepsilon)) \cdots\xi(n+O(\varepsilon))}.
\end{equation*}
Therefore, it suffices to show that, for each $w \in W$,
\begin{align}\label{eq:chicken}
&\lim_{\varepsilon\rightarrow0} (n-1)\varepsilon\cdot\frac{e^{\langle C, w\tilde\lambda(n)-\rho\rangle}}{\prod_{\alpha \in \Delta} \langle \alpha^\vee, w\tilde\lambda(n)-\rho \rangle}M(w,\tilde\lambda(n)) \\
&= \frac{1}{2\pi i}\mathrm{Res}_{s=n+O(\varepsilon)}\frac{e^{\langle C, w\tilde\lambda-\rho\rangle}}{\prod_{\alpha \in \Delta} \langle \alpha^\vee, w\tilde\lambda-\rho \rangle}M(w,\tilde\lambda). \notag
\end{align}

Before we start, let us remark that, in the expression
\begin{equation*}
\frac{e^{\langle C, w\tilde\lambda-\rho\rangle}}{\prod_{\alpha \in \Delta} \langle \alpha^\vee, w\tilde\lambda-\rho \rangle}M(w,\tilde\lambda),
\end{equation*}
any potential singularity in the $\varepsilon\rightarrow0$ limit that is independent of $s$ is a removable one. First of all, $M(w,\tilde\lambda)$ induces no such singularity, since there exists no positive root $\alpha$ such that $\langle \alpha^\vee, \lambda\rangle = 1$. As for the denominator terms, if $\langle \alpha^\vee_j,w\tilde\lambda-\rho\rangle = O(\varepsilon)$ for some $j$, then $w^{-1}(j)$ and $w^{-1}(j+1)$ must be adjacent and in the same block, and thus $w^{-1}(j) > w^{-1}(j+1)$; this causes $c(\langle \alpha^\vee_{w^{-1}(j+1)},\tilde\lambda\rangle)= c(-1+O(\varepsilon)) = O(\varepsilon)$ to appear as a factor of $M(w,\tilde\lambda)$, which offsets the size of $\langle \alpha^\vee_j,w\tilde\lambda-\rho\rangle$.

Therefore, for either side of \eqref{eq:chicken} to yield a nonzero value, there are only two mutually exclusive possibilities: $w(1)+1=w(n)$ or $w(1) > w(n)$ --- note that this also implies that all poles of \eqref{eq:s1res} at $s=n+O(\varepsilon)$ are simple.

In case $w(1)+1=w(n)$, the denominator on the left-hand side of \eqref{eq:chicken} has a factor $(n-1)\varepsilon$; in case $w(1) > w(n)$, $M(w,\tilde\lambda)$ contains the factor $c(\langle \alpha^\vee_{1n},\tilde\lambda(n)\rangle)$, which has the factor $\xi(1+(n-1)\varepsilon)=1/((n-1)\varepsilon+O(\varepsilon^2))$. On the right-hand side of \eqref{eq:chicken}, the counterpart of these corresponding terms are $s-n+(n-1)\varepsilon$ and $\xi(s-n+1+(n-1)\varepsilon)$, respectively. From these observations, it is clear that the equality \eqref{eq:chicken} holds.
\end{proof}

Proposition \ref{prop:chicken} accounts for the main term appearing in the statement of Theorem \ref{thm:L1-norm}. The error term arises from the integral
\begin{equation}\label{eq:L1error}
\lim_{\varepsilon\rightarrow0} \frac{1}{2\pi i}\int_{\re s = n-\eta} Mf(s)\sum_{w \in W} \frac{e^{\langle C, w\tilde\lambda-\rho\rangle}}{\prod_{\alpha \in \Delta} \langle\alpha^\vee,w\tilde\lambda-\rho\rangle}M(w,\tilde\lambda)ds
\end{equation}
that remains after the contour shift. The only obstacle to bounding this integral is the possibility that $\langle \alpha^\vee_j, w\tilde\lambda-\rho \rangle = O(\varepsilon)$ for some $j$. This means that $\langle \alpha^\vee_j,w\lambda\rangle = 1$, and thus $w^{-1}(j) = w^{-1}(j+1)+1$ and $j\neq n-k$. Hence the factor $c(\langle \alpha^\vee_{w^{-1}(j+1)},\tilde\lambda \rangle) = O(\varepsilon)$ appears in $M(w,\tilde\lambda)$, which balances out the denominator term $\langle \alpha^\vee_j, w\tilde\lambda-\rho \rangle$. Hence the above integral converges.

It is easy to show that $\langle \rho^\vee,w\lambda\rangle$ is maximized when the entries of $w\lambda$ are sorted in descending order. Thus the (real part of) $T$-degree of the above integral is maximized when
\begin{equation*}
w\lambda(n-\eta) = \rho +(0,\ldots,0;\underbrace{\eta,\ldots,\eta}_{\mbox{\tiny $k$ times}}),
\end{equation*}
i.e. $w$ ``reverses order in each block,'' to
\begin{equation*}
\eta \left(\frac{-n+1}{2} + \cdots + \frac{-n+2k-1}{2}\right) = -\frac{\eta k(n-k)}{2}.
\end{equation*}
Therefore \eqref{eq:L1error} is bounded by $O_{n,f}(T^{-\frac{\eta k(n-k)}{2}})$. This completes the proof of Theorem \ref{thm:L1-norm}.

\subsection{Proof of Proposition \ref{prop:tr_fund_vol}}

Recall from \eqref{eq:trunc_domain} that
\begin{equation*}
\vol^C(\Gamma\backslash G)=\lim_{\varepsilon \rightarrow 0}\sum_{w \in W} \frac{e^{\langle C, -w\tilde\rho-\rho\rangle}}{\prod_{\alpha \in \Delta} \langle \alpha^\vee, -w\tilde\rho-\rho \rangle}M(w,-\tilde\rho).
\end{equation*}
Suppose that $\langle \alpha^\vee_j,-w\tilde\rho-\rho\rangle = O(\varepsilon)$. Then $\langle \alpha^\vee_j, w\rho\rangle = -1$, so $w^{-1}(j) = w^{-1}(j+1)+1$. This causes $c(\langle \alpha^\vee_{w^{-1}(j+1)}, -\tilde\rho\rangle) = O(\varepsilon)$ to appear as a factor of $M(w,-\tilde\rho)$, balancing out the growth of $\langle \alpha^\vee_j,-w\tilde\rho-\rho\rangle$. In particular, the above limit converges.

The flip side of this is that, if $c(\langle \alpha^\vee_h, -\tilde\rho\rangle) = O(\varepsilon)$ appears as a factor of $M(w,-\tilde\rho)$, then the summand corresponding to $w\in W$ vanishes as $\varepsilon\rightarrow0$ unless $\langle \alpha^\vee_{w(h+1)},-w\tilde\rho-\rho\rangle=O(\varepsilon)$. This greatly limits the number of $w\in W$ that contribute nontrivially to the formula \eqref{eq:trunc_domain}: if $w(h+1) < w(h)$, then $w(h+1)+1=w(h)$. Let us say that $w \in W$ is \emph{permissible} if it satisfies this condition.

Permissible permutations lead to the following combinatorial consideration. Partition an ordered set $A=\{1,2,\ldots,n\}$ into $p$ ordered contiguous blocks $A_1,\ldots,A_p$, so that $A=A_1\cup \cdots \cup A_p$; here, taking union is order-sensitive too. Also, for an ordered set $S$, write $S^r$ for the set consisting of the same elements but in reversed order; for example, if $S= \{1,2,3\}$, then $S^r =\{3,2,1\}$.

As an ordered set, $-\rho$ can be naturally identified with $A$. Each permissible $w \in W$ must take $-\rho$ into an ordered set $\varphi$ of the form
\begin{equation*}
\varphi = A_1^r \cup A_2^r \cup \cdots \cup A_p^r.
\end{equation*}
If $\varphi_j < \varphi_{j+1}$, then $\langle \rho^\vee,(j\ j+1)\varphi\rangle > \langle \rho^\vee, \varphi\rangle$ --- the difference is precisely $\varphi_{j+1}-\varphi_j$. From this, we see that $\langle \rho^\vee, \varphi\rangle$ increases whenever we move $A_i^r$ to the left of $A_{i-1}^r$, and then concatenate the resulting $A_i^r \cup A_{i-1}^r$ into a single block. In particular, for a given $\varphi$, there exists $\varphi'$ with a smaller number of blocks, also coming from a permissible permutation, such that $\langle \rho^\vee, \varphi'\rangle > \langle \rho^\vee, \varphi'\rangle$.

Therefore, the highest $T$-degree $\langle \rho^\vee, -w\tilde\rho-\rho\rangle$ occurs when $p=1$ with $-w\tilde\rho=\rho$, and the degree is zero. Clearly, the next highest $T$-degree occurs when $p=2$, and a moment's reflection shows that it happens when
\begin{equation*}
\varphi = (n-1,\ldots,2,1,n) \mbox{ or } (1,n,\ldots,3,2),
\end{equation*}
and the corresponding $T$-degree is $-n(n-1)/2$. In fact, it can also be shown that the contributions from these two subleading $\varphi$'s are identical and thus do not cancel each other out; hence this error estimate is sharp.

\section{The truncated inner product}



{\color{blue} Note 4/13/25: move $c_1$ a little less than $c_2$ to avoid annoying problems in Secs. 5.3 and 5.4. This also facilitates the application.}

\subsection{Formulation}
Let $f_1,f_2:\R_{\geq 0} \rightarrow \R$ be smooth functions with compact support. For $1 \leq k_1,k_2 \leq n$, let $P_1,P_2$ be the maximal parabolic subgroups $P(k_1),P(k_2)$ of $G$, respectively. Accordingly, for $s_1,s_2 \in \C$, we write
\begin{align*}
\lambda_1(s_1) &\equiv (-n, \ldots, -k_1-1; -s_1-k_1, \ldots, -s_1-1), \\
\lambda_2(s_2) &\equiv (-n, \ldots, -k_2-1; -s_2-k_2, \ldots, -s_2-1),
\end{align*}
similarly to \eqref{eq:lambda}.
The goal of this section is to estimate the truncated correlation
\begin{equation} \label{eq:goal}
\int_{\Gamma\backslash G} \Lambda^C E_{P_1,f_1}(g)E_{P_2,f_2}(g) dg.
\end{equation}
The statement below provides our starting point.
\begin{proposition} \label{prop:formulation}
Continue with the above notations.
Then \eqref{eq:goal} is equal to, upon writing
$\tilde\lambda_1 = \lambda_1 + \varepsilon\rho$,
\begin{align} \label{eq:theexp'}
\lim_{\varepsilon\rightarrow 0}&\frac{1}{(2\pi i)^2}\int_{\re s_1 = c_1}\int_{\re s_2 = c_2} Mf_1(s_1)Mf_2(s_2) \\
&\sum_{w_1,w_2 \in W} \frac{e^{\langle C, w_1\tilde\lambda_1+w_2\lambda_2 \rangle}}{\prod_{\alpha \in \Delta}\langle \alpha^\vee, w_1\tilde\lambda_1 + w_2\lambda_2 \rangle} M(w_1,\tilde\lambda_1)M(w_2,\lambda_2) ds_2 ds_1.\notag
\end{align}
\end{proposition}
\begin{proof}
First we claim that, for any maximal parabolic $P$ and $f:\R_{\geq 0} \rightarrow \R$ smooth,
\begin{equation*}
\Lambda^CE_{P,f}(g) = \frac{1}{2\pi i} \int_{\re s=c} Mf(s)\Lambda^CE_P(s,g)ds.
\end{equation*}
By \eqref{eq:arthur_t}, $\Lambda^CE_{P,f}(g)$ equals
\begin{equation*}
\sum_{P \supseteq P_0} (-1)^{r(P)} \sum_{\delta \in \Gamma \cap P \backslash \Gamma}(E_f)_P(\delta g)\widehat\tau_P(H_P(\delta g) - C),
\end{equation*}
where
\begin{equation*}
(E_f)_P(g) = \int_{\Gamma \cap N_P \backslash N_P} E_f(ng)dn = \frac{1}{2\pi i}\int_{\Gamma \cap N_P \backslash N_P}\int_{\re s = c} Mf(s)E_P(s,ng)dsdn.
\end{equation*}
In the last expression, the order of the two integrals can be exchanged, since $\Gamma \cap N_P \backslash N_P$ is compact. Also, the foregoing summatory expression for $\Lambda^CE_f(g)$ is a finite sum for each $g$, so the integral over $s$ can be put before those sums as well. This proves the claim.

We now expand the square in \eqref{eq:goal} and apply the claim just proved, to rewrite it as
\begin{equation*}
\int_{\Gamma \backslash G} \frac{1}{(2\pi i)^2}\int_{\re s_1 = c_1}\int_{\re s_2 = c_2} Mf_1(s_1)Mf_2(s_2) \Lambda^CE_{P_1}(s_1,g)\Lambda^CE_{P_2}(s_2,g)ds_1ds_2dg.
\end{equation*}
At this point, we would like to ``perturb'' $E_{P_1}(s_1,g)$ into $E_B(\tilde\lambda_1,g)$. Clearly
\begin{equation*}
\Lambda^CE_{P_1}(s_1,g) = \Lambda^CE_B(\lambda_1,g) = \lim_{\varepsilon\rightarrow0}\Lambda^CE_B(\tilde\lambda_1,g),
\end{equation*}
so \eqref{eq:goal} equals
\begin{equation*}
\int_{\Gamma \backslash G} \frac{1}{(2\pi i)^2}\int_{\re s_1 = c_1}\int_{\re s_2 = c_2} \lim_{\varepsilon\rightarrow0}Mf_1(s_1)Mf_2(s_2) \Lambda^CE_B(\tilde\lambda_1,g)\Lambda^CE_B(\lambda_2,g)ds_1ds_2dg.
\end{equation*}

By the remark following the proof of Theorem \ref{thm:truncate}, both $\Lambda^CE_B(\tilde\lambda_1,g)$ and $\Lambda^CE_B(\lambda_2,g)$ outside the support of $\chi_{G,G}^C$ is bounded by $T^{-p}$ for any choice of $p>0$, times a polynomial in $s$ and $\varepsilon$. Therefore, for any $\varepsilon \in \R$ lying in a small neighborhood of zero, the integrand
\begin{equation*}
Mf_1(s_1)Mf_2(s_2) \Lambda^CE_B(\tilde\lambda_1,g)\Lambda^CE_B(\lambda_2,g)
\end{equation*}
is dominated by a function, independent of $\varepsilon$, whose integral over $\im s_1,\im s_2,$ and $g$ converges absolutely. Hence, we can exchange the order of the limit and the integral over $\Gamma\backslash G$. The Maass-Selberg relation then completes the proof.
\end{proof}

\subsection{Outline of the iterated residue calculus} \label{sec:outline}

Let us perform an iterated residue calculus on \eqref{eq:theexp'} --- for a small value of $\varepsilon \neq 0$, not necessarily positive, without the limit sign. Fix the parameters $0 < \eta_1 < \eta_2 < 1$ --- our intention is to take $\eta_1 \approx 0$ and $\eta_2 \approx 1$ --- and adjust $\varepsilon$ so that it is much smaller than $\eta_2(1-\eta_2)^2$.  We start with $c_1 = n + 2\eta_2-\eta_2^2$ and $c_2 = n + \eta_2(1-\eta_2)^2$. Next, we shift $c_2$ to $n-\eta_2$, and then and $c_1$ to $n-\eta_1$, picking up all the residues arising in the process. As in Section \ref{sec:l1}, Lemma \ref{lemma:misc_growth} justifies the contour shift.

Write $\tilde\lambda_2 = \lambda_2$ for convenience, even though we are not perturbing $\lambda_2$.
In Section \ref{sec:residue}, we will identify all poles of the integrand
\begin{equation} \label{eq:integrand}
Mf_1(s_1)Mf_2(s_2)\sum_{w_1,w_2 \in W} \frac{e^{\langle C, w_1\tilde\lambda_1+w_2\tilde\lambda_2 \rangle}}{\prod_{\alpha \in \Delta}\langle \alpha^\vee, w_1\tilde\lambda_1 + w_2\tilde\lambda_2 \rangle} M(w_1,\tilde\lambda_1)M(w_2,\tilde\lambda_2)
\end{equation}
on the strip $\re s_1 \in [n-\eta_1,n+2\eta_2-\eta_2^2]$ and $\re s_2 \in [n-\eta_2,n+\eta_2(1-\eta_2)^2]$, and evaluate the residues there. It turns out all the iterated residues occur at $s_2=n$ and $s_1=n+O(\varepsilon)$, and as $\varepsilon \rightarrow 0$, the sum of all residues will be shown to converge to
\begin{equation} \label{eq:residue}
Mf_1(n)Mf_2(n)\xi(n,k_1)\xi(n,k_2)\vol^C(\Gamma\backslash G).
\end{equation}
The $s_1$-residue of the integral of \eqref{eq:integrand} on $\re s_2=n-1/4$ will also be considered: it will be shown to be of size $O_{n,f_1,f_2,\eta_2}(T^{-\eta_2 k_2(n-k_2)/2})$ as $\varepsilon \rightarrow 0$.

The limit as $\varepsilon \rightarrow 0$ of the integral that remains after all the residues are picked up --- that is, the integral of \eqref{eq:integrand} on $\re s_1=n-\eta_1$ and $\re s_2=n-\eta_2$ --- will be estimated in Section \ref{sec:denom}.
The main obstacle is the possibility that
$ \langle \alpha^\vee_j,w_1\tilde\lambda_1+w_2\tilde\lambda_2\rangle = O(\varepsilon)$ for some $j$. We will handle such singularities appropriately, and prove that the limit of the integral as $\varepsilon \rightarrow 0$ converges to $O_{n,f_1,f_2,\eta_1,\eta_2}(T^\kappa)$ for an explicit $\kappa = \kappa(n,k_1,k_2,\eta_1,\eta_2)$.

These arguments constitute the proof the following theorem, that concludes our study of \eqref{eq:goal}.
\begin{theorem} \label{thm:res_conc}
Fix $0 < \eta_1 < \eta_2 < 1$, and let $f_1,f_2:\R_{\geq 0} \rightarrow \R$ be smooth functions with compact support. Then
\begin{align*}
&\int_{\Gamma\backslash G} \Lambda^C E_{P_1,f_1}(g)E_{P_2,f_2}(g) dg \\ 
&= Mf_1(n)Mf_2(n)\xi(n,k_1)\xi(n,k_2)\vol^C(\Gamma\backslash G)+O\left(T^\kappa\right),
\end{align*}
where the implicit constant depends on $n,f_1,f_2,\eta_1,\eta_2$, and $\kappa \leq n(n^2-1)/6$ can be written explicitly as a function of $n,k_1,k_2,\eta_1,\eta_2$. More precisely, the error term is the sum of \eqref{eq:res-half} and \eqref{eq:res-remain} below in the $\varepsilon\rightarrow0$ limit.
\end{theorem}



Theorems \ref{thm:truncate}, \ref{thm:res_conc} and Proposition \ref{prop:tr_fund_vol} immediately imply the following harshly truncated correlation formula.
\begin{theorem} \label{thm:wanted}
Continue with the notations of Theorem \ref{thm:res_conc}. Then for any $p>0$
\begin{align*}
&\int_{\Gamma\backslash G} \mathrm H^CE_{P_1,f_1}(g)E_{P_2,f_2}(g) dg \\
&= Mf_1(n)Mf_2(n)\xi(n,k_1)\xi(n,k_2)\vol^C(\Gamma\backslash G) +O\left(T^\kappa\right).
\end{align*}
The implicit constants depend on $n,f_1,f_2,\eta_1,\eta_2$, and $p$. More precisely, the error term is that of Theorem \ref{thm:res_conc} plus the product of the size bounds on $\Lambda^CE_{P_1,f_1}(g)$ and $\Lambda^CE_{P_2,f_2}(g)$ outside the truncated domain given by Theorem \ref{thm:truncate}.
\end{theorem}

\ignore{
Because of the way $\lambda_i$ is defined, $M(w_i,\lambda_i)$ possibly has a pole (as a function of $s_i$) only at $s_i = 2, 3, \ldots, n$. The poles arising from the zeros of the expression
\begin{equation} \label{eq:denom}
\prod_{\alpha \in \Delta}\langle \alpha^\vee, w_1\lambda_1 + w_2\lambda_2 \rangle
\end{equation}
require more consideration. 

What we are concerned with is the scenario in which an individual summand
\begin{equation}\label{eq:ind_term}
\frac{e^{\langle C, w_1\lambda_1+w_2\lambda_2 \rangle}}{\prod_{\alpha \in \Delta}\langle \alpha^\vee, w_1\lambda_1 + w_2\lambda_2 \rangle} M(w_1,\lambda_1)M(w_2,\lambda_2)
\end{equation}
of \eqref{eq:theexp} is singular for infinitely many --- possibly all --- values of $s_1$ and $s_2$. In this case, we need to take care in interpreting \eqref{eq:ind_term} and ultimately \eqref{eq:theexp} correctly. There are two ways in which the singularity can be removed. First, for a zero of the denominator \eqref{eq:denom}, there may be a zero of $M(w_i,\lambda_i)$ that ``cancels it out.'' Second, the singularity may ``cancel out'' with one or more other terms of the form \eqref{eq:ind_term} with different choices of $w_1,w_2 \in W$. It turns out that we need to combine both techniques.


\vspace{4mm}

Let us perceive $\lambda_i$ as a concatenation of two blocks, one consisting of the first $n-k$ entries and the other of the last $k$ entries; it is marked by a semicolon in \eqref{eq:lambda}. Take $\alpha \in \Delta$ for which $\langle \alpha^\vee, w_1\lambda_1 + w_2\lambda_2 \rangle = 0$ for at least one choice of $s_1$ and $s_2$. If $\alpha = \alpha_j$, this means
\begin{equation*}
(\lambda_1)_{w_1^{-1}(j)} + (\lambda_2)_{w_2^{-1}(j)} - (\lambda_1)_{w_1^{-1}(j+1)} - (\lambda_2)_{w_2^{-1}(j+1)} = 0.
\end{equation*}
We say that this zero of \eqref{eq:denom} is \emph{independent of $s_i$} if $w_i^{-1}(j)$ and $w_i^{-1}(j+1)$ belong to the same block, and say it is \emph{dependent on $s_i$} if otherwise. We use this notion to classify the zeros into different cases. Before the $s_2$-residues are taken, there are four different possibilities:
\begin{enumerate}[(i)]
\item The zero is independent of both $s_1$ and $s_2$. It will be studied in the next section.

\item The zero is independent of $s_1$ but dependent on $s_2$. Such a zero is accounted for upon taking the $s_2$-residue and hence causes no problem.

\item The zero is dependent on $s_1$ but independent of $s_2$. This will be taken care of when taking the $s_1$-residue.

\item The zero is dependent on both $s_1$ and $s_2$. Because of the way $c_1$ and $c_2$ move that we set up earlier, this case will never happen.
\end{enumerate}
After the $s_2$-residues are taken, there are two possibilities:
\begin{enumerate}
\item[(v)] The zero is independent of $s_1$. The possibility for such a zero is already removed by (i) and (ii) above.

\item[(vi)] The zero is dependent on $s_1$. It will simply be picked up by the $s_1$-residue computation.
\end{enumerate}

Thus we conclude that the singularities of \eqref{eq:ind_term} are either removable or to be picked up by the residue computation. By inspecting \eqref{eq:lambda} and \eqref{eq:ind_term}, it is clear that there exists only finitely many $s_2$-residues, and also only finitely many $s_1$-residues for each of those $s_2$-residues.


}

\subsection{Residue computation} \label{sec:residue}

We continue with the settings of the previous sections. Let us start by considering the $s_2$-residues of \eqref{eq:integrand} that would appear as $c_2$ moves from $n+\eta_2(1-\eta_2)^2$ to $n-\eta_2$, and $c_1$ stays at $n+2\eta_2-\eta_2^2$. $Mf_1,Mf_2,$ and $e^{\langle C,w_1\tilde\lambda_1+w_2\tilde\lambda_2\rangle}$ are holomorphic, so no poles arise from those. From the shape of $\tilde\lambda_2$, one sees that $M(w_2,\tilde\lambda_2)$ has a pole only if $s_2=n$. The denominator terms $\langle \alpha^\vee,w_1\tilde\lambda_1+w_2\tilde\lambda_2\rangle$ may vanish and become poles of \eqref{eq:integrand}, which we classify into two types:
\begin{itemize}
\item Type I: If $\langle \alpha^\vee,w_1\tilde\lambda_1\rangle$ is independent of $s_1$, then $\langle \alpha^\vee,w_1\tilde\lambda_1+w_2\tilde\lambda_2\rangle$ may have a zero at $s_2=n+a\varepsilon$ for some $a \in \{\pm1,\ldots,\pm(n-1)\}$.

\item Type II: If $\langle \alpha^\vee,w_1\tilde\lambda_1\rangle$ depends on $s_1$, $\langle \alpha^\vee,w_1\tilde\lambda_1+w_2\tilde\lambda_2\rangle$ may have a zero when $\re s_1 \pm \re s_2$ is an integer. Our choices for the ranges of $c_1,c_2$ above were contrived so that this happens only if $\eta_2 \geq (1-\eta_2)^2$ and $\re s_2 = n-(1-\eta_2)^2 +a\varepsilon$ for some $a \in \{\pm1,\ldots,\pm(n-1)\}$.
\end{itemize}


\subsubsection*{$s_2$-residues from the denominator}

It turns out that the residues of both Type I and Type II cancel out with one another and vanish. The argument is delegated to Section \ref{sec:poles} below.

\subsubsection*{$s_2$-residue from $M(w_2,\tilde\lambda_2)$}

Consider the $s_2$-residue arising from a pole of $M(w_2,\tilde\lambda_2)$ at $s_2=n$. For this pole to occur, it must be the case that $w_2(n)<w_2(1)$, and moreover that
\begin{align} \label{eq:w*}
w_2(n-k_2+1) &< w_2(n-k_2+2) < \ldots < w_2(n) \\
&< w_2(1) < w_2(2) < \ldots < w_2(n-k_2), \notag
\end{align}
since otherwise $M(w_2,\tilde\lambda_2)=0$, which would cancel the pole. Thus, the only possibility for $w_2$ is the one that swaps the position of the two ``blocks'' --- demarcated by the semicolon in \eqref{eq:lambda} --- so that
\begin{equation*}
w_2\lambda_2(n) \equiv (-n-k_2,\ldots,-n-1;-n,\ldots,-k_2-1) \equiv -\rho.
\end{equation*}
With this, one computes that
\begin{equation*}
\mathrm{Res}_{s_2=n}M(w_2,\lambda_2) = \xi(n,k_2),
\end{equation*}
so the $s_2$-residue of \eqref{eq:integrand} at $s_2 = n$ is equal to $2\pi i$ times
\begin{equation}\label{eq:res-s2}
Mf_1(s_1)Mf_2(n)\xi(n,k_2)\sum_{w_1 \in W} \frac{e^{\langle C, w_1\tilde\lambda_1 - \rho\rangle}}{\prod_{\alpha \in \Delta}\langle \alpha^\vee,w_1\tilde\lambda_1 - \rho \rangle }M(w_1,\tilde\lambda_1).
\end{equation}

\subsubsection*{$s_1$-residue of the $s_2$-residue}

The $s_1$-residue of \eqref{eq:res-s2} is handled already by Proposition \ref{prop:chicken}. This establishes the main contribution to Theorem \ref{thm:res_conc}.

\subsubsection*{$s_1$-residue of the ``$s_2$-remainder''}

We must also consider the $s_1$-residues that may occur as $c_2$ stays at $n-\eta_2$, and $c_1$ moves from $n+2\eta_2-\eta_2^2$ to $n-\eta_1$. There are the following possible spots for a nontrivial $s_1$-residue: 
\begin{enumerate}[(i)]
\item $\re s_1 = n+\eta_2+O(\varepsilon)$,
\item $\re s_1 = n-\eta_2+O(\varepsilon)$,
\item $\re s_1 =n+O(\varepsilon)$.
\end{enumerate}

Case (i) is in fact impossible: otherwise, it must hold that $\langle \alpha^\vee_j, w_1\tilde\lambda_1 + w_2\tilde\lambda_2 \rangle = s_1+s_2+b+O(\varepsilon)$ for some integer $b$; but then, $c_1+c_2 = 2n$ yet $|b|$ cannot exceed $2(n-1)$. Case (ii) is also impossible, since $\eta_2-\eta_1$ is much greater than the vanishing quantity $\varepsilon$. In case (iii), by repeating the argument for the $s_2$-residue computation earlier, the ensuing $s_1$-residue can be seen to be equal to
\begin{align} \label{eq:res-half}
&O(\varepsilon)+\frac{1}{2\pi i}\left(Mf_1(n)\xi(n,k_1) + O(\varepsilon)\right) \\
&\cdot\int_{\re s_2=n-\eta_2} Mf_2(s_2)\sum_{w_2 \in W}\frac{e^{\langle C, w_2\lambda_2-\rho+O(\varepsilon)\rangle}}{\prod_{\alpha \in \Delta}\langle\alpha^\vee,w_2\lambda_2-\rho+O(\varepsilon)\rangle}M(w_2,\lambda_2)ds_2. \notag
\end{align}
The denominator terms here yield no poles; neither can they be $O(\varepsilon)$, since it requires that $\langle\alpha_j^\vee,w_2\lambda_2\rangle=1$ for some $j$, which causes $c(-1)=0$ to appear as a factor of $M(w_2,\lambda_2)$. The $T$-degree can be analyzed as in the proof of Theorem \ref{thm:L1-norm}, with the maximum $T$-degree coming from $w_2 \in W$ such that
\begin{equation*}
w_2\lambda_2(n-\eta_2) = \rho + (0,\ldots,0;\underbrace{\eta_2,\ldots,\eta_2}_{\mbox{\tiny $k_2$ times}}).
\end{equation*}
From this, it can be computed that \eqref{eq:res-half} in the $\varepsilon\rightarrow0$ limit is bounded by $O_{n,f_1,f_2,\eta_2}(T^{-\eta_2 k_2(n-k_2)/2})$.

\subsection{The remaining integral} \label{sec:denom}

We conclude the estimate of \eqref{eq:goal} by proving that the integral
\begin{align} \label{eq:res-remain}
\frac{1}{(2\pi i)^2}&\int_{\re s_1 = c_1}\int_{\re s_2 = c_2} Mf_1(s_1)Mf_2(s_2) \\
&\sum_{w_1,w_2 \in W} \frac{e^{\langle C, w_1\tilde\lambda_1+w_2\tilde\lambda_2 \rangle}}{\prod_{\alpha \in \Delta}\langle \alpha^\vee, w_1\tilde\lambda_1 + w_2\tilde\lambda_2 \rangle} M(w_1,\tilde\lambda_1)M(w_2,\tilde\lambda_2) ds_2 ds_1 \notag
\end{align}
converges, where $c_1= n-\eta_1, c_2=n-\eta_2$. Since $M(w_i,\tilde\lambda_i)$ cannot diverge by our choice of $c_1,c_2$, the only obstacle to its convergence is the possibility that $\langle \alpha^\vee_j, w_1\tilde\lambda_1 + w_2\tilde\lambda_2 \rangle$ equals a multiple of $\varepsilon$ for some $j=1,\ldots,n-1$. For this to happen, it is necessary that $\langle \alpha^\vee_j, w_1\tilde\lambda_1 + w_2\tilde\lambda_2 \rangle$ is independent of both $s_1$ and $s_2$; otherwise, its real part is away from zero by at least $\min(\eta_1,1-\eta_2,\eta_2-\eta_1)$.

Suppose first that $w_2^{-1}(j) + 1 \neq w_2^{-1}(j+1)$.
Then there exists $j' \neq n-k_2$ such that $w_2(j') > w_2(j'+1)$, implying that $M(w_2,\tilde\lambda_2)=0$, since $c(\langle \alpha^\vee_{j'},\tilde\lambda_2\rangle) = c(-1) = 0$ appears as its factor. Thus the summand under question is zero, regardless of $\varepsilon$.

Next suppose that $w_2^{-1}(j) + 1 = w_2^{-1}(j+1)$. Then we must have $w_1^{-1}(j+1) +1 = w_1^{-1}(j)$; in other words, $w_1$ flips the sign of $\alpha_{w_1^{-1}(j+1)}$. Hence $M(w_1,\tilde\lambda_1)$ begets the factor $c(\langle \alpha^\vee_{w_1^{-1}(j+1)},\tilde\lambda_1\rangle) = c(-1+O(\varepsilon)) = O(\varepsilon)$, which balances out the growth of $\langle \alpha^\vee_j, w_1\tilde\lambda_1 + w_2\tilde\lambda_2 \rangle$ as $\varepsilon \rightarrow 0$. This proves that for every $j$ such that $\langle \alpha^\vee_j, w_1\tilde\lambda_1 + w_2\tilde\lambda_2 \rangle = O(\varepsilon)$, there corresponds a factor of $M(w_1,\tilde\lambda_1)$ that cancels it out. This proves that \eqref{eq:res-remain} converges as $\varepsilon \rightarrow 0$.

As for the size of \eqref{eq:res-remain} in terms of $T$, first off we know that it increases with $T$ in general, by \cite[Corollary to Theorem 4]{Kim22}. (Of course, the cases $k_1=k_2=1$ and $k_1=k_2=n-1$ are well-known exceptions.) The trivial upper bound would be $O(T^{\langle \rho^\vee,2\rho\rangle}) = O(T^{n(n^2-1)/6})$. The exponent can be slightly improved by observing that, if $w_2(j) > w_2(j+1)$ for $j\neq n-k_2$, then the corresponding term in \eqref{eq:res-remain} equals zero. Taking this into account, the $w_2$ yielding the maximum $T$-degree is in fact the trivial permutation, so that
\begin{equation*}
w_2\tilde\lambda_2(n-\eta_2) = -\rho + (\underbrace{n,\ldots,n}_{n-k_2};\underbrace{\eta_2,\ldots,\eta_2}_{k_2}).
\end{equation*}
On the other hand, the candidate $w_1$ yielding the maximum $T$-degree is
\begin{equation*}
w_1\tilde\lambda_1(n-\eta_1) = \rho + (0,\ldots,0;\underbrace{\eta_1,\ldots,\eta_1}_{k_1}).
\end{equation*}
With these choices of $w_1,w_2$, it may be said that \eqref{eq:res-remain} is of order $O(T^{\kappa})$, where
\begin{align*}
\kappa := \langle \rho^\vee, w_1\tilde\lambda_1(n-\eta_1)+w_2\tilde\lambda_2(n-\eta_2)\rangle.
\end{align*}
For instance, when $k_1=k_2=k$, we can compute
\begin{align*}
\kappa &= \frac{k(n-k)}{2}(n-\eta_1-\eta_2) < \frac{n^3}{8},
\end{align*}
which is indeed smaller than $n(n^2-1)/6$ for all $n \geq 3$ and $1 \leq k \leq n-1$. This completes the proof that \eqref{eq:res-remain} contributes $O_{n,f_1,f_2,\eta_1,\eta_2}\left(T^\kappa\right)$.

\section{The cancellation of the poles} \label{sec:poles}

\ignore{
\subsection{Some easy opening}

The goal is to correct the error in Sec. 5.3 about the $s_2$-residues not from $M(w_2,\tilde\lambda_2)$. The following proposition provides a helpful starting point.

\begin{proposition} \label{prop:lambda2}
For $w_1,w_2 \in W$, let
\begin{equation} \label{eq:summand}
I(w_1,w_2) = \frac{e^{\langle C, w_1\lambda_1+w_2\lambda_2 \rangle}}{\prod_{\alpha \in \Delta}\langle \alpha^\vee, w_1\lambda_1 + w_2\lambda_2 \rangle} M(w_1,\lambda_1)M(w_2,\lambda_2),
\end{equation}
and similarly for $\tilde I(w_1,w_2)$ with $\lambda_i$ replaced with $\tilde\lambda_i$.
Then both $I(w_1,w_2)$ and $\tilde I(w_1,w_2)$ are nonzero (as a function of $s_1,s_2$) only if $w_2(i) < w_2(j)$ for all $i<j$ in the same block.
\end{proposition}
\begin{proof}
Suppose otherwise, so that there exists $i<j$ in the same block such that $w_2(i) > w_2(j)$. Then there exists $i \leq l < j$ such that $w_2(l) > w_2(l+1)$ so that $M(w_2,\lambda_2)=M(w_2,\tilde\lambda_2)$ has the factor $c(\langle \alpha^\vee_l, \tilde\lambda_2\rangle)$. But this vanishes, since $\langle \alpha^\vee_l, \tilde\lambda_2\rangle=-1$.
\end{proof}
Henceforth one may assume that $w_2 \in W$ preserves the order within each block. Although we do not cite this fact directly below, it is a useful constraint to keep in mind for following our computations.
}

\subsection{The simplest case}

We continue from the context of the discussion in Section \ref{sec:residue}. Recall that we are considering the $s_2$-residues of of Type I and Type II, where $s_1$ and $\varepsilon \neq 0$ are fixed. We write $z=s_2-n-a\varepsilon$ for a residue of Type I, and $z = s_2-s_1+1-a\varepsilon$ for a residue of Type II.

Let us first discuss the simplest case in which $\langle \alpha_j^\vee,w_1\tilde\lambda_1+w_2\tilde\lambda_2\rangle = \pm z$ for precisely one value of $j$. Write
\begin{equation} \label{eq:summand}
\tilde I(w_1,w_2) = \frac{e^{\langle C, w_1\tilde\lambda_1+w_2\tilde\lambda_2 \rangle}}{\prod_{\alpha \in \Delta}\langle \alpha^\vee, w_1\tilde\lambda_1 + w_2\tilde\lambda_2 \rangle} M(w_1,\tilde\lambda_1)M(w_2,\tilde\lambda_2),
\end{equation}
and $w_i' = (j\ j+1)w_i$. We will show that the poles of $\tilde I(w_1,w_2)$ and $\tilde I(w_1',w_2')$ at $z=0$ add up to zero.

Let us examine how each factor of $\tilde I(w_1,w_2)$ changes upon replacing $w_i$ by $w_i'$. Clearly
\begin{equation} \label{eq:sign_change}
\langle \alpha^\vee_j, w'_1\tilde\lambda_1+w'_2\tilde\lambda_2\rangle = - \langle \alpha^\vee_j, w_1\tilde\lambda_1+w_2\tilde\lambda_2\rangle .
\end{equation}
Also,
\begin{equation*}
\langle \alpha^\vee_{j\pm1}, w'_1\tilde\lambda_1+w'_2\tilde\lambda_2\rangle = \langle \alpha^\vee_{j\pm1} + \alpha^\vee_j, w_1\tilde\lambda_1+w_2\tilde\lambda_2\rangle,
\end{equation*}
and other $\langle \alpha^\vee, w_1\tilde\lambda_1+w_2\tilde\lambda_2\rangle$ remain unchanged. Using these, one can also compute that
\begin{equation*}
e^{\langle C, w'_1\tilde\lambda_1+w'_2\tilde\lambda_2\rangle} = e^{\langle C, w_1\tilde\lambda_1+w_2\tilde\lambda_2\rangle}\cdot T^{-\langle \alpha^\vee_{j}, w_1\tilde\lambda_1+w_2\tilde\lambda_2\rangle}.
\end{equation*}
On the other hand, from the functional equation \eqref{eq:fnleq_M} of the intertwining operator $M$, we have
\begin{equation*}
M(w_1',\tilde\lambda_1)M(w_2',\tilde\lambda_2) = c(\langle\alpha^\vee_j,w_1\tilde\lambda_1\rangle)c(\langle\alpha^\vee_j,w_2\tilde\lambda_2\rangle)M(w_1,\tilde\lambda_1)M(w_2,\tilde\lambda_2).
\end{equation*}

For $z$ near zero, one observes that
\begin{align*}
c(\langle \alpha^\vee_{j},w_1\tilde\lambda_1\rangle)c(\langle \alpha^\vee_{j},w_2\tilde\lambda_2\rangle) &= 1+O(z),\\
T^{\langle \alpha^\vee_{j},w_1\tilde\lambda_1+w_2\tilde\lambda_2\rangle} &= 1+O(z), \\
\frac{\langle \alpha^\vee_{j\pm1},w_1\tilde\lambda_1+w_2\tilde\lambda_2\rangle}{\langle \alpha^\vee_{j}+\alpha^\vee_{j\pm1}, w_1\tilde\lambda_1+w_2\tilde\lambda_2\rangle} &= \sum_{p=0}^\infty\left(\frac{\pm z}{\langle \alpha^\vee_{j\pm1},w_1\tilde\lambda_1+w_2\tilde\lambda_2\rangle}\right)^p = 1+O(z).
\end{align*}
Therefore every individual factor of $\tilde I(w'_1,w'_2)$, with the exception of \eqref{eq:sign_change}, differs from the corresponding factor of $\tilde I(w_1,w_2)$ by a multiplicative factor of size $1+O(z)$.

It follows that
\begin{equation*}
\tilde I(w_1,w_2) + \tilde I(w_1',w_2') = O(z\tilde I(w_1,w_2)),
\end{equation*}
which converges to a constant as $z \rightarrow 0$. This implies that the poles of $\tilde I(w_1,w_2)$ and $\tilde I(w_1',w_2')$ at $z=0$ indeed cancel with each other, as claimed.

\subsection{Multiple poles}

We proceed to consider the general case, in which $\langle \alpha_j^\vee,w_1\tilde\lambda_1+w_2\tilde\lambda_2\rangle = \pm z$ for more than one value of $j$. We first observe that it cannot be the case for two consecutive values, say $j$ and $j+1$. If this were the case, then $\langle \alpha^\vee_j,w_1\tilde\lambda_1+w_2\tilde\lambda_2\rangle = \pm z$ and $\langle \alpha^\vee_{j+1},w_1\tilde\lambda_1+w_2\tilde\lambda_2\rangle =\mp z$; however, this implies that $(w_1\tilde\lambda_1)_j=(w_1\tilde\lambda_1)_{j+2}$, which is impossible. Thus if we write $j_1 < \ldots < j_r$ for all the indices such that $\langle \alpha_{j_i}^\vee,w_1\tilde\lambda_1+w_2\tilde\lambda_2\rangle = \pm z$, then $j_{i}-j_{i-1} \geq 2$.

Write $\tau_i = (j_i \ j_i+1)$, and for $e \in \{0,1\}^r$, write $\tau^e = \prod_i \tau_i^{e_i}$. We will show that 
\begin{equation*}
\sum_{e \in \{0,1\}^r} \tilde I(\tau^{e}w_1,\tau^{e}w_2) = O(z^r\tilde I(w_1,w_2)),
\end{equation*}
which immediately implies that it has zero residue at $z=0$, or equivalently $s_2=n+a\varepsilon$.

To this end, let us introduce a few additional notations. We write
\begin{align*}
\mu_{j_i} &= c(\langle \alpha^\vee_{j_i},w_1\tilde\lambda_1\rangle)c(\langle \alpha^\vee_{j_i},w_2\tilde\lambda_2\rangle)\cdot T^{-\langle \alpha^\vee_{j_i},w_1\tilde\lambda_1+w_2\tilde\lambda_2\rangle}, \\
\underline{\Delta}_{j_i-1} &= \frac{\langle \alpha^\vee_{j_i-1},w_1\tilde\lambda_1+w_2\tilde\lambda_2\rangle }{\langle \alpha^\vee_{j_i-1}+\alpha^\vee_{j_i}, w_1\tilde\lambda_1+w_2\tilde\lambda_2\rangle} \cdot {T^{\frac{1}{2}(j_i-1)(n-j_i+1)\langle \alpha^\vee_{j_i},w_1\tilde\lambda_1+w_2\tilde\lambda_2\rangle}}, \\
\bar\Delta_{j_i+1} &= \frac{\langle \alpha^\vee_{j_i+1},w_1\tilde\lambda_1+w_2\tilde\lambda_2\rangle }{\langle \alpha^\vee_{j_i}+\alpha^\vee_{j_i+1}, w_1\tilde\lambda_1+ w_2\tilde\lambda_2\rangle}\cdot {T^{\frac{1}{2}(j_i+1)(n-j_i-1)\langle \alpha^\vee_{j_i},w_1\tilde\lambda_1+w_2\tilde\lambda_2\rangle}},
\end{align*}
and $\mathcal E_i = \mu_{j_i}\underline{\Delta}_{j_i-1}\overline\Delta_{j_i+1}$. Arguments from the previous section show that $\mu_{j_i},\underline{\Delta}_{j_i-1},\overline\Delta_{j_i+1}$, and $\mathcal E_i$ are all terms of size $1+O(z)$, and that
\begin{equation*}
\tilde I(\tau_iw_1,\tau_iw_2) = -\tilde I(w_1,w_2) \cdot \mathcal E_i.
\end{equation*}

Informally speaking, $\underline{\Delta}_{j_i-1}$ account for the change in the ``($j_i-1$)-st term'' of $\tilde I(w_1,w_2)$ caused by composing $w_1$ and $w_2$ by $\tau_i$; similarly for $\overline\Delta_{j_i+1}$ and $\mu_{j_i}$.

\subsubsection{All zeros are apart}

The easiest case to handle is in which $j_{i+1}-j_i > 2$ for all $i=1,\ldots,r-1$. Then we have
\begin{equation*}
\tilde I(\tau^ew_1,\tau^ew_2) = \tilde I(w_1,w_2)\prod_{i:e_i\neq 0} (-\mathcal E_i),
\end{equation*}
that is, the errors are multiplicative. This implies that 
\begin{equation*}
\sum_{e \in \{0,1\}^r} \tilde I(\tau^{e}w_1,\tau^{e}w_2)=\tilde I(w_1,w_2)\prod_{i=1}^r(1-\mathcal E_i) = O(z^r\tilde I(w_1,w_2)),
\end{equation*}
which is $O(1)$ as $z \rightarrow 0$, showing that the residues indeed cancel to zero.

\subsubsection{All zeros are together}

We next consider the polar opposite case in which $j_{i+1}-j_i=2$ for all $i=1,\ldots,r-1$. The errors are no longer strictly multiplicative, which comes from the fact that both $\tau_i$ and $\tau_{i+1}$ alter $\langle \alpha^\vee_{j_i+1},w_1\tilde\lambda_1+w_2\tilde\lambda_2\rangle$. Still, it turns out that an analogous principle continues to hold.

For $e \in \{0,1\}^r$, we can write
\begin{equation*}
\tilde I(\tau^ew_1,\tau^ew_2) = (-1)^{|\{i:e_i=1\}|} \tilde I(w_1,w_2) \cdot \mathcal E_e,
\end{equation*}
where
\begin{align*}
\mathcal E_e = \prod_{\ell=j_1-1}^{j_r+1} \mathcal L_\ell,
\end{align*} 
and
\begin{align*}
\mathcal L_\ell = \mathcal L_{e,\ell} = \begin{cases}
\mu_{j_i} & \mbox{if $\ell = j_i$ and $e_i=1$} \\
\underline{\Delta}_{j_i-1} & \mbox{if $\ell = j_i-1$, and $e_i=1$ and $e_{i-1}$=0} \\
\overline {\Delta}_{j_i+1} & \mbox{if $\ell = j_i+1$, and $e_i=1$ and $e_{i+1} = 0$} \\
\underline{\overline{\Delta}}_{j_i+1} & \mbox{if $\ell = j_i+1$, and $e_i=1$ and $e_{i+1} = 1$.}
\end{cases}
\end{align*}
Here we understand $e_0=e_{r+1}=0$ always, $\mu_{j_i}, \underline{\Delta}_{j_i-1},\overline {\Delta}_{j_i+1}$ are as earlier, and
\begin{align*}
\underline{\overline{\Delta}}_{j_i+1} = \frac{\langle \alpha^\vee_{j_i+1}, w_1\tilde\lambda_1+w_2\tilde\lambda_2 \rangle}{\langle \alpha^\vee_{j_i}+\alpha^\vee_{j_i+1}+\alpha^\vee_{j_{i+1}}, w_1\tilde\lambda_1+w_2\tilde\lambda_2 \rangle} \cdot T^{\frac{1}{2}(j_i+1)(n-j_i-1)\langle \alpha^\vee_{j_i}+\alpha^\vee_{j_{i+1}}, w_1\tilde\lambda_1+w_2\tilde\lambda_2 \rangle}.
\end{align*}

\begin{lemma} \label{lemma:approx_cancel}
$\underline{\overline{\Delta}}_{j_i+1} = \overline{\Delta}_{j_i+1}\underline{\Delta}_{j_{i+1}-1} + O(z^2).$
\end{lemma}
\begin{proof}
Suppose first that the signs of $\langle \alpha^\vee_{j_i},w_1\tilde\lambda_1+w_2\tilde\lambda_2\rangle$ and $\langle \alpha^\vee_{j_{i+1}},w_1\tilde\lambda_1+w_2\tilde\lambda_2\rangle$ are equal; without loss of generality, we assume both are equal to $z$.
One then observes that
\begin{align*}
\underline{\overline{\Delta}}_{j_i+1} - \overline{\Delta}_{j_i+1}\underline{\Delta}_{j_{i+1}-1} = T^{(j_i+1)(n-j_i-1)z}\left(\frac{u}{2z+u} - \left(\frac{u}{z+u}\right)^2\right),
\end{align*}
where we temporarily denote
\begin{align*}
u = \langle \alpha^\vee_{j_i+1},w_1\tilde\lambda_1+w_2\tilde\lambda_2\rangle.
\end{align*}
It can be easily seen that
\begin{align*}
\frac{u}{2z+u} - \left(\frac{u}{z+u}\right)^2= O(z^2),
\end{align*}
which is the desired result. In case the signs of $\langle \alpha^\vee_{j_i},w_1\tilde\lambda_1+w_2\tilde\lambda_2\rangle$ and $\langle \alpha^\vee_{j_{i+1}},w_1\tilde\lambda_1+w_2\tilde\lambda_2\rangle$ differ, we have
\begin{align*}
\underline{\overline{\Delta}}_{j_i+1} - \overline{\Delta}_{j_i+1}\underline{\Delta}_{j_{i+1}-1} &= 1-\frac{u}{z+u}\cdot\frac{u}{-z+u} \\
&=O(z^2)
\end{align*}
again, as desired.
\end{proof}

The point of this lemma is that the errors are ``almost multiplicative'' in a certain sense.
We will use this to prove
\begin{equation*}
\sum_{e\in\{0,1\}^r} \tilde I(\tau^ew_1,\tau^ew_2) = \tilde I(w_1,w_2)\sum_{e\in\{0,1\}^r} (-1)^{|\{i:e_i=1\}|} \mathcal E_e = O(z^r\tilde I(w_1,w_2)).
\end{equation*}

The base case $r=1$ is precisely the simplest case proved above.
In general, we write
\begin{align} \label{eq:complicated_sum}
&\sum_{e\in\{0,1\}^r} (-1)^{|\{i:e_i=1\}|} \mathcal E_e \\ \notag
&= \sum_{e\in\{0,1\}^r \atop e_{r-1}=0,e_r=0} (-1)^{|\{i:e_i=1\}|} \prod_{\ell=j_1-1}^{j_{r-2}+1}\mathcal L_\ell \cdot (1-\mu_r\overline{\Delta}_{j_r+1}\underline{\Delta}_{j_r-1}) \\ \notag &\hspace{4mm}+ \sum_{e \in \{0,1\}^r \atop e_{r-1}=1,e_r=0}(-1)^{|\{i:e_i=1\}|} \prod_{\ell=j_1-1}^{j_{r-2}+1}\mathcal L_\ell \cdot(\mu_{r-1}\overline{\Delta}_{j_{r-1}+1} - \mu_{r-1}\mu_r\overline{\underline{\Delta}}_{j_{r-1}+1}\overline{\Delta}_{j_r+1}).
\end{align}
By the induction hypothesis applied to $r-2$,
\begin{equation}\label{eq:complicated_sums}
\sum_{e \in \{0,1\}^r \atop e_{r-1}=e_r=0}(-1)^{|\{i:e_i=1\}|} \prod_{\ell=j_1-1}^{j_{r-2}+1}\mathcal L_\ell, \sum_{e \in \{0,1\}^r \atop e_{r-1}=1, e_r=0}(-1)^{|\{i:e_i=1\}|} \prod_{\ell=j_1-1}^{j_{r-2}+1}\mathcal L_\ell 
\end{equation}
are both of order $O(z^{r-2})$. 
For the first sum, this is clear.
As for the second sum, by two applications of the induction hypothesis,
\begin{align*}
&\sum_{e \in \{0,1\}^r \atop e_{r-2}=0, e_{r-1}=1, e_r=0}(-1)^{|\{i:e_i=1\}|} \prod_{\ell=j_1-1}^{j_{r-2}}\mathcal L_\ell \cdot \underline{\Delta}_{j_{r-2}+1} \\
&\hspace{4mm}+ \sum_{e \in \{0,1\}^r \atop e_{r-2}=1, e_{r-1}=1, e_r=0}(-1)^{|\{i:e_i=1\}|} \prod_{\ell=j_1-1}^{j_{r-2}}\mathcal L_\ell \cdot (\overline\Delta_{j_{r-2}+1}\underline{\Delta}_{j_{r-2}+1}+O(z^2)) \\
&= O(z^{r-2}\underline{\Delta}_{j_{r-2}+1}) + \sum_{e \in \{0,1\}^r \atop e_{r-2}=1, e_{r-1}=1, e_r=0}(-1)^{|\{i:e_i=1\}|} \prod_{\ell=j_1-1}^{j_{r-2}}\mathcal L_\ell \cdot O(z^2) \\
&= O(z^{r-2}) + O(z^{r-1}) = O(z^{r-2})
\end{align*}
for $z$ near zero.

Moreover, recall that the first sum in \eqref{eq:complicated_sums} equals
\begin{align*}
&\sum_{e \in \{0,1\}^r \atop e_{r-2}=e_{r-1}=e_r=0}(-1)^{|\{i:e_i=1\}|} \prod_{\ell=j_1-1}^{j_{r-2}}\mathcal L_\ell + \sum_{e \in \{0,1\}^r \atop e_{r-2}=1, e_{r-1}= e_r=0}(-1)^{|\{i:e_i=1\}|} \prod_{\ell=j_1-1}^{j_{r-2}}\mathcal L_\ell \cdot \overline\Delta_{j_{r-2}+1}.
\end{align*}
It follows by adding this expression to the previous equality that the two sums in \eqref{eq:complicated_sums} add to $O(z^{r-1})$.

Thus \eqref{eq:complicated_sum} becomes
\begin{align*}
&\sum_{e\in\{0,1\}^r \atop e_{r-1}=e_r=0} (-1)^{|\{i:e_i=1\}|} \prod_{\ell=j_1-1}^{j_{r-2}+1}\mathcal L_\ell \cdot (1-\mu_r\overline{\Delta}_{j_r+1}\underline{\Delta}_{j_r-1} -\mu_{r-1}\overline{\Delta}_{j_{r-1}+1} \\
&\hspace{4mm}+\mu_{r-1}\mu_r\overline{\underline{\Delta}}_{j_{r-1}+1}\overline{\Delta}_{j_r+1})
+ O(z^{r-1}(\mu_{r-1}\overline{\Delta}_{j_{r-1}+1} - \mu_{r-1}\mu_r\overline{\underline{\Delta}}_{j_{r-1}+1}\overline{\Delta}_{j_r+1})) \\
&=O(z^{r-2}(1-\mu_r\overline{\Delta}_{j_r+1}\underline{\Delta}_{j_r-1} -\mu_{r-1}\overline{\Delta}_{j_{r-1}+1} + \mu_{r-1}\mu_r\overline{\underline{\Delta}}_{j_{r-1}+1}\overline{\Delta}_{j_r+1})) \\
&\hspace{4mm}+ O(z^{r-1}(\mu_{r-1}\overline{\Delta}_{j_{r-1}+1} - \mu_{r-1}\mu_r\overline{\underline{\Delta}}_{j_{r-1}+1}\overline{\Delta}_{j_r+1})).
\end{align*}
By Lemma \ref{lemma:approx_cancel}, this is
\begin{align*}
&O(z^{r-2}(1-\mu_r\overline{\Delta}_{j_r+1}\underline{\Delta}_{j_r-1})(1-\mu_{r-1}\overline{\Delta}_{j_{r-1}+1})) + O(z^r) \\
&+O(z^{r-1}\mu_{r-1}\overline{\Delta}_{j_{r-1}+1}(1-\mu_{r-1}\underline{\Delta}_{j_r-1}\overline{\Delta}_{j_r+1} ))+O(z^{r+1}) \\
&=O(z^r).
\end{align*}

\subsubsection{Mixed case}

Lastly, consider the case in which there are $p > 1$ ``consecutive zeros,'' each of length $r_\ell$, so that $j_{i+1}-j_i = 2$ if and only if $i \neq r_1 + \ldots + r_{\ell}$ for $\ell=1,\ldots,p-1$. Write
\begin{align*}
\{0,1\}^r = \{0,1\}^{r_1} \times \cdots \times \{0,1\}^{r_p},
\end{align*}
and identify $\{0,1\}^{r_\ell}$ with the $\ell$-th component of this expression, with all other components being zero. By the work of the previous section, for any $e' \in \{0,1\}^r$ whose $\ell$-th component is zero,
\begin{align*}
\sum_{e \in \{0,1\}^{r_\ell}}\tilde I(\tau^{e+e'}w_1,\tau^{e+e'}w_2) = \tilde I(\tau^{e'}w_1,\tau^{e'}w_2)\mathcal F_\ell 
\end{align*}
for some expression $\mathcal F_\ell$ of order $O(z^{r_\ell})$. This implies
\begin{align*}
\sum_{e \in \{0,1\}^{r}}\tilde I(\tau^{e}w_1,\tau^{e}w_2) = \tilde I(w_1,w_2)\mathcal F_1\cdots\mathcal F_p = O(z^r\tilde I(w_1,w_2)),
\end{align*}
as claimed.

\ignore{
We consider a situation in which $I(w_1,w_2)$ is possibly singular, that is, not assigned a finite value.
Due to the shape of $\lambda_2$ and our assumption about $w_2$, $M(w_2,\lambda_2)$ can be singular only when $w_2\lambda_2(n) = -\rho$, which happens if and only if $w_2$ is the unique permutation satisfying \eqref{eq:w*}. Henceforth we assume that $w_2$ is not this element; in other words, $w_2(1)<w_2(n)$ from now on. We assume $\re s_1= n+1/2$ so that $M(w_1,\lambda_1)$ cannot be singular for the time being. This leaves only the factors $\langle \alpha^\vee, w_1\lambda_1 + w_2\lambda_2 \rangle$ in the denominator to consider. There are two possibilities:
\begin{enumerate}[(i)]
\item $\langle \alpha^\vee_j,w_1\lambda_1+w_2\lambda_2\rangle=0$ independently of $s_1$ or $s_2$.
\item $\langle \alpha^\vee_j,w_1\lambda_1+w_2\lambda_2\rangle=\pm(s_2-n)$. 
\end{enumerate}
(i) is easier to handle. In that case, equivalently if $\langle \alpha^\vee_j,w_1\tilde\lambda_1+w_2\tilde\lambda_2\rangle=O(\varepsilon)$, then by the assumption on $w_2$ it is forced that $\langle \alpha^\vee_j,w_2\tilde\lambda_2\rangle = -1$, and so $\langle \alpha^\vee_j,w_1\tilde\lambda_1\rangle = 1+O(\varepsilon)$. This means that $w_1$ swaps the order of $w_1^{-1}(j+1)$ and $w_1^{-1}(j)=w_1^{-1}(j+1)+1$, which causes $c(\langle \alpha^\vee_{w_1^{-1}(j+1)},\tilde\lambda_1\rangle) = O(\varepsilon)$ to appear as a factor of $M(w_1,\tilde\lambda_1)$. As $\varepsilon \rightarrow 0$, the factor
\begin{equation*}
\frac{c(\langle \alpha^\vee_{w_1^{-1}(j+1)},\tilde\lambda_1\rangle)}{\langle\alpha^\vee_j,w_1\tilde\lambda_1+w_2\tilde\lambda_2\rangle}
\end{equation*}
converges to a finite number, that is, they ``cancel each other out.'' Note the clear correspondence between these factors.

Case (ii) is much harder, and will take up most of this section. Indeed, it is possible for $\tilde I(w_1,w_2)$ to have a pole near $s_2=n$, even with all our assumptions about $w_2$. What we will argue is that for any $(w_1,w_2) \in W^2$, there exists a set $P \in W^2$ containing it so that the sum of $\tilde I$'s over $P$ will stay finite as $\varepsilon \rightarrow 0$. In other words, most singularities of \eqref{eq:integrand} will cancel with one another.

\subsection{The simplest case}

We demonstrate the phenomenon with the simplest case in which there exists precisely one $j$ such that (ii) above holds. We will take $w_i' = (j \ j+1) w_i$, and show that $\tilde I(w_1,w_2) + \tilde I(w'_1,w'_2)$ converges as $\varepsilon \rightarrow 0$.

We study how each term of $\tilde I(w'_1,w'_2)$ differs from its counterpart in $\tilde I(w_1,w_2)$. First of all, we have at $s_2=n$
\begin{equation*}
\langle \alpha^\vee_j,w_1\tilde\lambda_1+w_2\tilde\lambda_2\rangle= a\varepsilon,
\langle \alpha^\vee_j,w'_1\tilde\lambda_1+w'_2\tilde\lambda_2\rangle= -a\varepsilon
\end{equation*}
for some integer $a \neq 0$. For $\alpha^\vee = \alpha^\vee_{j-1}$ or $\alpha^\vee_{j+1}$, $\langle \alpha^\vee,w_1\tilde\lambda_1+w_2\tilde\lambda_2\rangle$ and $\langle \alpha^\vee,w'_1\tilde\lambda_1+w'_2\tilde\lambda_2\rangle$ differ by a factor of $1+O(\varepsilon)$; for all other choices of $\alpha^\vee$, they are identical. These also imply that $e^{\langle C,w_1\tilde\lambda_1+w_2\tilde\lambda_2\rangle}$ and $e^{\langle C,w'_1\tilde\lambda_1+w'_2\tilde\lambda_2\rangle}$ differ by a factor of $1+O(\varepsilon)$ as well.

On the other hand, from the functional equation \eqref{eq:fnleq_M}, we have
\begin{align*}
M(w'_i,\lambda_i) &= M((j \ j+1),w_i\lambda_i)M(w_i,\lambda_i) \\
&= c(\langle \alpha^\vee_j, w_i\lambda_i\rangle)M(w_i,\lambda_i),
\end{align*}
since $\alpha_j$ is the only positive root whose sign is changed by $(j \ j+1)$. By the symmetry of $c(z)$, we have $c(\langle \alpha^\vee_j, w_1\lambda_1\rangle) c(\langle \alpha^\vee_j, w_2\lambda_2\rangle) = 1$, and thus 
\begin{equation*}
M(w'_1,\lambda_1)M(w'_2,\lambda_2) = M(w_1,\lambda_1)M(w_2,\lambda_2),
\end{equation*}
or in other words,
\begin{align*}
M(w'_1,\tilde\lambda_1)M(w'_2,\tilde\lambda_2) &= \frac{c(\langle \alpha^\vee_j, w_1\tilde\lambda_1\rangle)}{c(\langle \alpha^\vee_j, w_1\lambda_1\rangle)}
M(w_1,\tilde\lambda_1)M(w_2,\tilde\lambda_2) \\
&= (1+O(\varepsilon))M(w_1,\tilde\lambda_1)M(w_2,\tilde\lambda_2).
\end{align*}

Collecting all our observations so far, we find that
\begin{equation*}
\langle \alpha^\vee_j,w_1\tilde\lambda_1+w_2\tilde\lambda_2\rangle\tilde I(w_1,w_2) + \langle\alpha^\vee_j,w'_1\tilde\lambda_1+w'_2\tilde\lambda_2\rangle\tilde I(w'_1,w'_2) = O(\varepsilon)
\end{equation*}
and thus
\begin{equation*}
\tilde I(w_1,w_2) + \tilde I(w'_1,w'_2) = O(1)
\end{equation*}
as desired.

In general, the analysis becomes far more complicated: it can take more than just two pairs for the cancellation, and the errors lumped into $O(\varepsilon)$ above require much greater scrutiny. In the next few sections, we identify a couple of patterns that make up the general situation, and study those one at a time. Then we combine them to extend the logic just demonstrated above.

\subsection{Consecutive zeros}

Let us consider the situation in which $\langle \alpha^\vee,w_1\lambda_1+ w_2\lambda_2\rangle = \pm(s_2-n)$ for $\alpha^\vee = \alpha^\vee_{j}, \ldots,\alpha^\vee_{j+k-1}$, and $\langle \alpha^\vee,w_1\lambda_1+ w_2\lambda_2\rangle \neq \pm(s_2-n)$ for $\alpha^\vee = \alpha^\vee_{j-1}$ and $\alpha^\vee_{j+k}$. {\color{red} boundaries can be identically zero however, and I must account for this carefully} For this to be the case, $w_2^{-1}(j), w_2^{-1}(j+1),\ldots,w_2^{-1}(j+k)$ must come from alternating blocks. 

We divide cases according to whether $w_2^{-1}(j)$ is in the first or the second block. First suppose it is in the first block. Then, for some integers $c,d>0$ such that $c+d=n$,
\begin{align*}
\langle \alpha^\vee, w_2\lambda_2\rangle &= s_2-c, -s_2+(c-1), s_2-c, -s_2+(c-1), \ldots \\
\langle \alpha^\vee, w_1\lambda_1\rangle &= -d, d+1, -d, d+1,\ldots
\end{align*}
for $\alpha^\vee = \alpha^\vee_j,\ldots,\alpha^\vee_{j+k-1}$. This implies that if $w_1^{-1}(j) = x$, then $w_1^{-1}(j+1) = x+d=:y$, and $w_1^{-1}(j+2) = x-1, w_1^{-1}(j+3)= y-1,$ and so on. By the shape of $\lambda_1$, this implies that $c(\langle \alpha^\vee,\tilde\lambda_1 \rangle)$ appears as a factor of $M(w_1,\tilde\lambda_1)$ in $\tilde I(w_1,w_2)$ for $\alpha^\vee = \alpha^\vee_{w_1^{-1}(j+2)}, \alpha^\vee_{w_1^{-1}(j+3)},\alpha^\vee_{w_1^{-1}(j+2)},\ldots,\alpha^\vee_{w_1^{-1}(j+k)}$. {\color{red} may help to have a picture here.} Therefore we can arrange the following cancellations: for $i=1,\ldots,k-1$, at $s_2=n$,
\begin{align*}
\frac{c(\langle \alpha^\vee_{w^{-1}(j+i+1)},\tilde\lambda_1\rangle)}{\langle \alpha^\vee_{j+i},w_1\tilde\lambda_1+ w_2\tilde\lambda_2\rangle} &= \frac{c(-1+\varepsilon)}{d\varepsilon} \mbox{ or } \frac{c(-1+\varepsilon)}{-(d+1)\varepsilon} \\
&= \frac{\pi}{-6d}(1+O(\varepsilon)) \mbox{ or } \frac{\pi}{6(d+1)}(1+O(\varepsilon)),
\end{align*}
according to whether $i$ is even or odd, respectively.

The situation is analogous when $w^{-1}_2(j)$ lies in the second block. In that case, we have
\begin{align*}
\langle \alpha^\vee, w_2\lambda_2\rangle &= -s_2+(c-1), s_2-c, -s_2+(c-1), s_2-c, \ldots \\
\langle \alpha^\vee, w_1\lambda_1\rangle &= d+1,-d,d+1,-d,\ldots
\end{align*}
for $\alpha^\vee = \alpha^\vee_j,\ldots,\alpha^\vee_{j+k-1}$, and for $i=1,\ldots,k-1,$
\begin{align*}
\frac{c(\langle \alpha^\vee_{w^{-1}(j+i+1)},\tilde\lambda_1\rangle)}{\langle \alpha^\vee_{j+i},w_1\tilde\lambda_1+ w_2\tilde\lambda_2\rangle} &= \frac{c(-1+\varepsilon)}{d\varepsilon} \mbox{ or } \frac{c(-1+\varepsilon)}{-(d+1)\varepsilon},
\end{align*}
according to whether $i$ is odd or even, respectively. To recap, in a set of consecutive denominator zeros, all but one cancel out.

Suppose $I(w_1,w_2)$ has precisely one consecutive zero, in the sense of our discussion so far. Let
\begin{equation*}
P = \{(ww_1,ww_2) \in W^2: w \in S_{\{j,\ldots,j+k\}}\}.
\end{equation*}
We claim that
\begin{equation} \label{eq:cc}
\sum_{(w_1',w_2') \in P} \tilde I(w_1',w_2') = O(1)
\end{equation}
as $\varepsilon \rightarrow 0$ at $s_2=n$.

It is first necessary to understand how $\tilde I(w_1',w_2')$ compare to $\tilde I(w_1,w_2)$. We continue to write $w_i'=ww_i$. We may assume that $w_2'$ also preserves the orders within each block, because if not then $\tilde I(w_1',w_2')=0$. Now recall
\begin{equation*}
(w_1\lambda_1+w_2\lambda_2)_j = (w_1\lambda_1+w_2\lambda_2)_{j+1} = \ldots = (w_1\lambda_1+w_2\lambda_2)_{j+k};
\end{equation*}
clearly the same holds if we replace $w_i$ by $w_i'$. Therefore, the $e^{\langle T,w_1\tilde\lambda_1+w_2\tilde\lambda_2\rangle}$ component changes by a factor of $1+O(\varepsilon)$. Clearly so do $\langle \alpha^\vee_{j-1},w_1\tilde\lambda_1+w_2\tilde\lambda_2\rangle$ and $\langle \alpha^\vee_{j+k},w_1\tilde\lambda_1+w_2\tilde\lambda_2\rangle$. In addition, we claim that
\begin{equation*}
M(w'_1,\lambda_1)M(w'_2,\lambda_2) = M(w_1,\lambda_1)M(w_2,\lambda_2),
\end{equation*}
so that
\begin{equation*}
M(w'_1,\lambda_1)M(w'_2,\lambda_2) = M(w_1,\lambda_1)M(w_2,\lambda_2)(1+O(\varepsilon)).
\end{equation*}
It suffices to show that
\begin{equation*}
M(w,w_1\lambda_1)M(w,w_2\lambda_2)=1,
\end{equation*}
whose proof is as follows. 
We color each of the numbers $j,j+1,\ldots,j+k$ red(blue) if its preimage by $w_2$ is from the first(second) block. {\color{red} may help to have a figure here} As $w \in S_{\{j,\ldots,j+k\}}$ permutes $j,j+1,\ldots,j+k$, it also permutes those colors. We may decompose $w$ into transpositions
\begin{equation*}
w = w_{(r)} = (\gamma_r\ \gamma_r+1) \cdots (\gamma_2 \ \gamma_2+1)(\gamma_1 \ \gamma_1+1)
\end{equation*}
such that $(\gamma_i\ \gamma_i+1)$ always swaps the red and blue numbers (from the coloring already permuted by the preceding transpositions $w_{(i-1)}$). On the other hand, we have $(w_1\lambda_1+w_2\lambda_2)_j = \ldots = (w_1\lambda_1+w_2\lambda_2)_{j+k}$, a property that is also preserved by every $w_{(i)}$. Thus
\begin{align*}
&M( (\gamma_i\ \gamma_i+1), w_{(i-1)}w_1\lambda_1)M( (\gamma_i\ \gamma_i+1), w_{(i-1)}w_2\lambda_2) \\
&= c(\langle \alpha^\vee_{\gamma_i},w_{(i-1)}w_1\lambda_1\rangle)c(\langle \alpha^\vee_{\gamma_i},w_{(i-1)}w_2\lambda_2\rangle)\\
&=1.
\end{align*}
This proves the claim.

For the terms $\langle \alpha^\vee,w_1\tilde\lambda_1+w_2\tilde\lambda_2\rangle$ for $\alpha^\vee=\alpha^\vee_j,\ldots,\alpha^\vee_{j+k-1}$, it is necessary to take a closer look. There still is the cancellation phenomenon for the terms
\begin{equation*}
\frac{c(\langle \alpha^\vee_{w^{-1}(j+i+1)},\tilde\lambda_1\rangle)}{\langle \alpha^\vee_{j+i},w'_1\tilde\lambda_1+ w'_2\tilde\lambda_2\rangle},
\end{equation*}
but its ratio in the $\varepsilon\rightarrow0$ limit is different from earlier in general. Basically we wish to show
\begin{equation} \label{eq:cc_core}
\sum_{(w_1',w_2') \in P} \prod_{i=0}^{k-1}\frac{1}{\langle \alpha^\vee_{j+i},w'_1\tilde\lambda_1+w'_2\tilde\lambda_2 \rangle} = 0.
\end{equation}
This will imply \eqref{eq:cc}.


Let us abstract out situation a bit before moving forward. 
Suppose $d,L,R$ are integers with $R \leq d$, and let $\mathfrak L\in W$ be the set of all permutations $w$ that maps $\{1,\ldots,L\} \cup \{L+d-R+1, \ldots,L+d\}$ bijectively onto $\{1,\ldots,L+R\}$ (and maps other elements in a pre-chosen manner), subject to the rule that $w(1)<\ldots<w(L)$ and $w(L+d-R+1)<\ldots<w(L+d)$, and that $w(1)=L+R$; define $\mathfrak R\in W$ analogously, except that the last condition becomes $w(L+d-R+1)=L+R$. Define
\begin{equation*}
l_{L,R}(d) = \sum_{w \in \mathfrak L}\prod_{i=1}^{L+R-1}\frac{1}{\langle \alpha^\vee_i,w\rho\rangle}, r_{L,R}(d) = \sum_{w \in \mathfrak R}\prod_{i=1}^{L+R-1}\frac{1}{\langle \alpha^\vee_i,w\rho\rangle}.
\end{equation*}
Then \eqref{eq:cc_core} equals $\varepsilon^{-k}(l_{L,R}(d)+r_{L,R}(d))$ with a certain choice of $d,L,R$, with $l_{L,R}(d)$(resp. $r_{L,R}(d)$) counting precisely the contribution from those pairs with $w_2'^{-1}(j+1)$ lying on the first(resp. second) block. Specifically, $d = |w_1'^{-1}(j)-w_1'^{-1}(j+1)|$, $L+R = k+1$, and
\begin{align*}
L=\lceil (k+1)/2\rceil, R=\lfloor (k+1)/2\rfloor &\mbox{ $w_2^{-1}(j)$ lies on the first block}\\
L=\lfloor (k+1)/2\rfloor, R=\lceil (k+1)/2\rceil &\mbox{ $w_2^{-1}(j)$ lies on the second block}.
\end{align*}
To see this, one simply notices that $\tilde\lambda_1+\tilde\lambda_2=\varepsilon\rho$. {\color{red} elaborate}

\begin{lemma}\label{lemma:cc1}
Suppose $R \leq d$. Then
$l_{L,0}(d)=(-1)^{L-1}, r_{0,R}(d)=(-1)^{R-1}$ for $L,R>0$, and $l_{0,R}(d)=r_{L,0}(d)=0$. Also there are the recursive relations
\begin{align*}
l_{L,R}(d) &= -l_{L-1,R}(d)-\frac{1}{d+L-R}r_{L-1,R}(d), \\
r_{L,R}(d) &= \frac{1}{d+L-R}l_{L,R-1}(d)-r_{L,R-1}(d).
\end{align*}
\end{lemma}
\begin{proof}
The first four equalities are straightforward. To show the recursive relation on $l_{L,R}$, we divide into cases in which $w^{-1}(L+R-1) = 2$ or $L+d$. The $w$'s belonging to the former case contribute $1/\langle \alpha^\vee_{L+R-1},w\rho \rangle=-1$ times $l_{L-1,R}(d)$, whereas those in the latter category contribute $1/\langle \alpha^\vee_{L+R-1},w\rho \rangle = -1/(d+L-R)$ times $r_{L-1,R}(d)$, proving the desired statement. The relation on $r_{L,R}$ can be proved by an analogous argument.
\end{proof}

The following immediately implies \eqref{eq:cc_core} and therefore \eqref{eq:cc}.
\begin{lemma}\label{lemma:cc2}
Suppose $L,R\geq 1$ and $R \leq d$. Then
\begin{align*}
l_{L,R}(d) = \frac{(-1)^{L+R-1}}{d+L-R}, r_{L,R}(d) = \frac{(-1)^{L+R}}{d+L-R}.
\end{align*}
\end{lemma}
\begin{proof}
We induct on $\max(L,R)$. The base case $\max(L,R)=1$ can be checked directly. For the inductive step, we use Lemma \ref{lemma:cc1} and the inductive hypothesis to write
\begin{align*}
l_{L,R}(d) = -\frac{d+L-R-1}{d+L-R}l_{L-1,R}(d), r_{L,R}(d) = -\frac{d+L-R+1}{d+L-R}r_{L,R-1}(d).
\end{align*}
From these, one easily computes
\begin{align*}
l_{L,R}(d) &= (-1)^{L-1}\frac{d-R+1}{d+L-R}l_{1,R}(d)=\frac{(-1)^{L+R-1}}{d+L-R},\\
r_{L,R}(d) &= (-1)^{R-1}\frac{d+L-1}{d+L-R}r_{L,1}(d) =\frac{(-1)^{L+R}}{d+L-R},
\end{align*}
as desired.
\end{proof}

\subsection{When different zeros are $\geq 2$ distances apart}

We move on to the case in which the denominator of $I(w_1,w_2)$ has two or more zeros, regarding a set of consecutive zeros as one zero, as established in the previous section. It is easier when all zeros are at least two distances apart, i.e., if $\langle \alpha^\vee_{j_1},w_1\lambda_1+w_2\lambda_2 \rangle$ and $\langle \alpha^\vee_{j_2},w_1\lambda_1+w_2\lambda_2 \rangle$ vanish at (only) $s_2=n$ and nothing in between does, then $j_2-j_1 \geq 3$.

\subsection{When different zeros are $1$ distance apart}
}

\section{An application} \label{sec:app}


\subsection{Settings}

Fix $0 < \eta_2 < 1$, and integers $n \geq 4$ and $2 \leq k \leq n-2$. Let $P=P(k)$ the maximal parabolic subgroup of $G$, and $h:\R_{\geq 0} \rightarrow \R_{\geq 0}$ be a nonnegative function of compact support whose Mellin transform $Mh(s)$ and its inverse transform exist on the strip $\re s \in [n-\eta_2,n+2\eta_2-\eta_2^2]$.

Since our main formula Theorem \ref{thm:res_conc} requires that the input function be smooth, we need to convolve $h$ by a bump function $g$. Define
\begin{equation*}
\sigma(x) = \begin{cases} Ae^{-1/(1-x^2)} & \mbox{if $|x| < 1$}, \\ 0 & \mbox{otherwise,} \end{cases}
\end{equation*}
where $A\in\R$ is defined so that $\int\sigma(x)dx=1$. For a (small) parameter $0 < \delta < 1$, let $B_\delta(x) = \delta^{-1} \cdot \sigma(\delta^{-1}x)$ and $\beta_\delta(x) = B_\delta(\log x)$.

\begin{lemma}\label{lemma:sundry}
Write $s=c+it$. The following statements hold.
\begin{enumerate}[(i)]
\item $|\widehat\sigma(z)| \ll |(\im z)^{-3/4}e^{-\sqrt{\im z}}|$ for any $z \in \C$ with bounded $\re z$.
\item $|M\beta_\delta(s)| \ll |(\delta t)^{-3/4}e^{-\sqrt{\delta t}}|$
\item $M\beta_\delta(c) = 1+O(c\delta)$ for $\delta$ in a neighborhood of zero.
\end{enumerate}
\end{lemma}
\ignore{\color{red} The saddle point stuff is kind of a pain. Though it's possible to get around it --- we know $\hat\sigma(t)$ and $Mg(c+it)$, as a function of $t$ decrease faster than any polynomial, since they are Fourier transforms of bump functions; and that's enough for us actually --- it'd be nice to know what's going on more precisely.

I also may want to change $G(x) = \delta^{-1} \cdot \sigma(\delta^{-1}x)e^{-nx}$. This has the effect of making $Mg(n) = 1$.

Below I end up using $g$ for two things. Change it.
}
\begin{proof}
(i) is well-known, see e.g., \cite[Section 2]{Joh15}; we also added a more elaborate computation in the appendix below. (ii) follows from (i) and 
\begin{align*}
M\beta_\delta(s) &= \int_0^\infty \beta_\delta(x)x^{s-1}dx \\
&= \int_{-\infty}^\infty B_\delta(x)e^{(c+it)x}dx \\
&= \frac{1}{\delta}\int_{-\delta}^\delta \sigma(\delta^{-1}x)e^{i(t-ci)x}dx \\
&= \widehat\sigma(\delta(t-ci)).
\end{align*}
The above calculation also yields that
\begin{equation*}
M\beta_\delta(c) = \frac{1}{\delta}\int_{-\delta}^\delta \sigma(\delta^{-1}x)e^{cx}dx.
\end{equation*}
Since $e^{cx}= 1+O(c\delta)$ in the range of integration, this proves (iii).
\end{proof}

For a subset $A$ of $\R$, define
\begin{equation*}
h_A(x) = \begin{cases} 1 & \mbox{if $x \in A$} \\ 0 & \mbox{otherwise.}\end{cases}
\end{equation*}
For $N \in \R$, we write $h_N=h_{[0,N]}$. It may be helpful to note that $Mh_N(s) = N^s/s$. Let $f_{A,\delta}=h_A*\beta_\delta$, and similarly for $f_{N,\delta}$.

\subsection{Second moment estimate}

For our application Theorem \ref{thm:intro_grass}, we are interested in the second moment
\begin{align*}
\frac{1}{\vol^C(\Gamma\backslash G)}\int_{\Gamma \backslash G} \left|\mathrm H^CE_{f_{A,\delta}}(g)\right|^2dg.
\end{align*}
for suitable choices of intervals $A \subseteq \R$ and $\delta$.
Since $\Lambda^CE_{f_{A,\delta}}(g)$ is real-valued, and $|\mathrm H^CE_{f_{A,\delta}}(g)| \leq |\Lambda^CE_{f_{A,\delta}}(g)|$ for all $g \in \Gamma\backslash G$, this is bounded above by
\begin{align*}
\frac{1}{\vol^C(\Gamma\backslash G)}\int_{\Gamma \backslash G} \left(\Lambda^CE_{f_{A,\delta}}(g)\right)^2dg.
\end{align*}

We apply Theorem \ref{thm:res_conc} to estimate this integral. Clearly, the main term is given by
\begin{align*}
\left(Mf_{A,\delta}(n)\xi(n,k)\right)^2.
\end{align*}

For a more precise error estimate, we revisit the expressions \eqref{eq:res-half} and \eqref{eq:res-remain}. First, consider the integral
\begin{align*}
\int_{\re s_2=n-\eta_2} Mf_2(s_2)\sum_{w_2 \in W}\frac{e^{\langle C, w_2\lambda_2-\rho+O(\varepsilon)\rangle}}{\prod_{\alpha \in \Delta}\langle\alpha^\vee,w_2\lambda_2-\rho+O(\varepsilon)\rangle}M(w_2,\lambda_2)ds_2
\end{align*}
in \eqref{eq:res-half}. By Lemma \ref{lemma:misc_growth}, $|M(w_2,\lambda_2)| \ll_n |\im s_2|^{0.501\eta_2}$. On the other hand, none of the denominator terms $\langle\alpha^\vee,w_2\lambda_2-\rho\rangle$ vanish on the line $\re s_2=n-\eta_2$; moreover, at least one of them is linear in $\im s_2$. Therefore, for $f_2 = f_{A,\delta}$, this integral is bounded by a constant in $n,\eta_1,\eta_2$ times
\begin{align} \label{eq:prec_err_est1}
T^{-\eta_2 \frac{k_2(n-k_2)}{2}}\int_{\re s_2=n-\eta_2} \left|Mh_A(s_2)M\beta_\delta(s_2) (1+\im s_2)^{-0.499\eta_2}\right|ds_2.
\end{align}

Let us also consider the integral \eqref{eq:res-remain}. Again, we can use Lemma \ref{lemma:misc_growth} to bound the growths of $M(w_1,\lambda_1)$ and $M(w_2,\lambda_2)$. The denominator terms again does not vanish in the region of integration. Hence, for $f_1=f_2=f_{A,\delta}$, \eqref{eq:res-remain} is bounded by a constant in $n,\eta_1,\eta_2$ times
\begin{align} \label{eq:prec_err_est2}
T^{\kappa}\prod_{i=1,2}\int_{\re s_i=n-\eta_i} \left|Mh_A(s_i)M\beta_\delta(s_i) (\im s_i)^{0.501\eta_i}\right|ds_i.
\end{align}

At this point, we need the following lemma.

\begin{lemma} \label{lemma:mellin}
Suppose $s = c + it$, $c \geq 1$, and $A = (N_1,N_2]$, so that either $N_1=0$ or $N_2^c - N_1^c \leq N_1^c/2$. Then $|Mh_A(s)| \ll |Mh_A(c)|$ for any $t \in \R$.
\end{lemma}
\begin{proof}
First note that
\begin{align*}
Mh_A(s) = \frac{N_2^s}{s}-\frac{N_1^s}{s}.
\end{align*}
Thus the case $N_1=0$ is trivial, and we assume $N_1 > 0$ for the rest of the proof.
Rewriting this as
\begin{align*}
&\frac{1}{s}\left(N_2^ce^{it\log N_2} - N_1^ce^{it\log N_1}\right) \\
&= \frac{1}{s}\left((N_2^c-N_1^c)e^{it\log N_2} + N_1^ce^{it\log N_2}(1-e^{it\log\frac{N_1}{N_2}})\right),
\end{align*}
it suffices to show that
\begin{align*}
\left|\frac{1}{c+it}N_1^c(1-e^{it\log\frac{N_1}{N_2}})\right| \ll \frac{1}{c}(N_2^c-N_1^c)
\end{align*}
for any $t \in \R$.

When $|t\log(N_2/N_1)| \leq 1/2$, the left-hand side of the above inequality is bounded by a constant times
\begin{align*}
\left|\frac{t}{c+it}N_1^c\log\frac{N_2}{N_1}\right| \ll N_1^c\log\frac{N_2}{N_1}.
\end{align*}
If $|t\log(N_2/N_1)| > 1/2$, then it is bounded by
\begin{align*}
\left|\frac{2N_1^c}{c+it}\right| \ll \frac{N_1^c}{t} \ll N_1^c\log\frac{N_2}{N_1}.
\end{align*}

On the other hand, the assumption implies that
\begin{align*}
c\log \frac{N_2}{N_1} = \log\left(1+\frac{N_2^c-N_1^c}{N_1^c}\right) \ll \frac{N_2^c-N_1^c}{N_1^c},
\end{align*}
which completes the proof.
\end{proof}

Thanks to Lemma \ref{lemma:mellin} above, the integrals in \eqref{eq:prec_err_est1} and \eqref{eq:prec_err_est2} are bounded by

\begin{align*}
Mh_A(n-\eta_2)\int_{\re s_2=n-\eta_2} \left|M\beta_\delta(s_2) (1+\im s_2)^{-0.499\eta_2}\right|ds_2
\end{align*}
and
\begin{align*}
Mh_A(n-\eta_i)\int_{\re s_i=n-\eta_i} \left|M\beta_\delta(s_i) (\im s_i)^{0.501\eta_i}\right|ds_i,
\end{align*}
respectively. Plugging in Lemma \ref{lemma:sundry}(ii) and performing the change of variables $\delta \cdot \im s_i = u$, it turns out that they are of order $O(\delta^{-1}Mh_A(n-\eta_2))$ and $O(\delta^{-1-0.501\eta_i}Mh_A(n-\eta_i))$, respectively. This completes the error estimate. We summarize our discussion in the form of a proposition.

\begin{proposition} \label{prop:prec_err_est}
Fix $T > 0$. Let $A=(N_1,N_2]$, so that either $N_1=0$ or $N_2^c-N_1^c \leq N_1^c/2$ for $c=n-\eta_1$ and $n-\eta_2$. Assume in addition that
\begin{align*}
Mh_A(n)T^{-\eta_2\frac{k_2(n-k_2)}{2}} > Mh_A(n-\eta_1)T^\kappa.
\end{align*}
This is satisfied, for instance, if $Mh_A(n-\eta_1)$ is sufficiently large relative to $T$.
Then
\begin{align*}
&\frac{1}{\vol^C(\Gamma\backslash G)}\int_{\Gamma \backslash G} \left|\mathrm H^CE_{f_{A,\delta}}(g)\right|^2dg \\
&= \left(Mf_{A,\delta}(n)\xi(n,k)\right)^2 + O(\delta^{-\alpha}Mh_A(n)Mh_A(n-\eta_2)),
\end{align*}
where the implicit constant depends on $n, \eta_1, \eta_2,$ and $T$, and $\alpha = 2+0.501(\eta_1+\eta_2)$.
\end{proposition}
\begin{remark}
By the remark under Lemma \ref{lemma:misc_growth}, and setting $\eta_1\approx0$ and $\eta_2\approx1$, $\alpha$ can in fact be made arbitrarily close to (but greater than) $2.5$.
\end{remark}
\begin{proof}
The only additional work done here is the fusion of the two error terms \eqref{eq:prec_err_est1} and \eqref{eq:prec_err_est2}: recall that they are
\begin{align*}
O\left(\delta^{-1}Mh_A(n)Mh_A(n-\eta_2)T^{-\eta_2\frac{k_2(n-k_2)}{2}}\right)
\end{align*}
and
\begin{align*}
O\left(\delta^{-\alpha}Mh_A(n-\eta_1)Mh_A(n-\eta_2)T^\kappa\right),
\end{align*}
respectively.
\end{proof}

\subsection{Adaptation of Schmidt's argument}

In this section, we adapt the famous Borel-Cantelli argument of Schmidt \cite{Sch60} to the problem of counting sublattices. We closely follow especially pages 520-521 of \cite{Sch60}. We will first derive a result for the smoothed function $f_{A,\delta}$ first (Theorem \ref{thm:grass_smooth} below), and then finesse a result for $h_A$ from it (Theorem \ref{thm:grass}).

From now on, take $\eta_1=1-\eta_2$, and write $\eta = \eta_2$. We intend $\eta$ to be close to $1$ from the left. Fix $T \gg 0$ and $H_0 \gg_T 0$. Choose a sequence $\{p_j\}_{j=0}^\infty$ so that $p_0=0$, $p_1= (2H_0)^{2n-\eta}$, and $p_{j+1}^{2n-\eta} - p_j^{2n-\eta} = H_0^{2n-\eta}$. For integers $k>j\geq0$, write
\begin{align*}
V_j^k = \mathrm{Var}(\Lambda^CE_{f_{(p_j,p_k],\delta}}).
\end{align*}
Note that $p_k^a-p_j^a \leq p_j^a/2$ for all $a < 2n-\eta$, so that the condition of Proposition \ref{prop:prec_err_est} is satisfied.

We claim $V_j^k \leq \delta^{-\alpha}H_1(k-j)$ for some constant $H_1=H_1(n,\eta,T)$, where $\alpha$ is as in the statement of Proposition \ref{prop:prec_err_est}. By Proposition \ref{prop:prec_err_est},
\begin{align*}
\delta^{\alpha}V_j^k &\ll Mh_{(p_j,p_k]}(n)Mh_{(p_j,p_k]}(n-\eta) \\
&\ll (p_k^n-p_j^n)(p_k^{n-\eta}-p_j^{n-\eta}) \\
&< p_k^{2n-\eta}-p_j^{2n-\eta} \leq 2H_0^{2n-\eta}(k-j),
\end{align*}
where the implicit constants depend on $n, \eta$, and $T$.
This verifies the claim.

Following Schmidt \cite{Sch60}, define for $S \in \Z_{> 0}$
\begin{align*}
K_S = \{(j,k):0 \leq j<k\leq 2^S, j=u2^s,k-j=2^s \mbox{ for $u,s\in\Z_{\geq 0}$}\}.
\end{align*}
\begin{lemma}\label{lemma:sch1}
In the above context, it holds that
\begin{align*}
\sum_{(j,k)\in K_S} V_j^k \leq \delta^{-\alpha}H_1(S+1)2^S.
\end{align*}
\end{lemma}
\begin{proof}
For each $s = 0,\ldots,S$, their are precisely $2^{S-s}$ elements $(j,k)\in K_S$ such that $k-j=2^s$. These elements contribute at most $H_12^{S-s}\cdot2^s=H_12^S$ to the sum. Since there are $(S+1)$ possible choices for $s$, this completes the proof.
\end{proof}
\begin{lemma}\label{lemma:sch2}
Let $\psi:\R_{\geq 0} \rightarrow \R_{\geq 0}$ be a nondecreasing function such that
\begin{align*}
\int_0^\infty \frac{dx}{\psi(x)} < \infty.
\end{align*}
Then for almost every $g \in \Gamma \backslash G$, there exists a constant $H_2=H_2(n,\eta,g)>0$ such that 
\begin{align*}
\left(E_{f_{p_N,\delta}}(g) - Mf_{p_N,\delta}(n)\xi(n,k)\right)^2 < \delta^{-\alpha}H_2 N(\log N)^2\psi(\log N)
\end{align*}
for all integers $N \gg_g 0$.
\end{lemma}
\begin{proof}
It suffices to prove this for almost every $g$ in the truncated domain $\mathrm{supp}\,\chi^C_{G,G}$ for all sufficiently large $T$.

Write $\tilde\psi(S) = \psi((S-1)\log 2)$. By Lemma \ref{lemma:sch1} above and Markov's inequality, the set of $g \in \mathrm{supp}\,\chi^C_{G,G}$ such that
\begin{align*}
\sum_{(j,k)\in K_S}\left(E_{f_{(p_k-p_j],\delta}}(g) - Mf_{(p_k-p_j],\delta}(n)\xi(n,k)\right)^2 > \delta^{-\alpha}H_1(S+1)2^S\tilde\psi(S)
\end{align*}
has measure $O(1/\tilde\psi(S))$. Since $\sum_{S \gg 0} 1/\tilde\psi(S)$ converges by assumption, the Borel-Cantelli lemma implies that almost every $g \in \mathrm{supp}\,\chi^C_{G,G}$ satisfies
\begin{align*}
\sum_{(j,k)\in K_S}\left(E_{f_{(p_k-p_j],\delta}}(g) - Mf_{(p_k-p_j],\delta}(n)\xi(n,k)\right)^2 \leq \delta^{-\alpha}H_1(S+1)2^S\tilde\psi(S)
\end{align*}
for all $S \gg_g 0$. On the other hand, for any $2^{S-1} < N \leq 2^S$,
\begin{align*}
E_{f_{p_N,\delta}}(g) - Mf_{p_N,\delta}(n)\xi(n,k)
\end{align*}
is the sum of at most $S$ elements of the form
\begin{align*}
E_{f_{(p_k-p_j],\delta}}(g) - Mf_{(p_k-p_j],\delta}(n)\xi(n,k)
\end{align*}
for $(j,k)\in K_S$. Therefore, by the Cauchy-Schwarz inequality,
\begin{align*}
\left(E_{f_{p_N,\delta}}(g) - Mf_{p_N,\delta}(n)\xi(n,k)\right)^2 &\leq \delta^{-\alpha}H_1S(S+1)2^S\tilde\psi(S) \\
&\leq \delta^{-\alpha}H_2N(\log N)^2\psi(\log N)
\end{align*}
for some constant $H_2 > 0$. $H_2$ depends on $n,\eta,$ and $T$; but the dependence on $T$ can be replaced by the dependence on $g$.
\end{proof}

By taking $\epsilon_0 >0$ such that $\eta>1-2\epsilon_0$, the above two lemmas readily imply the following result.
\begin{theorem} \label{thm:grass_smooth}
Choose any $\epsilon_0 > 0$. Then for almost every $g \in \Gamma\backslash G$ and $p \gg_{g} 0$,
\begin{align*}
E_{f_{p,\delta}}(g) - Mf_{p,\delta}(n)\xi(n,k) = O_{n,\epsilon_0,g}\left(\delta^{-\alpha}p^{n-\frac{1}{2}+\epsilon_0}\right).
\end{align*}
\end{theorem}
\begin{remark}
Notice that the error term here is much better than what is suggested by the previous works on the topic \cite{Sch68, Thu93, Kim23}. In fact, if the bound on $\kappa$ can be refined to be slightly better than $\kappa < k(n-k)/2$, it is possible to improve it even further: it allows one to make different choices of $T$ and $\eta_1$ so that \eqref{eq:res-remain} becomes smaller than \eqref{eq:res-half}.
\end{remark}

\begin{proof}
We continue from the proof of Lemma \ref{lemma:sch2}. Take a number $p$ between $p_N$ and $p_{N+1}$. Then
\begin{align*}
E_{f_{p_N}}(g) - Mf_{p_{N+1}}(n)\xi(n,k) &\leq E_{f_{p}}(g) - Mf_{p}(n)\xi(n,k) \\ &\leq E_{f_{p_{N+1}}}(g) - Mf_{p_N}(n)\xi(n,k).
\end{align*}
Thus, in the light of Lemma \ref{lemma:sch2}, it suffices to show that $Mf_{p_N}(n) = p_N^n/n$ and $Mf_{p_{N+1}}(n)= p_{N+1}^n/n$ differ by at most $O_n(p^{n-\eta/2})$ for $N$ sufficiently large. In fact, they differ by a much smaller amount, since by assumption $p_{N+1}^{2n-\eta}-p_N^{2n-\eta}$ is a constant.
\end{proof}

The freedom over the choice of $\delta$ allows us to do without smoothing, with a slightly worse error term.

\begin{theorem} \label{thm:grass}
Choose any $\epsilon>0$. Then for almost every $g \in \Gamma\backslash G$ and $p \gg_{g} 0$,
\begin{align*}
E_{h_p}(g) - Mh_p(n)\xi(n,k) = O_{n,\epsilon,g}(p^{n-\frac{1}{7}+\epsilon}).
\end{align*}
\end{theorem}
\begin{proof}
One can check that, for every $g \in \G\backslash G$,
\begin{align*}
E_{f_{pe^{-\delta},\delta}}(g) \leq E_{h_p}(g) \leq E_{f_{pe^\delta},\delta}(g).
\end{align*}
By Lemma \ref{lemma:sundry}(iii) and Theorem \ref{thm:grass_smooth}, this implies that for almost every $g \in \G\backslash G$ and $p \gg_{g} 0$
\begin{align*}
E_{h_p}(g) - Mh_p(n)\xi(n,k) = O_{n,\epsilon_0,g}\left(\delta p^n + \delta^{-\alpha}p^{n-\frac{1}{2}+\epsilon_0} \right).
\end{align*}
As remarked after the statement of Proposition \ref{prop:prec_err_est}, $\alpha > 5/2$ can be made arbitrarily close to $5/2$; and of course, $\epsilon_0$ can be set to be arbitrarily small. Hence if we choose $\delta = p^{-1/7}$, the right-hand side above becomes $O_{n,\epsilon,g}(p^{n-1/7+\epsilon})$ for an arbitrarily small $\epsilon$.
\end{proof}

\newpage
\appendix

\section{The saddle point method}

The purpose of this appendix is to elaborate on the computations of \cite{Joh15}, for completeness of our paper. 
For $k \in \C$, let $F(k)=\int_{-1}^1e^{ikx-\frac{1}{1-x^2}}dx$. Because for $\re k\leq 0$, we have $F(k)=F(-k)$, it is sufficient to consider the case that $\re k\geq 0$. Letting $t=1-x$, this equals to $\int_0^2 e^{ik(1-t)-\frac{1}{t(2-t)}}dt$, so 
$$F(k)=e^{ik}\int_0^2 e^{-ikt-\frac{1}{2t}-\frac{1}{4-2t}}dt.$$
Note that on $t\in (0,\frac{1}{\re k})$, $e^{-\frac{1}{2t}}<e^{-\frac{\re k}{2}}$, and similarly on $t\in (2-\frac{1}{\re k}, 2)$, $e^{-\frac{1}{4-2t}}<e^{-\frac{\re k}{2}}$, hence the integrand near $0$ or $2$ is negligible. 

Letting $h(t)=-ikt-\frac{1}{2t}-\frac{1}{4-2t}$, $h'(t)=-ik+\frac{1}{2t^2}-\frac{2}{(4-2t)^2}=-ik+\frac{-16(t-1)}{8t^2(2-t)^2}$. Solving $h'(t)=0$ is not easy. Instead, we consider expansions of it near $0$ and $2$ respectively: for example, near $0$ we have

$$-ikt-\frac{1}{2t}-\frac{1}{4-2t}=-ikt-\frac{1}{2t}-\frac{1}{4}+O(t).$$
Then the integral $\int_0^2 e^{-ikt-\frac{1}{2t}-\frac{1}{4}+O(t)}dt= \int_0^2 e^{-ikt-\frac{1}{2t}-\frac{1}{4}}(1+O(t))dt$ near $0$. 

Let $$g(t)=-ikt-\frac{1}{2t}-\frac{1}{4}.$$ 
Then $g'(t)=-ik+\frac{1}{2t^2}$ and solving $g'(t)=0$ gives $$t_1=\frac{1}{\sqrt{2ik}}=\frac{1-i}{2\sqrt{k}},$$ at which $e^{-ikt-\frac{1}{2t}-\frac{1}{4}}$ is maximal. Similarly, near $2$ the same reasoning gives us $$t_2=2-\frac{1+i}{2\sqrt{k}}.$$ We want to move the path of integral to another path through these saddle points and apply the saddle point method to find the main term, and do an error analysis to bound error. Let $k=re^{i\theta}$, where $r=|k|$ and $\theta\in(-\frac{\pi}{2},\frac{\pi}{2})$. Then $\sqrt{k}=\sqrt{r}e^{\frac{i\theta}{2}}$, $t_1=\frac{e^{-i(\frac{\theta}{2}+\frac{\pi}{4})}}{\sqrt{2r}}$ and $t_2=2-\frac{e^{-i(\frac{\theta}{2}-\frac{\pi}{4})}}{\sqrt{2r}}$.

Let $C$ be the path consisting of the line $C_1(\theta)\subset \{x+yi\in\mathbb{C} \ | \ y=-\tan(\frac{\theta}{2}+\frac{\pi}{4})x\} $ joining $0$ and $1-ie^{-i\theta}$, and the line $C_2(\theta)\subset \{x+yi\in\mathbb{C} \ | \ y=\cot(\frac{\theta}{2}+\frac{\pi}{4})(x-2)\} $ joining $1-ie^{-i\theta}$ and $2$ --- one can easily check that the two lines indeed meet at $1-ie^{-i\theta}$.
Then it holds that
\begin{align*}
F(k)=e^{ik}\left(\int_{C_1(\theta)}e^{-ikt-\frac{1}{2t}-\frac{1}{4-2t}}dt+\int_{C_2(\theta)} e^{-ikt-\frac{1}{2t}-\frac{1}{4-2t}}dt\right).
\end{align*}
We will work on one integral at a time.

For the integral over $C_1(\theta)$, first consider, for $\alpha\in (r^{-\frac{1}{4}+\epsilon},1)$, the quantity $t=\frac{1}{\sqrt{2k}}e^{-\frac{i\pi}{4}}(1-\alpha)=\frac{1}{\sqrt{2r}}e^{-i(\frac{\theta}{2}+\frac{\pi}{4})}(1-\alpha)=t_1(1- \alpha )$ in the line connecting $0$ and $t_1$. We have 
$$|e^{-ikt-\frac{1}{2t}-\frac{1}{4-2t}}|=|e^{-ikt_1-\frac{1}{2t_1}-\frac{1}{4}}||e^{ikt_1\alpha-\frac{\alpha}{2t_1(1-\alpha)}}e^{O(t_1)}|$$
and
\begin{align*}
\re \left(ikt_1\alpha-\frac{\alpha}{2t_1(1-\alpha)}\right)&=\re \left(ire^{i\theta}\frac{1}{\sqrt{2r}}e^{-i(\frac{\theta}{2}+\frac{\pi}{4})}\alpha-\frac{\alpha}{2\frac{1}{\sqrt{2r}}e^{-i(\frac{\theta}{2}+\frac{\pi}{4})}(1-\alpha)}\right)\\
& = -\frac{\alpha\sqrt{r}}{\sqrt{2}}\left(\sin\left(\frac{\theta}{2}-\frac{\pi}{4}\right)+\cos\left(\frac{\theta}{2}+\frac{\pi}{4}\right)\frac{1}{1-\alpha}\right) \\
&=-\frac{\alpha^2\sqrt{r}}{\sqrt{2}(1-\alpha)} \cos\left(\frac{\theta}{2}+\frac{\pi}{4}\right) \leq -\frac{r^{2\epsilon}}{\sqrt{2}} \cos\left(\frac{\theta}{2}+\frac{\pi}{4}\right),
\end{align*}
which imply
\begin{align*}
|e^{-ikt-\frac{1}{2t}-\frac{1}{4-2t}}|\ll|e^{-ikt_1-\frac{1}{2t_1}-\frac{1}{4}}| e^{- \frac{r^{2\epsilon}}{\sqrt{2}} \cos\left(\frac{\theta}{2}+\frac{\pi}{4}\right)}.
\end{align*}
Similarly for $t=t_1(1+\beta)\in C$, where $\beta\geq r^{ -\frac{1}{4}+\epsilon}$, we have
$$|e^{-ikt-\frac{1}{2t}-\frac{1}{4-2t}}|=|e^{-ikt_1-\frac{1}{2t_1}-\frac{1}{4}}||e^{-ikt_1\beta+\frac{\beta}{2t_1(1+\beta)}}e^{O(t_1)}|$$ and
\begin{align*}
\re \left(-ikt_1\beta+\frac{\beta}{2t_1(1+\beta)}\right)&=\frac{\beta\sqrt{r}}{\sqrt{2}}\left(\sin\left(\frac{\theta}{2}-\frac{\pi}{4}\right)+\cos\left(\frac{\theta}{2}+\frac{\pi}{4}\right)\frac{1}{1+\beta}\right)\\
&=-\frac{\beta^2\sqrt{r}}{\sqrt{2}(1+\beta)} \cos\left(\frac{\theta}{2}+\frac{\pi}{4}\right) \\ & \leq -\frac{r^{2\epsilon}}{\sqrt{2}(1+r^{-\frac{1}{4}+\epsilon})} \cos\left(\frac{\theta}{2}+\frac{\pi}{4}\right),
\end{align*}
and hence
\begin{align*}
|e^{-ikt-\frac{1}{2t}-\frac{1}{4-2t}}|\ll|e^{-ikt_1-\frac{1}{2t_1}-\frac{1}{4}}| e^{- \frac{r^{\frac{1}{2}-2\epsilon}}{\sqrt{2}(1+r^{-\epsilon})} \cos\left(\frac{\theta}{2}+\frac{\pi}{4}\right)}.
\end{align*}

From these estimations, we see that the main contribution comes from $t=t_1(1+\gamma)\in C$ with $\gamma\in (-r^{-\frac{1}{4}+\epsilon},r^{-\frac{1}{4}+\epsilon})$.
We have for $g(t)=-ikt-\frac{1}{2t}-\frac{1}{4}$, 
\begin{align*}
g(t)&=g(t_1)+\frac{g'(t_1)}{1}(t-t_1)+\frac{g''(t_1)}{2!}(t-t_1)^2+\cdots\\
&=g(t_1)+\frac{g''(t_1)}{2!}(t-t_1)^2+\frac{1}{2}\sum_{m=3}^{\infty}\left(\frac{-1}{t_1}\right)^{m+1}(t-t_1 )^m \\
&=-ikt_1-\frac{1}{2t_1}-\frac{1}{4}-\frac{\gamma^2}{2t_1}-\frac{1}{2t_1}\sum_{m=3}^{\infty}( -\gamma)^m\\
&=-ikt_1-\frac{1}{2t_1}-\frac{1}{4}-\frac{\gamma^2}{2t_1}+O(|t_1|^{-1}\gamma^3).
\end{align*}
Also, writing $h(t)=e^{-\frac{2t}{4(4-2t)}}$,
\begin{align*}
&\int_{t_1(1-r^{-\frac{1}{4}+\epsilon})}^{t_1(1+r^ { -\frac{1}{4}+\epsilon})}e^{-ikt-\frac{1}{2t}-\frac{1}{4-2t}}dt=\int_{t_1(1-r^{-\frac{1}{4}+\epsilon})}^{t_1(1+r^ { -\frac{1}{4}+\epsilon})}h(t)e^{g(t) }dt\\
&=t_1\int_{ -r^{-\frac{1}{4}+\epsilon} }^{ r^ { -\frac{1}{4}+\epsilon} }h(t_1(1+\gamma))e^{g(t_1(1+\gamma)) }d\gamma\\
&=t_1e^{-ikt_1-\frac{1}{2t_1}-\frac{1}{4}}\int_{ -r^{-\frac{1}{4}+\epsilon} }^{ r^ { -\frac{1}{4}+\epsilon} }h(t_1(1+\gamma))e^{-\frac{\gamma^2}{2t_1}-\frac{1}{t_1}\sum_{m=3}^{\infty}(-\gamma)^m}d\gamma \\
&=t_1\sqrt{2t_1}e^{-ikt_1-\frac{1}{2t_1}-\frac{1}{4}}\int_{ -\frac{r^{-\frac{1}{4}+\epsilon}}{\sqrt{2t_1}} }^{\frac{r^ { -\frac{1}{4}+\epsilon}}{\sqrt{2t_1}}}h(t_1(1+\sqrt{2t_1}\gamma))e^{- \gamma^2 - 2\sum_{m=3}^{\infty}(2t_1)^{\frac{m}{2}-1}(-\gamma)^m}d\gamma\\ 
&=t_1\sqrt{2t_1}e^{-ikt_1-\frac{1}{2t_1}-\frac{1}{4}}\int_{ -\frac{r^{-\frac{1}{4}+\epsilon}}{\sqrt{2t_1}} }^{\frac{r^ { -\frac{1}{4}+\epsilon}}{\sqrt{2t_1}}} e^{- \gamma^2 - \sum_{m=3}^{\infty}t_1^{\frac{m}{2}-1}(-\gamma)^m}d\gamma+O\left(|t_1|^{\frac{5}{2}}|e^{-ikt_1-\frac{1}{2t_1}-\frac{1}{4}}|\right).
\end{align*}
Here, it holds that
\begin{align*}
\int_{ -\frac{r^{-\frac{1}{4}+\epsilon}}{\sqrt{2t_1}} }^{\frac{r^ { -\frac{1}{4}+\epsilon}}{\sqrt{2t_1}}} e^{- \gamma^2 - 2\sum_{m=3}^{\infty}(2t_1)^{\frac{m}{2}-1}(-\gamma)^m}d\gamma &=  \int_{ -  r^{\epsilon}}^{ r^ { \epsilon}} e^{- \gamma^2 -2 \sum_{m=3}^{\infty}(2t_1)^{\frac{m}{2}-1}(-\gamma)^m}d\gamma+O(e^{-\frac{1}{2}r^{\epsilon}})\\
&=\int_{ -r^{\epsilon} }^{ r^ {\epsilon}} e^{- \gamma^2} d\gamma+O(|t|^{\frac{1}{2}}r^{3\epsilon})+O(e^{-\frac{1}{2}r^{-\epsilon}(2r)^{\frac{1}{4}}}) \\ &= \int_{ - \infty }^{ \infty} e^{- \gamma^2 }d\gamma+O(r^{-\frac{1}{4}+3\epsilon})= \sqrt{\pi}+O(r^{-\frac{1}{4}+3\epsilon}).
\end{align*}
Summing up everything together, we conclude that
\begin{align*}
e^{ik}\int_{C_1(\theta)}e^{-ikt-\frac{1}{2t}-\frac{1}{4-2t}}dt&=e^{ik}t_1\sqrt{2t_1}e^{-ikt_1-\frac{1}{2t_1}-\frac{1}{4}}\left(\sqrt{\pi}+O(r^{-\frac{1}{4}+3\epsilon})\right)\\
&=\sqrt{\frac{-i\pi}{\sqrt{2i}k^{\frac{3}{2}}}}e^{ik-\sqrt{2ik}-\frac{1}{4 }}\Big(1+O(|k|^{-\frac{1}{4}+3\epsilon})\Big).
\end{align*}

For the integral over $C_2(\theta)$, we first substitute $t$ by $2-t$ to rewrite
\begin{align*}
 e^{ik}\int_{C_2(\theta)} e^{-ikt-\frac{1}{2t}-\frac{1}{4-2t}}dt=e^{-ik}\int_{C_2'(\theta)} e^{ikt-\frac{1}{2t}-\frac{1}{4-2t}}dt,
\end{align*}
where $C_2'(\theta)\subset \{x+yi\in\mathbb{C} \ | \ y=\cot(\frac{\theta}{2}+\frac{\pi}{4})x\} $ joining $0$ and $1+e^{-i\theta}$.
Note that $\cot(\frac{\theta}{2}+\frac{\pi}{4})=-\tan (\frac{\pi+\theta}{2}+\frac{\pi}{4})$ and  and $1+ie^{i\theta}=1-ie^{i(\pi+\theta)}$.
Therefore we see that 
\begin{align*}
 e^{ik}\int_{C_2(\theta)} e^{-ikt-\frac{1}{2t}-\frac{1}{4-2t}}dt= e^{i(-k)}\int_{C_1(\pi +\theta)} e^{-i(-k)t-\frac{1}{2t}-\frac{1}{4-2t}}dt.
\end{align*}
Similarly as in the integral over $C_1(\theta)$ earlier, we obtain
\begin{align*}
e^{ik}\int_{C_2(\theta)}e^{-ikt-\frac{1}{2t}-\frac{1}{4-2t}}dt& =e^{i(-k)}\int_{C_1(\pi+\theta)}e^{-i(-k)t-\frac{1}{2t}-\frac{1}{4-2t}}dt\\
&=\sqrt{\frac{-i\pi}{\sqrt{2i}(-k)^{\frac{3}{2}}}}e^{i(-k)-\sqrt{2i(-k)}-\frac{1}{4 }}\Big(1+O(|k|^{-\frac{1}{4}+3\epsilon})\Big)\\
&=\sqrt{\frac{i\pi}{\sqrt{-2i}k^{\frac{3}{2}}}}e^{-ik-\sqrt{-2ik}-\frac{1}{4 }}\Big(1+O(|k|^{-\frac{1}{4}+3\epsilon})\Big)
\end{align*}

Our computations so far yield the conclusion that
\begin{align*}
F(k)&=\sqrt{\frac{-i\pi}{\sqrt{2i}k^{\frac{3}{2}}}}e^{ik-\sqrt{2ik}-\frac{1}{4 }}\Big(1+O(|k|^{-\frac{1}{4}+3\epsilon})\Big)\\
& \ \ +\sqrt{\frac{i\pi}{\sqrt{-2i}k^{\frac{3}{2}}}}e^{-ik-\sqrt{-2ik}-\frac{1}{4 }}\Big(1+O(|k|^{-\frac{1}{4}+3\epsilon})\Big).
\end{align*}
In particular, for real $k$,
$$F(k)=2\re \left(\sqrt{\frac{-i\pi}{\sqrt{2i}k^{\frac{3}{2}}}}e^{ik-\sqrt{2ik}-\frac{1}{4 }}\Big(1+O(|k|^{-\frac{1}{4}+3\epsilon})\Big)\right).$$


\begin{thebibliography}{99}

\bibitem[AGH24]{AGH24} M. Alam, A. Ghosh, and J. Han. Higher moment formulae and limiting distributions of lattice points. J. Inst. Math. Jussieu. 2024; 23(5): 2081-2125.

\bibitem[Art05]{Art05} J. Arthur. An introduction to the trace formula. Clay Mathematical Proceedings, Vol. 4, Clay Mathematics Institute, 2005.

\bibitem[AM09]{AM09} J. Athreya and G. Margulis. Logarithm laws for unipotent flows I. J. Mod. Dyn., 3 (2009), no.3, 359-378.

\bibitem[AM15]{AM15} J. Athreya and G. Margulis. Values of random polynomials at integer points. J. Mod. Dyn., 12 (2018), 9-16.

\bibitem[CE03]{CE03} H. Cohn and N. Elkies. New upper bounds on sphere packings I. Ann. Math. (2)157 (2003), 689-714.

\bibitem[CKMRV]{CKMRV} H. Cohn, A. Kumar, S. D. Miller, D. Radchenko, and M. Viazovska. The sphere packing problem in
dimension 24. Ann. Math., 185:1017-1033, 2017.

\bibitem[ELMV]{ELMV} M. Einsiedler, E. Lindenstrauss, P. Michel, and A. Venkatesh. Distribution of periodic torus orbits and Duke's theorem for cubic fields. Ann. Math. 173(2011), 815-885.

\bibitem[EMM98]{EMM98} A. Eskin, G. Margulis, S. Mozes. Upper bounds and asymptotics in a quantitative version of the Oppenheim conjecture. Ann. Math. 147 (1998), 93-141.

\bibitem[FMT89]{FMT} J. Franke, Y. Manin, and Y. Tschinkel. Rational points of bounded height on Fano varieties. Invent. Math. 95 (1989), no. 2, 421-435.

\bibitem[GV24]{GV24} N. Gargava and M. Viazovska. Mean value for random ideal lattices. arXiv:2411.14973.

\bibitem[GMP17]{GMP17} H. Garland, S.D. Miller, and M.M. Patnaik. Entirety of cuspidal Eisenstein series on loop groups. Amer. J. Math. 139 (2017) no.2, 461-512

\bibitem[Gar18]{Gar18} P. Garrett. Modern analysis of automorphic forms by example, Vol 1. Cambridge studies in advanced mathematics 173, Cambridge University Press, 2018.

\bibitem[GKY22]{GKY22} A. Ghosh, D. Kelmer, and S. Yu. Effective density for inhomogeneous quadratic forms I: generic forms and fixed shifts. Int. Math. Res. Not. 2022, 4682.

\bibitem[Har76]{Har76} Harish-Chandra. Harmonic analysis on real reductive groups I; The theory of the constant term. J. Funct. Anal. 19, 104-204.

\bibitem[IK04]{IK04} H. Iwaniec and E. Kowalski. Analytic number theory. American Mathematical Society Colloquium Publications vol. 53, American Mathematical Society, 2004.

\bibitem[Joh15]{Joh15} S. Johnson. Saddle-point integration of $C^\infty$ ``bump'' functions. arXiv:1508.04376.

\bibitem[KW07]{KW07} H. Kim and L. Weng. Volume of truncated fundamental domains. Proc. Amer. Math. Soc. 135 (2007), no. 6, 1681-1688.

\bibitem[Kim22]{Kim22} S. Kim. Mean value formulas for sublattices of a random lattice. J. Number Theory 241 (2022),
330-351.

\bibitem[Kim23]{Kim23} S. Kim. Counting rational points of a Grassmannian. Monatsh. Math. (2023) 201:825-864.

\bibitem[Kim24]{Kim24} S. Kim. Adelic Rogers integral formula. J. Lond. Math. Soc. 109(1) 2024, e12830.

\bibitem[KV18]{KV18} S. Kim and A. Venkatesh. The behavior of random reduced bases. Int. Math. Res. Not. 2018(20), 6442-6480.

\bibitem[Kla25]{Kla25} B. Klartag. Lattice packing of spheres in high dimensions using a stochastically evolving ellipsoid. arXiv:2504.05042.

\bibitem[KM99]{KM99} D. Kleinbock and G. Margulis. Logarithm laws for flows on homogeneous spaces. Inv. Math. 138 (1999), 451-494.

\bibitem[LO12]{LO12} E. Lapid and K. Ouellette. Truncation of Eisenstein series. Pacific J. Math. 260 (2012), no. 2, 665-686.

\bibitem[LLL82]{LLL82} A. Lenstra, H. Lenstra, and L. Lov\'asz. Factoring polynomials with rational coefficients. Math. Ann. 261 (1982), no. 4, 515-534.

\bibitem[MS10]{MS10} J. Markl\"of and A. Str\"ombergsson. The distribution of free path lengths in the periodic Lorentz gas and related lattice point problems. Ann. Math. 172 (2010), 1949-2033.

\bibitem[Mil20]{Mil20} S.D. Miller. The geometry of numbers: an automorphic perspective; Spring 2020 course notes. Available at \url{https://sites.math.rutgers.edu/~sdmiller/Miller_automorphic_lattice_course.pdf}

\bibitem[MV06]{MV06} H.L. Montgomery and R.C. Vaughan. Multiplicative Number Theory, 1. Classical Theory. Cambridge studies in advanced mathematics 97, Cambridge University Press, 2006.

\bibitem[MW95]{MW95} C. Moeglin and J.-L. Waldspurger. Spectral decomposition and Eisenstein series. Cambridge University Press, 1995.

\bibitem[PY23]{PY23} I. Petrow and M.P. Young. The fourth moment of Dirichlet L-functions along a coset and the Weyl bound. Duke
Math. J., 172(10):1879-1960, 2023.

\bibitem[Rog47]{Rog47} C. A. Rogers. Existence theorems in the geometry of numbers. Ann. Math. (1947), 994-1002.

\bibitem[Rog55]{Rog55} C.A. Rogers. Mean values over the space of lattices. Acta Math. 94 (1955), 249-287.

\bibitem[Rog59]{Rog59} C.A. Rogers. Lattice covering of space. Mathematika 6 (1959), 33-39.

\bibitem[Sch57]{Sch57} W. Schmidt. Mittelwerte \"uber Gitter I, II. Monatsh. Math. 61 (1957), 269-276, and 62 (1958), 250-258.

\bibitem[Sch60]{Sch60} W. Schmidt. A metrical theorem in geometry of numbers. Trans. Amer. Math. Soc. 95 (1960), 516-529.

\bibitem[Sch68]{Sch68} W. Schmidt. Asymptotic formulae for point lattices of bounded determinant and subspaces of
bounded height. Duke Math. J. 35 (1968), 327-339.

\bibitem[Sie45]{Sie45} C.L. Siegel. A mean value theorem in geometry of numbers. Ann. Math. 46(2) (1945), 340-347.

\bibitem[Thu93]{Thu93} J. L. Thunder. Asymptotic estimates for rational points of bounded height on flag varieties. Compositio Math. 88 (1993), no. 2, 155-186.

\bibitem[Thu24]{Thu24} F. Thurman. Siegel integration for flags. Ph.D. thesis, Rutgers University.

\bibitem[Vee98]{Vee98} W. Veech. Siegel measures. Ann. Math. 148 (1998), 895-944.

\bibitem[Ven13]{Ven13} A. Venkatesh. A note on sphere packings in high dimension. Int. Math. Res. Not., 2013(7):1628-1642.

\bibitem[Via17]{Via17} M. Viazovska. The sphere packing problem in dimension 8. Ann. of Math., 185:991-1015, 2017



\end{thebibliography}
\end{document}